\documentclass[11pt, a4paper, reqno]{amsart}

\usepackage[utf8]{inputenc}
\usepackage[T1]{fontenc}
\usepackage{amsmath, mathtools, amsthm}
\usepackage{amssymb,url,xspace}
\usepackage[mathscr]{eucal}
\usepackage{bm}
\usepackage[dvipsnames]{xcolor}
\usepackage{cases}
\usepackage[hidelinks]{hyperref}
\usepackage{cleveref}

\usepackage{verbatim}

\usepackage{subfigure}
\usepackage{enumitem}

\definecolor{dukeblue}{rgb}{0.0, 0.0, 0.61}
\definecolor{darkcandyapplered}{rgb}{0.64, 0.0, 0.0}

\hypersetup{
	colorlinks,
	citecolor=darkcandyapplered,
	filecolor=black,
	linkcolor=dukeblue,
	urlcolor=black
}
\usepackage{empheq}
\usepackage{xfrac}

\usepackage{anysize}
\marginsize{3.25cm}{3.25cm}{3.25cm}{3.25cm}

\newcommand{\one}{\mathbf{1}}
\newcommand{\diff}{\, \mathrm{d}}
\newcommand{\del}{\partial}
\newcommand{\N}{\mathbb{N}}

\newcommand{\R}{\mathbb{R}}

\def\ad{{\mathrm{ad}}}

\def\aux{\mathrm{aux}}
\def\Lip{\mathrm{Lip}}

\def\*#1{\mathbf{#1}}

\theoremstyle{plain}
\numberwithin{equation}{section}

\newtheorem{theorem}{Theorem}[section]
\newtheorem{lemma}{Lemma}[section]
\newtheorem{claim}{Claim}[section]
\newtheorem{cor}{Corollary}[section]

\newtheorem{assumption}{Assumption}
\newtheorem{remark}{Remark}

\theoremstyle{definition}

	\title 
	[Turnpike in Lipschitz--nonlinear optimal control]{Turnpike in Lipschitz--nonlinear optimal control}
				
	\author{Carlos Esteve-Yag\"ue}

	\author{Borjan Geshkovski}
	
	\author{Dario Pighin}
	
	\address {\textbf{\textup{Carlos Esteve-Yag\"ue, Borjan Geshkovski, Dario Pighin}}
	\newline \indent
	\textup{Chair of Computational Mathematics} \hspace{1.28cm}
	\newline \indent
	\textup{Fundaci\'on Deusto}
	\newline \indent
	\textup{Av. de las Universidades, 24}
	\newline \indent
	\textup{48007 Bilbao, Basque Country, Spain} 
	}
	\email{\href{mailto:borjan.geshkovski@uam.es}{\textcolor{dukeblue}{\texttt{\{carlos.esteve, borjan.geshkovski, dario.pighin\}@deusto.es}} }}
		
	\author{Enrique Zuazua}
	\address{\textbf{\textup{Enrique Zuazua}}
		\newline \indent
		\textup{Chair in Dynamics, Control, and Numerics}
		\newline \indent
		\textup{Alexander von Humboldt-Professorship}
		\newline \indent
		\textup{Friedrich-Alexander-Universit\"at Erlangen-N\"urnberg}
		\newline \indent
		\textup{91058 Erlangen, Germany}
		\newline \indent \hspace{2.5cm} \textit{and} \newline \indent
	\textup{Chair of Computational Mathematics} 
	\newline \indent
	\textup{Fundaci\'on Deusto}
	\newline \indent
	\textup{Av. de las Universidades, 24}
	\newline \indent
	\textup{48007 Bilbao, Basque Country, Spain}
	\newline \indent \hspace{2.5cm} \textit{and} \newline \indent
	\textup{Departamento de Matemáticas} \newline \indent
		\textup{Universidad Autónoma de Madrid}
	\newline \indent
	\textup{28049 Madrid, Spain}
	}
	\email{\href{mailto:enrique.zuazua@fau.de}{\textcolor{dukeblue}{\texttt{enrique.zuazua@fau.de}}}}

\date{\today}

\begin{document}
	
		\begin{abstract}
			We present a new proof of the turnpike property for nonlinear optimal control problems, when the running target is a steady control-state pair of the underlying system.
			Our strategy combines the construction of quasi-turnpike controls via controllability, and a bootstrap argument, and does not rely on analyzing the optimality system or linearization techniques.
			This in turn allows us to address several optimal control problems for finite-dimensional, control-affine systems with globally Lipschitz (possibly nonsmooth) nonlinearities, without any smallness conditions on the initial data or the running target. 
			These results are motivated by applications in machine learning through deep residual neural networks, which may be fit within our setting.
			We show that our methodology is applicable to controlled PDEs as well, such as the semilinear wave and heat equation with a globally Lipschitz  nonlinearity, once again without any smallness assumptions. 
		\end{abstract}
		
	\maketitle	
	\setcounter{tocdepth}{1}
	
	\tableofcontents
	

{\small{\href{https://mathscinet.ams.org/msc/msc2010.html}{{\bf \color{dukeblue}{AMS Subject Classification}}}}. 34H05; 34H15; 93C15; 93C20.}
		
\section{Introduction}

	\subsection{Motivation}
	The \emph{turnpike property} reflects the fact that, for suitable optimal control problems set in a sufficiently large time horizon, any optimal solution thereof remains, during most of the time, close to the optimal solution of a corresponding “static” optimal control problem. 
	This optimal static solution is referred to as \emph{the turnpike} -- the name stems from the idea that a turnpike is the fastest route between two points which are far apart, even if it is not the most direct route.
	In many cases, the turnpike property is described by an exponential estimate -- for instance, the optimal trajectory $y_T(t)$ is $\mathcal{O}\left(e^{-\mu t}+e^{-\mu(T-t)}\right)$--close to the optimal static solution $\overline{y}$, for $t \in [0, T]$ and for some $\mu>0$. 
	\smallskip
	
	\subsubsection{Background} \label{sec: background}
	The prevalent (but not exclusive) argument for proving exponential turnpike results relies on a thorough analysis of the optimality system provided by the Pontryagin Maximum Principle. 
	In the context of linear quadratic optimal control problems, under appropriate controllability or stabilizability conditions, turnpike is established via properties of the optimality system characterizing the optimal controls and states through the coupling with the adjoint system (\cite{porretta2013long, grune2020exponential}).
	In the case of nonlinear dynamics, this argument thus requires nonlinearities which are continuously differentiable. A linearization argument is used -- the linear study and a fixed point argument provide nonlinear results under smallness assumptions on the initial data and the target (\cite{PZ2, trelat2015turnpike}). The smallness conditions on the initial data can be removed in some specific cases (\cite{pighin2020semilinear}), but to the best of our knowledge, the assumptions on the running target have not been as of yet (albeit, they may be removed under restrictive assumptions, such as strict dissipativity, uniqueness of minimizers and $C^2$--regular nonlinearities -- see \cite{trelat2020linear}). This is due to the lack of tools for showing that the linearized optimality system corresponds to a linear-quadratic control problem satisfying the turnpike property, when the running target of the original nonlinear control problem is large.
	
	\subsubsection{A question raised by machine learning}
	There has been an ever-increasing need, brought by applications in machine learning via residual neural networks (ResNets, \cite{weinan2017proposal, esteve2020large, he2016deep}), of turnpike results for nonlinear optimal control problems without smallness conditions on the initial data or the running target, and for systems with globally Lipschitz-continuous but possibly nonsmooth nonlinearities. 
	
	In (supervised) machine learning, one looks for a map which interpolates a dataset $$\left\{x^{(i)}, y^{(i)}\right\}_{i\in\{1,\ldots,n\}}\subset \mathbb{R}^{d_x}\times\mathbb{R}^{d_y},$$ and which gives accurate predictions on unknown points $x\in\R^{d_x}$ (\cite{lecun2015deep}). 
	Such a task may (in many cases) be accomplished by solving
	\begin{equation} \label{eq: J_T.deep}
	\inf_{\substack{u=(w,b)\in L^2(0,T;\mathbb{R}^{d_u})\\ \*x_i \text{ solves} \eqref{eq: neural.net}}}\sum_{i=1}^n\int_0^T\left\|P\*x_i(t)-y^{(i)}\right\|^2\diff t + \int_0^T \left\|u(t)\right\|^2\diff t,
	\end{equation}
	where $P:\R^{d_x}\to\R^{d_y}$ is a given surjective map (possibly nonlinear, see \Cref{sec: numerics}), and the constraint is given by the continuous-time residual neural network\footnote{Also referred to as a neural ODE \cite{chen2018neural}.}
	\begin{equation} \label{eq: neural.net}
	\begin{dcases}
	\dot{\*x}_i(t) = \sigma(w(t) \*x_i(t) + b(t)) &\text{ in } (0, T)\\
	\*x_i(0) = x^{(i)},
	\end{dcases}
	\end{equation}
	with $w(t)\in\R^{d_x\times d_x}$ and $b(t)\in\R^{d_x}$ designate the controls (thus $d_u=d_x^2+d_x$), whereas $\sigma \in \Lip(\R)$ with $\sigma(0)=0$ is a scalar nonlinear function, defined componentwise in \eqref{eq: neural.net}. 
	The most frequently used nonlinearities in practice are \emph{rectifiers}: $\sigma(x) = \max\{\alpha x, x\}$ for $\alpha \in [0,1)$, and \emph{sigmoids}: $\sigma(x) = \tanh(x)$.
	The order of the nonlinearity $\sigma$ and the affine map within may be permuted to obtain a driftless control-affine system
	\begin{equation} \label{eq: neural.net.2}
	\begin{dcases}
	\dot{\*x}_i(t) = w(t)\sigma(\*x_i(t))+ b(t) &\text{ in } (0, T)\\
	\*x_i(0) = x^{(i)}.
	\end{dcases}
	\end{equation}
	
	\begin{figure}[h] 
	\includegraphics[scale=0.475]{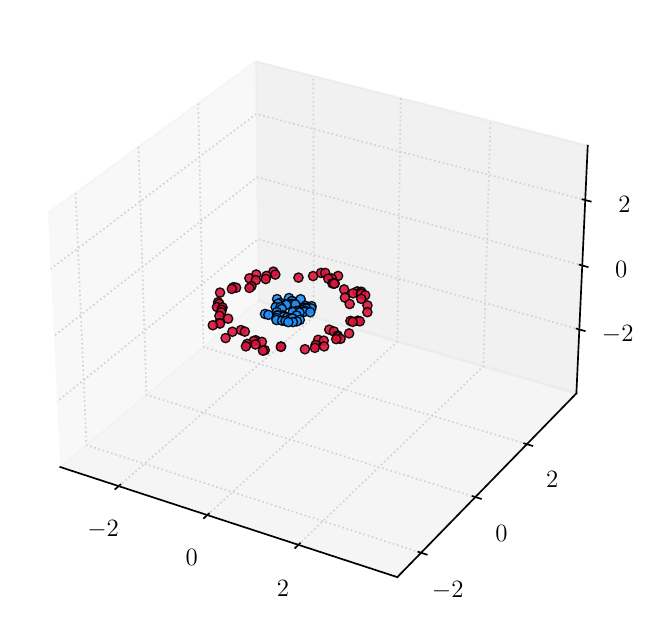}
	\includegraphics[scale=0.475]{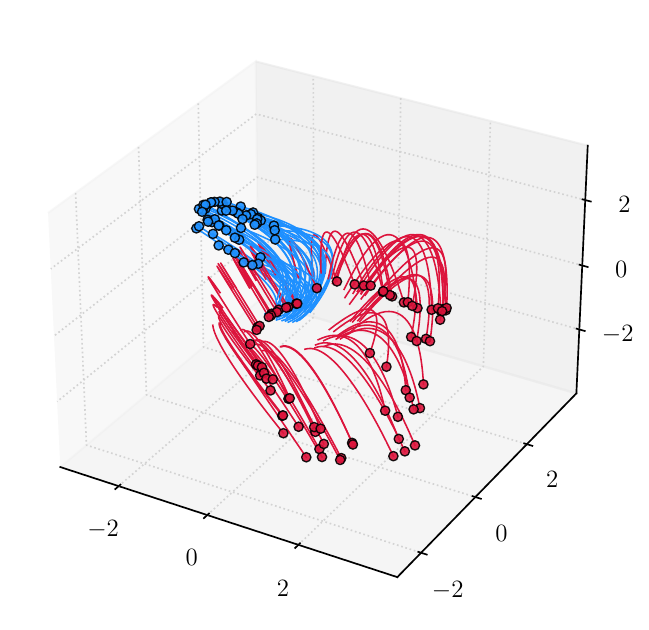}
	\includegraphics[scale=0.475]{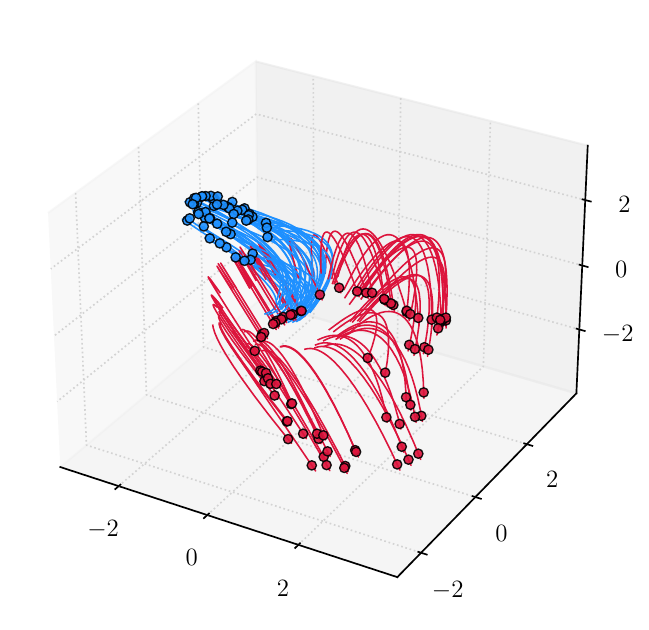}
	\includegraphics[scale=0.475]{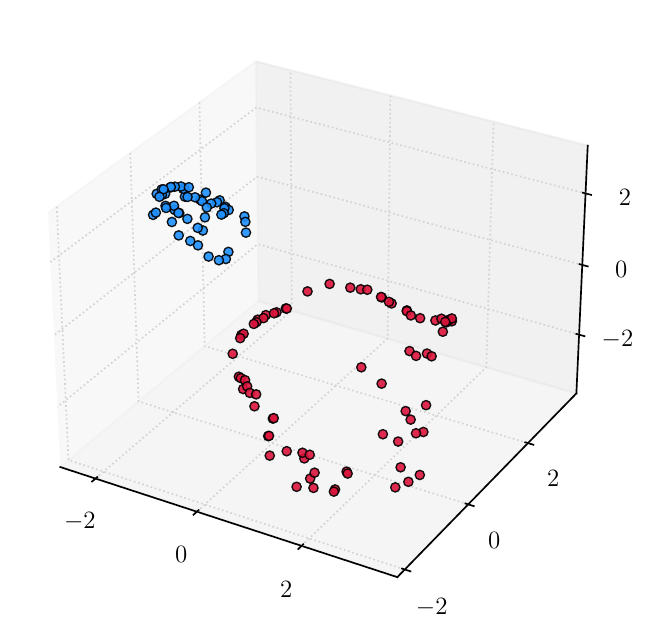}
	\caption{\emph{Binary classification task}. 
	One aims to separate the data $\left\{x^{(i)}\right\}_{i\in\{1,\ldots,n\}}\subset\R^3$ (top left) with respect to their color ($y^{(i)}=\pm1$ for red, blue) by using the flow of \eqref{eq: neural.net.2}, found by minimizing \eqref{eq: J_T.deep}.
	We plot the trajectories $\*x(t):=\{\*x_i(t)\}_{i\in\{1,\ldots,n\}}$ of \eqref{eq: neural.net.2} for $t\leqslant2$ (top right), $t\leqslant5$ (mid left), and in time $T=5$ (mid right). 
	We see stabilization for the projections, for the controls to $0$, and hence also for the trajectories to some points $\overline{\*x}_i \in P^{-1}(\{y^{(i)}\})$, which are steady states (bottom).
	}
	\vspace{0.25cm}
	\label{fig: figure.dl}
	\includegraphics[scale=0.465]{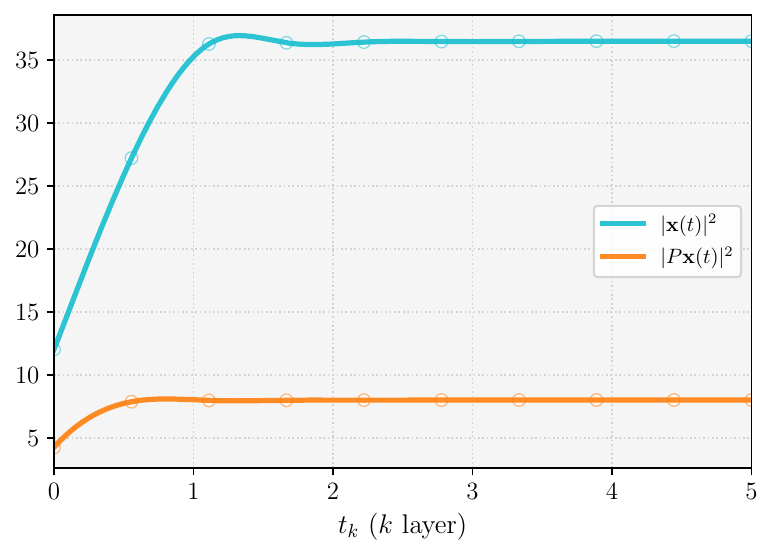}
	\includegraphics[scale=0.465]{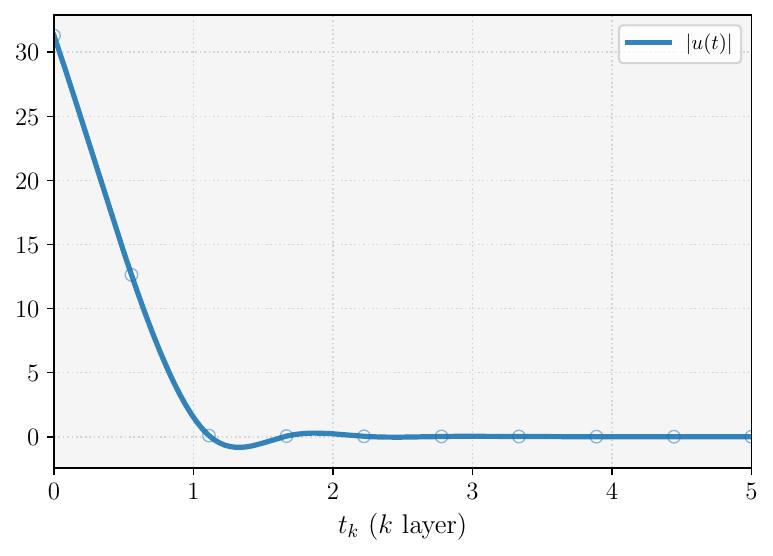}
	\end{figure}
	
	\noindent
	Combinations and variants of \eqref{eq: neural.net} and \eqref{eq: neural.net.2} may also be used (\cite{lin2018resnet}).
	Optimizing $u$ over $n\gg1$ different initial data establishes robustness, so that the neural networks \eqref{eq: neural.net} and \eqref{eq: neural.net.2}  may correctly perform future predictions on unknown points (\Cref{fig: train.test}).
	
	One notes a feature specific to the dynamics $f(\*x_i, u)$ in \eqref{eq: neural.net} and \eqref{eq: neural.net.2} for fixed $i$: any constant vector in $\mathbb{R}^{d_x}$ is a steady state with control $u=(w,b)\equiv0$. 
	Whence, the solutions to the optimal steady problem
	\begin{equation*}
	\inf_{\substack{(u,\*x_i)\in\mathbb{R}^{d_u}\times\mathbb{R}^{d_x}\\ f(\*x_i, u)=0}} \sum_{i=1}^n\left\|P\*x_i-y^{(i)}\right\|^2+\|u\|^2
	\end{equation*}
	with $f(\*x_i,u)$ as in \eqref{eq: neural.net} or \eqref{eq: neural.net.2}, are precisely given by $(0,\overline{\*x}_i)$, where $\overline{\*x}_i\in P^{-1}\left(\left\{y^{(i)}\right\}\right)$ for $i\in\{1,\ldots,n\}$ (the preimage of $P$ might not be a singleton if $d_x\neq d_y$).
	 In \Cref{fig: figure.dl}, we see that not only the projections $P\*x_i(t)$, but also the optimal trajectories $\*x_i(t)$ stabilize to such points: $\overline{\*x}_i \in P^{-1}\left(\left\{y^{(i)}\right\}\right)$, for $i\in\{1,\ldots,n\}$, which are, as said above, steady states of the underlying system without control (i.e. with $0$ control).
	Existing turnpike results do not immediately apply to explain this artifact, as discussed in \Cref{sec: background}, due to the use of nonsmooth nonlinearities and the lack of smallness assumptions on the targets, which would be unrealistic. 
	This motivates the setting of our study (even for more general dynamics), namely, the consideration running targets which are steady states of the underlying dynamics without control (see \eqref{eq: J_T.first} below). 
	
	\subsubsection{Interpretation}
	The practical interest of this stabilization property is regarding the approximation capacity of ResNets, which are the forward Euler discretization of \eqref{eq: neural.net} and \eqref{eq: neural.net.2} with fixed time-step $\bigtriangleup t = \sfrac{T}{n_t}$. Here $n_t$ is the \emph{number of layers}, and when $n_t\gg1$, one is said to be doing \emph{deep learning}. 
	As $\bigtriangleup t$ is fixed, exponential decay would provide a quantitative estimate of the number of layers needed to fit the data, whilst keeping the controls small (thus possibly ensuring generalization). 
	Such an estimate would actually indicate that the time horizon (or number of layers) ought not be big at all so that the approximation error (the first term in \eqref{eq: J_T.deep}) reaches $0$ with controls of small amplitude. In other words, any layers beyond a certain stopping time $T^*$ can be dropped\footnote{It should to be said that a sharp and applicable conclusion would depend on the "complexity" of the dataset, as well as the number $n$ of datapoints, which we do not specifically take into account. But an exponential decay estimate would give a rough idea on how to design methods for numerically estimating the stopping time $T^*$. We refer the reader to \cite{esteve2020large, faulwasser2021turnpike, geshkovski2021control, grune2019sensitivity, grune2020efficient} for further detail. Similar conclusions have been drawn in the context of $L^1(0,T;\mathbb{R}^{d_u})$ control penalties in \cite{yague2021sparse}.} from the optimization scheme. 
	 In our numerical experiment (see \Cref{sec: numerics} for detail and further illustrations), we use $T=5$ and $\bigtriangleup t = \sfrac12$, and stability occurs beyond $T^*\sim 2$.
	
	\subsection{Our contributions.} To answer this need, and motivated by problems as those above, in this work we provide a different perspective on the turnpike property in the context of nonlinear dynamics, and we bring forth the following contributions.
	
	\begin{itemize}
	\item In \Cref{sec: ode}, we consider optimal control problems consisting of minimizing\footnote{While not precisely the same as considering \eqref{eq: J_T.deep} for $d_x\neq d_y$ and a non-invertible map $P$, we believe that this setting is a first step towards a complete understanding of \eqref{eq: J_T.deep}. See \Cref{sec: conclusion} for a discussion.}
	\begin{equation} \label{eq: J_T.first}
	J_T(u) := \phi(y(T))+ \int_0^T \|y(t)-\overline{y}\|^2 \diff t + \int_0^T \|u(t)-\overline{u}\|^2 \diff t
	\end{equation}
	subject to $\dot{y} = f(y, u)$, where $f$ is of control-affine form. 
	Under the assumption that the running target $(\overline{u}, \overline{y})$ is a steady control-state pair, namely $f(\overline{y}, \overline{u}) = 0$, and that the system is controllable with an estimate on the cost (see \Cref{def: ctrl}), in  \Cref{thm: turnpike} we prove the exponential turnpike property described above. 
	The main novelty lies in the fact that the nonlinearity $f$ is only assumed to be globally Lipschitz continuous, and the result comes without any smallness conditions on the initial data or the specific running target. 
	In this case, existing results such as those presented in \cite{trelat2015turnpike} do not apply, as they require smallness assumptions and $C^2$--nonlinearities. 
	
	Moreover, whenever the functional to be minimized does not contain a final-time cost (such as $\phi(y(T))$ in $J_T$ above), we can prove (see \Cref{cor: stabilisation} below) that the exponential arc near the final time $t=T$ disappears, thus entailing an exponential stabilization property for the optimal state to the running target. 
	\smallskip
	
	\item In \Cref{sec: pde}, the finite-dimensional results are extended to analogue optimal control problems for underlying PDE dynamics. 
	This is illustrated in \Cref{thm: turnpike.wave}, \Cref{cor: stabilisation.wave} and \Cref{thm: turnpike.heat} in the context of the semilinear wave and heat equation with globally Lipschitz--only nonlinearity, once again under the assumption that the running target is a steady control-state pair. 
	We make no smallness assumptions neither on it, nor on the initial data, thus covering some cases where results from \cite{grune2021abstract, pighin2020semilinear, PZ2, zuazua2017large} are not applicable.
	\end{itemize}
	
	\subsection{Outline} The paper is organized as follows. \textbf{\Cref{sec: ode}} contains statements of our main results in the setting of finite-dimensional, control-affine systems (namely \Cref{thm: turnpike}, and \Cref{cor: stabilisation}, \Cref{cor: control.decay}). We also provide a sketch of our new, purely nonlinear strategy in \Cref{sec: sketch}.
	 \textbf{\Cref{sec: pde}} states the extensions of the finite-dimensional results to the setting of the semilinear wave and heat equation (\Cref{thm: turnpike.wave} and \Cref{thm: turnpike.heat} respectively). \textbf{\Cref{sec: preliminary.results}} provides some preliminary, but key lemmas, mainly to ensure appropriate $L^\infty_t$ bounds for the discrepancy between a trajectory $y(t)$ and the target steady state $\overline{y}$.  \textbf{\Cref{sec: proof.fin.dim}} provides our proofs of the main results in the finite-dimensional case, namely \Cref{thm: turnpike} and \Cref{cor: control.decay}.  \textbf{\Cref{sec: turnpike.wave.proof}} provides our proof of \Cref{thm: turnpike.wave}, namely the main result for the semilinear wave equation, which is a straightforward adaptation of the arguments in the finite-dimensional case. \textbf{\Cref{sec: turnpike.heat.proof}} presents our proof of \Cref{thm: turnpike.heat}, namely the main result for the semilinear heat equation. The same proof applies to \Cref{cor: stabilisation}. \textbf{\Cref{sec: conclusion}} concludes the paper with a selection of open problems.
	
	\subsection{Notation.} 
	We denote by $\| \cdot \|$ the standard euclidean norm, and $\N:=\{1,2,\ldots\}$.
	We denote by $\Lip(\R)$ the set of functions $f: \R \to \R$ which are globally Lipschitz continuous.

	\section{Finite-dimensional systems} \label{sec: ode}
	
	\subsection{Setup}
	Let $d\geqslant 1$ and $m\geqslant 1$. 
	We will consider differential control systems where the state $y(t)$ lives in $\R^d$ and the control input $u(t)$ in $\R^m$. 
	Given $T>0$, we focus on \emph{control-affine} systems, namely canonical nonlinear systems
	\begin{equation} \label{eq: nonlinear.control.system}
	\dot{y} = f(y,u) \hspace{1cm} \text{ in } (0,T)
	\end{equation}
	with a nonlinearity $f$ of the form
	\begin{equation} \label{eq: nonlinearity.control.affine}
	f(y,u)= f_0(y) + \sum_{j=1}^m u_j f_j(y) \hspace{1cm} \text{ for } (y,u) \in \R^d \times \R^m,
	\end{equation}
	where the vector fields $f_0,\ldots, f_m \in \Lip(\R^d; \R^d)$ are only assumed to be globally Lipschitz continuous. This formulation includes \eqref{eq: neural.net.2} -- see \Cref{rem: on.the.nonlinearity} for possible extensions to \eqref{eq: neural.net}. 
	For any given initial datum $y^0 \in \R^d$ and control input $u \in L^1(0,T;\R^m)$, system \eqref{eq: nonlinear.control.system}, with $f$ as in \eqref{eq: nonlinearity.control.affine}, admits a unique solution $y \in C^0([0,T];\R^d)$ with $y(0) = y^0$. This can be shown by means of a fixed point theorem and the Gr\"onwall inequality applied to the integral formulation
	\begin{equation*}
	y(t) = y^0 + \int_0^t f(y(s), u(s)) \diff s.
	\end{equation*}
	Given $y^0 \in \R^d$, we will investigate the behavior when $T\gg1$ of global minimizers $u_T \in L^2(0,T;\R^m)$ to nonnegative functionals of the form 
	\begin{equation} \label{eq: J_T}
	J_T(u):= \phi(y(T))+ \int_0^T \|y(t)-\overline{y}\|^2 \diff t + \int_0^T \|u(t)\|^2 \diff t,
	\end{equation}
	and of the corresponding solutions $y_T$ to \eqref{eq: nonlinear.control.system} with $y_T(0)=y^0$. 
	Here, $\phi \in C^0(\R^d; \R_+)$ is a given final cost, while $\overline{y} \in \R^d$ is a given running target which we select as an \emph{uncontrolled steady state} of the nonlinear dynamics, namely	
	\begin{equation} \label{eq: f0}
	f_0(\overline{y}) = 0.
	\end{equation} 
	We provide further comments on the specific choice of the running target just below, in \Cref{rem: running.target}. 
	Due to the coercivity of $J_T$ and the explicit form of $f$ in \eqref{eq: nonlinearity.control.affine}, the existence of a minimizer of $J_T$ follows from the direct method in the calculus of variations.
	
	Due to the presence of the state tracking term in the definition of $J_T$, which regulates the state over the entire time interval $[0,T]$, the turnpike property is expected to hold: over long time horizons, the optimal control-state pair $(u_T, y_T)$ should be "near" the optimal steady control-state pair $(u_s, y_s)$, namely a solution to the problem
	\begin{equation} \label{eq: steady.state.ocp}
	\inf_{\substack{(y,u)\in\mathbb{R}^d\times\R^m\\ \text{ subject to } f(y,u)=0}} \|y-\overline{y}\|^2 + \|u\|^2.
	\end{equation}
	Now note that, due to the assumption \eqref{eq: f0} on the running target $\overline{y}$, and the form of the nonlinearity $f$ in \eqref{eq: nonlinearity.control.affine}, 
	it can be seen that $(u_s, y_s) \equiv (0, \overline{y})$ designates the unique optimal stationary solution, namely the unique solution to \eqref{eq: steady.state.ocp}.	

	\begin{remark}[Controlled steady states] \label{rem: running.target}

	The choice of the running target $\overline{y}$ in \eqref{eq: f0} is tailored to our proof strategy and the choice of the functional $J_T$ in \eqref{eq: J_T}. 
	The key feature our methodology requires is that the Lagrangian $\mathscr{L}(u, y)=\|y-\overline{y}\|^2 + \|u-\overline{u}\|^2$ equals zero when evaluated at the optimal steady state. 
	In fact, we could more generally consider the functional
	\begin{equation*}
	J_T(u) := \phi(y(T)) +  \int_0^T \|y(t)-\overline{y}\|^2 \diff t + \int_0^T \|u(t)-\overline{u}\|^2\diff t
	\end{equation*}
	where $(\overline{u}, \overline{y}) \in \R^m\times\R^d$ is chosen so that $f(\overline{y}, \overline{u}) = 0$ (with $f$ as in \eqref{eq: nonlinearity.control.affine}), as discussed in the introduction. The results presented below could then readily be adapted to this case (by additionally changing \eqref{eq: turnpike.1} and \Cref{def: ctrl} to an $L^2$--bound of $u_T-\overline{u}$).
	We have taken $\overline{u}=0$ for presentational simplicity.

	\end{remark}
	
	\noindent
	In the context of nonlinear optimal control, such turnpike results have been shown in \cite{trelat2015turnpike} (see also \cite{trelat2020linear}) for $C^2$--regular nonlinearities $f$. 
	This order of regularity is required due to the proof strategy, which relies on linearizing the optimality system given by the Pontryagin Maximum Principle. 
	As a consequence, the results in \cite{trelat2015turnpike} are also local, in the sense that smallness conditions are assumed on the initial data and target in view of applying a fixed point argument.
	In this work, we take a further step and obtain global results for globally Lipschitz nonlinearities.
	
	\subsection{Main results}
	Controllability plays a key role in the context of turnpike.
	Hence, before proceeding, we state the following assumption.
	
	\begin{assumption}[Controllability \& cost estimate] \label{def: ctrl}
	We will assume that \eqref{eq: nonlinear.control.system} is \emph{controllable} in some time $T_0>0$, meaning that there exists some time $T_0>0$ such that for any $y^0, y^1\in \R^d$, there exists a control $u \in L^2(0,T_0; \R^m)$ such that the corresponding solution $y \in C^0([0,T_0];\R^d)$ to \eqref{eq: nonlinear.control.system} with $y(0) = y^0$ satisfies $y(T_0) = y^1$.
	
	We will moreover assume that there exists a radius $r>0$ and a constant $C(T_0)>0$ such that 
	\begin{equation} \label{eq: control.cost.estimate.1}
	\inf_{\substack{u \\ \text{ such that } \\ y(0) = y^0, \, y(T_0) = \overline{y}} } \|u\|_{L^2(0,T_0; \R^m)}\leqslant C(T_0)\left\|y^0-\overline{y}\right\|,
	\end{equation}
	and
	\begin{equation} \label{eq: control.cost.estimate.2}
	\inf_{\substack{u \\ \text{ such that } \\ y(0) = \overline{y}, \, y(T_0) = y^1} } \|u\|_{L^2(0,T_0; \R^m)}\leqslant C(T_0)\left\|y^1-\overline{y}\right\|,
	\end{equation}
	hold for any $y^0, y^1 \in \left\{x\in \R^d \colon \|x-\overline{y}\|\leqslant r\right\}$, where $\overline{y} \in \R^d$ is fixed as in \eqref{eq: f0}.
	\end{assumption}
	
	\noindent
	We discuss the feasibility of this assumption later on, in \Cref{rem: later.on}.
	Note that this is not a smallness assumption -- it merely stipulates that, inside some ball centered at $\overline{y}$, the cost of controlling from $y^0$ to $\overline{y}$ and from $\overline{y}$ to $y^1$ can be estimated by means of the distance of $y^0$ and $y^1$ to $\overline{y}$.
	We may now state our first main result.
		
	\begin{theorem}[Turnpike] \label{thm: turnpike}
	Assume that $f_0, \ldots, f_m \in \Lip(\R^d; \R^d)$ in \eqref{eq: nonlinearity.control.affine}, and assume that \eqref{eq: nonlinear.control.system} is controllable in some time $T_0>0$ in the sense of \Cref{def: ctrl}.
	Let $y^0\in\R^d$ be given, and let $\overline{y}\in\R^d$ be as in \eqref{eq: f0}.
	Then there exists a time $T^*>0$, and constants $C>0$ and $\mu>0$, all depending on $T_0,y^0,\overline{y}$, such that for any $T>T^*$, any global minimizer $u_T\in L^2(0,T;\R^m)$ to $J_T$ defined in \eqref{eq: J_T} and corresponding optimal state $y_T$ solution to \eqref{eq: nonlinear.control.system} with $y_T(0)=y^0$ satisfy
	\begin{equation} \label{eq: turnpike.2}
	\left\|y_T(t)-\overline{y}\right\| \leqslant C\left( e^{-\mu t} + e^{-\mu(T-t)} \right)
	\end{equation}
	for all $t \in [0, T]$, and
	\begin{equation} \label{eq: turnpike.1}
	\| u_T\|_{L^2(0,T; \R^m)} \leqslant C.
	\end{equation}
	\end{theorem}	
	
	\noindent
	We sketch the idea of the proof (which may be found in \Cref{sec: proof.turnpike}) in \Cref{sec: sketch} below.
	The rate $\mu>0$ appearing in \eqref{eq: turnpike.2} depends on the datum $y^0$ due to the multiplicative form of the control, but is uniform with respect to $y^0$ when the control is \emph{additive}, namely, when $f_1, \ldots, f_m$ are nonzero constants. This is due to the form of the constant provided by Gr\"onwall inequality-based arguments (e.g. in \Cref{lem: C1} and \Cref{lem: unif time estimates}). We delay a discussion of the specific form of the constants to \Cref{sec: discussion}.
	
	\begin{remark}[On \eqref{eq: turnpike.1}]
	An exponential estimate for the optimal control $u_T$ is a hallmark of turnpike results obtained by analyzing the optimality system. Therein, the optimal control can be characterized explicitly via the adjoint state, which, much like the optimal state, fulfills an exponential estimate. 
	Since in this work we do not use the optimality system, we do not have as much information on $u_T(t)$ as we have on $y_T(t)-\overline{y}$. The latter quantity, in addition to being penalized by $J_T$, may be further estimated by using the system dynamics.
	In the context of \emph{driftless systems}, we show that $u_T(t)$ too decays exponentially in \Cref{cor: control.decay}, by using the homogeneity of the system with respect to the control.
	\end{remark}
	
	\noindent
	Before proceeding with further remarks, which we postpone to \Cref{sec: discussion}, let us state a couple of important corollaries of \Cref{thm: turnpike}.
	Firstly, when one considers an optimal control problem for $J_T$ without a final cost for the endpoint $y(T)$, namely taking $\phi\equiv0$ in \eqref{eq: J_T}, \Cref{thm: turnpike} can in fact be improved to an \emph{exponential stabilization} estimate to the running target $\overline{y}$. 
	
    \begin{cor}[Stabilization] \label{cor: stabilisation}
    Suppose that $\phi\equiv 0$ in \eqref{eq: J_T}.
    Under the assumptions of \Cref{thm: turnpike}, there exists a time $T^*>0$, and constants $C>0$ and $\mu>0$, all depending on $T_0,y^0,\overline{y}$, such that for any $T>T^*$, any global minimizer $u_T\in L^2(0,T;\R^m)$ to $J_T$ defined in \eqref{eq: J_T} and corresponding optimal state $y_T$ solution to \eqref{eq: nonlinear.control.system} with $y_T(0)=y^0$ satisfy \eqref{eq: turnpike.1} as well as
	\begin{equation}  \label{eq: stabilisation.statement}
	\left\|y_T(t)-\overline{y}\right\| \leqslant Ce^{-\mu t}
	\end{equation}
	for all $t\in[0,T]$.
    \end{cor}
    
    \noindent
    Strictly speaking, we see \Cref{cor: stabilisation} as a consequence of the strategy of proof of \Cref{thm: turnpike}, rather than a direct corollary of the statement. 
   \Cref{cor: stabilisation} may be proven independently of \Cref{thm: turnpike} by a simple adaptation of the proof strategy presented in \Cref{sec: sketch}, so we omit the proof. 
   This adaptation is transparent in the proof of \Cref{thm: turnpike.heat}, for which we provide  greater detail, albeit for more specific dynamics (the semilinear heat equation). 
    
    On another hand, when the underlying dynamics \eqref{eq: nonlinear.control.system} are of \emph{driftless control affine} form (namely, $f_0\equiv0$ in \eqref{eq: nonlinearity.control.affine}), 
    we can obtain an exponential decay for the optimal controls as well. Note that in this case, any $\overline{y} \in \R^d$ is an admissible running target for $J_T$, since $f(\overline{y},0)=0$ for any $\overline{y}\in\R^d$.
    
      \begin{cor}[Control decay] \label{cor: control.decay}
      Suppose that $f_0\equiv 0$ in \eqref{eq: nonlinearity.control.affine} and $\phi\equiv 0$ in \eqref{eq: J_T}.
    Under the assumptions of \Cref{thm: turnpike}, there exists a time $T^*>0$, and constants $C>0$ and $\mu>0$, all depending on $T_0,y^0,\overline{y}$, such that for any $T>T^*$, any global minimizer $u_T\in L^2(0,T;\R^m)$ to $J_T$ defined in \eqref{eq: J_T} and corresponding optimal state $y_T$ solution to \eqref{eq: nonlinear.control.system} with $y_T(0)=y^0$ satisfy \eqref{eq: stabilisation.statement} as well as
	\begin{equation} \label{eq: turnpike.4}
	\left\|u_T(t)\right\| \leqslant C e^{-\mu t}
	\end{equation}
	for a.e. $t \in [0,T]$.
    \end{cor}
    
    \noindent
   	\Cref{cor: stabilisation} and \Cref{cor: control.decay} are in particular applicable for the continuous time analog \eqref{eq: neural.net.2} of ResNets (see \Cref{rem: on.the.nonlinearity} for \eqref{eq: neural.net}). 
   The proof of \Cref{cor: control.decay} (see \Cref{sec: proof.control.decay}) will follow by firstly using a specific suboptimal control (constructed using the time-scaling specific to driftless systems) to estimate $J_T(u_T)$ and obtain 
    \begin{equation*}
    \frac12\int_{t}^{t+h} \|u_T(s)\|^2 \diff s \leqslant\int_{t}^{t+h} \|y_T(s)-\overline{y}\|^2 \diff s
    \end{equation*}
    for $h$ small enough, an estimate which, chained with \Cref{cor: stabilisation} and the Lebesgue differentiation theorem, will suffice to conclude.    
	
	\subsubsection{Sketch of the proof of \Cref{thm: turnpike}} \label{sec: sketch}
	
	The proof of \Cref{thm: turnpike} may be found in \Cref{sec: proof.turnpike}.
	It roughly follows the following scheme (see \Cref{fig: fig.1}).
	
		\begin{itemize}
\item \textbf{Uniform bound of $J_T(u_T)$.} 
In \Cref{lem: quasiturnpike.1} we show that there exists a constant $C_0>0$ independent of $T$ (but depending on $y^0,\overline{y},f,\phi$) such that
\begin{equation} \label{eq: JT.bounded}
J_T(u_T)\leqslant C_0
\end{equation}
holds for all $T>0$. As $\overline{y}$ is a steady state, \eqref{eq: JT.bounded} can be shown easily. Indeed, using controllability (without the estimates on the cost) one finds a control $u^\dagger$ such that the solution $y^\dagger$ to \eqref{eq: nonlinear.control.system} on $[0,T_0]$ satisfies $
y^\dagger(T_0)=\overline{y}.$
Setting $u^{\text{aux}}(t):=u^\dagger(t)1_{[0,T_0]}(t)$ for $t\in[0,T]$, one sees that $y^{\text{aux}}(t)=\overline{y}$ for $t\in[T_0,T]$,
whence $J_T(u^{\text{aux}})=J_{T_0}(u^\dagger)$, and using the inequality $J_T(u_T)\leqslant J_T(u^{\text{aux}})$ yields \eqref{eq: JT.bounded}.
\smallskip
	
\item \textbf{$L^\infty_t$ bound of $y(t)-\overline{y}$.}
In \Cref{lem: C1}, we show that
\begin{equation} \label{eq: estimate.no.number}
\sup_{t\in[0,T]}\|y(t)-\overline{y}\|\leqslant C\Big(\|y(0)-\overline{y}\|+\|y-\overline{y}\|_{L^2(0,T;\mathbb{R}^d)}+\|u\|_{L^2(0,T;\mathbb{R}^m)}\Big)
\end{equation}
holds for some constant $C>0$ depending on $T$ solely through $\|u\|_{L^2(0,T;\mathbb{R}^m)}$, in a continuous and increasing manner. The estimate holds for any, not necessarily optimal $u$. 
The globally Lipschitz assumption on the dynamics in \eqref{eq: nonlinear.control.system} is used precisely for this estimate.
Combined with \eqref{eq: JT.bounded}, estimate \eqref{eq: estimate.no.number}  yields
\begin{equation} \label{eq: unif.bound.explain}
\sup_{t\in[0,T]}\|y_T(t)-\overline{y}\|^2+J_T(u_T)\leqslant C_1^2
\end{equation}
for some constant $C_1>0$ independent of $T$ (but depending on $y^0,\overline{y},f,\phi$). 
\smallskip 

\item \textbf{Turnpike away from the middle of $[0,T]$.} Estimate \eqref{eq: unif.bound.explain} yields turnpike for $t\in[0,\tau+T_0]\cup[T-(\tau+T_0),T]$, where $\tau>0$ is a degree of freedom, independent of $T$, to be chosen later on, while $T_0$ is the controllability time for \eqref{eq: nonlinear.control.system}. Indeed, for $t\in[0,\tau+T_0]$, from \eqref{eq: unif.bound.explain} one sees that
\begin{equation*}
\|y_T(t)-\overline{y}\|\leqslant C_1e^{\mu t} e^{-\mu t}\leqslant C_1e^{\mu(\tau+T_0)}\Big(e^{-\mu t} + e^{-\mu(T-t)}\Big)
\end{equation*}
holds for all $\mu>0$, with $C_1>0$ as in \eqref{eq: unif.bound.explain} (thus independent of $T,\tau$). A similar computation can be repeated for $t\in[T-(\tau+T_0),T]$. At this point, one already notes that $T$ needs to be chosen sufficiently large, namely, 
\begin{equation}\label{eq: T>}
T>2(\tau+T_0).
\end{equation}
Actually, $T^*:=2(\tau+T_0)$ in Theorem \ref{thm: turnpike}.
\smallskip

\item \textbf{Turnpike in the middle of $[0,T]$.}
To obtain the exponential estimate for $t\in[\tau+T_0, T-(\tau+T_0)]$, it is critical to choose $\tau>0$ large enough. The clue is to first prove that there exists some constant $C_*>0$ independent of both $T$ and $\tau$ such that
\begin{equation} \label{eq: bootstrap.ineq}
\sup_{t\in[n\tau,T-n\tau]}\|y_T(t)-\overline{y}\|\leqslant\left(\frac{C_*}{\sqrt{\tau}}\right)^n
\end{equation}
holds for all integers $1\leqslant n\leqslant \frac{1}{\tau}(\frac{T}{2}-T_0)$. See \Cref{fig: fig.1} for a graphical depiction. 
By virtue of \eqref{eq: T>} we have $\frac{1}{\tau}(\frac{T}{2}-T_0)>1$, and this upper bound on $n$ will become clear in the next step.
Suppose that \eqref{eq: bootstrap.ineq} does indeed hold. Then from \eqref{eq: bootstrap.ineq}
\begin{equation} \label{eq: 10.13}
\|y_T(t)-\overline{y}\|\leqslant\exp\left(-n\log\left(\frac{\sqrt{\tau}}{C_*}\right)\right)
\end{equation}
holds for all $t\in[n\tau,T-n\tau]$ and $n$ as above. Choosing $\tau>C_*^2$, which up to this point was arbitrary, yields the positivity of the logarithm appearing in \eqref{eq: 10.13}: 
\begin{equation*}
\sigma:=\log\left(\frac{\sqrt{\tau}}{C_*}\right)>0.
\end{equation*}
Now fix $t\in[\tau+T_0,T-(\tau+T_0)].$
We look to choose the integer $n=n(t)$ as to have $t\in[n(t)\tau,T-n(t)\tau]$ as well as $1\leqslant n(t)\leqslant\frac{1}{\tau}(\frac{T}{2}-T_0)$, so that estimate \eqref{eq: 10.13} also holds for $t$ fixed as just before. 
After some elementary computations, one can see that for both these conditions to hold, it is necessary and sufficient for $n(t)$ to be such that 
\begin{equation*}
1\leqslant n(t)\leqslant\frac{t}{\tau+T_0} \hspace{0.5cm} \text{ and } \hspace{0.5cm} 1\leqslant n(t)\leqslant\frac{T-t}{\tau+T_0}.
\end{equation*}
This leads us to set
\begin{equation*}
n(t):=\min\left\{\left\lfloor\frac{t}{\tau+T_0}\right\rfloor,\left\lfloor\frac{T-t}{\tau+T_0}\right\rfloor\right\}.
\end{equation*}
With $t$ fixed as above, and $n(t)$ set as such, one sees that \eqref{eq: 10.13} holds. 
Namely, we have
\begin{equation} \label{eq: 10.14}
\|y_T(t)-\overline{y}\|\leqslant\exp\left(-n(t)\sigma\right).
\end{equation}
But furthermore, one of either 
\begin{equation*}
n(t)\geqslant\frac{t}{\tau+T_0}-1 \hspace{0.5cm} \text{ or } \hspace{0.5cm} n(t)\geqslant\frac{T-t}{\tau+T_0}-1
\end{equation*}
holds by definition of $n(t)$, and so
\begin{equation}\label{eq: 10.15}
\exp\left(-n(t)\sigma\right)\leqslant\exp(\sigma)\left(\exp\left(-\frac{t}{\tau+T_0}\sigma\right)+\exp\left(-\frac{T-t}{\tau+T_0}\sigma\right)\right).
\end{equation}
Since $t\in[\tau+T_0, T-(\tau+T_0)]$ was arbitrary, chaining \eqref{eq: 10.14} and \eqref{eq: 10.15} leads us to the desired turnpike inequality in $[\tau+T_0, T-(\tau+T_0)]$: 
\begin{equation*}
\|y_T(t)-\overline{y}\|\leqslant e^\sigma\Big(e^{-\mu t}+e^{-\mu (T-t)}\Big),
\end{equation*}
where 
\begin{equation*}
\mu:=\frac{\sigma}{\tau+T_0}=\frac{\log\left(\frac{\sqrt{\tau}}{C_*}\right)}{\tau+T_0}>0.
\end{equation*}
Thus, the proof would be complete once \eqref{eq: bootstrap.ineq} is shown to hold.
 \smallskip
 
 \begin{figure}
	\center
	\includegraphics[scale=0.75]{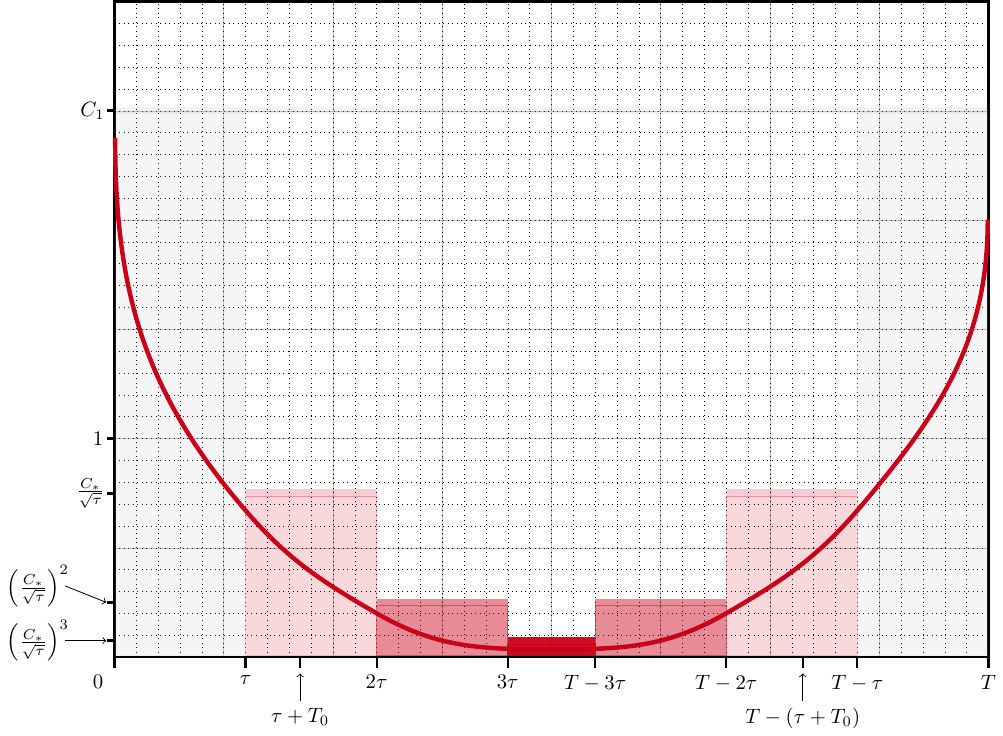}
\caption{The strategy is inspired by the shape of $t\mapsto e^{-t}+e^{-(T-t)}$ on $[0,T]$, which attains its minimum at $t=\frac{T}{2}$. Over successively smaller intervals of the form $[n\tau,T-n\tau]$, with $1\leqslant n\leqslant\frac{1}{\tau}(\frac{T}{2}-T_0)$, one sees that the upper bound in \eqref{eq: bootstrap.ineq}, which is uniform in $T$, decreases exponentially and captures the double-arc exponential of $\|y_T(t)-\overline{y}\|$ (\emph{dotted}) more accurately.} \label{fig: fig.1}

\end{figure}

\item \textbf{Proving \eqref{eq: bootstrap.ineq} by induction.}
To show \eqref{eq: bootstrap.ineq}, one performs an induction over $1\leqslant n\leqslant\frac{1}{\tau}(\frac{T}{2}-T_0)$. The upper bound on $n$ appears so that there are always a couple of disjoint intervals of length $T_0$ within every interval of the form $[n\tau,T-n\tau]$. Let us sketch the proof of the initialization stage $n=1$. 
Arguing by contradiction (see \Cref{lem: D1}), one can first find  
$\tau_1\in[0,\tau)$ and $\tau_2 \in (T-\tau,T]$ such that 
\begin{equation} \label{eq: 10.10}
\|y_T(\tau_j)-\overline{y}\|\stackrel{\mathrm{Lem.} \ref{lem: D1}}{\leqslant}\frac{\|y_T-\overline{y}\|_{L^2(0,T; \mathbb{R}^d)}}{\sqrt{\tau}}\stackrel{\eqref{eq: unif.bound.explain}}{\leqslant} \frac{C_1}{\sqrt{\tau}}.
\end{equation}
This allows us to recover the desired factor of an inverse of $\tau$.
Then, restricting $u_T$ to the subinterval $[\tau_1,\tau_2]$, one sees that it is a solution to
\begin{equation} \label{eq: aux.ocp}
\inf_{\substack{u\in L^2(\tau_1,\tau_2;\mathbb{R}^m)\\ \dot{y}=f(y,u) \text{ in } (\tau_1,\tau_2)\\y(\tau_1)=y_T(\tau_1)\\y(\tau_2)=y_T(\tau_2)} } \int_{\tau_1}^{\tau_2}\|y(t)-\overline{y}\|^2\diff t + \int_{\tau_1}^{\tau_2} \|u(t)\|^2\diff t.
\end{equation}
Just as in the first step of the proof, we look to estimate the functional minimized in \eqref{eq: aux.ocp}, but this time, not only uniformly in $T, \tau_1,\tau_2$, rather also in terms of the distance of $y_T(\tau_j)$ to $\overline{y}$. 
More specifically, we show that there exists some constant $C(r)>0$ independent of $T,\tau_1,\tau_2, \tau$ (but depending on $r, T_0, f$) such that
\begin{equation} \label{eq: 10.12}
\|y_T(t)-\overline{y}\|\leqslant C(r)\Big(\|y_T(\tau_1)-\overline{y}\| + \|y_T(\tau_2)-\overline{y}\|\Big)
\end{equation}
holds for all $t\in[\tau_1,\tau_2]$. 
To this end, we invoke \Cref{lem: C1}, and thus it suffices to bound the functional in \eqref{eq: aux.ocp} by means of the right-hand-side in \eqref{eq: 10.12}. To achieve this, since $\tau_2-\tau_1\geqslant 2T_0$, we may first control starting from $y_T(\tau_1)$ at time $t=\tau_1$ to $\overline{y}$ in time $t=\tau_1+T_0$; we stay at $\overline{y}$ by switching off the control until time $t=\tau_2-T_0$, from which we control to $y_T(\tau_2)$ in time $t=\tau_2$. By chaining this strategy into one single control\footnote{Such controls are referred to as \emph{quasi-turnpike controls}, since they (and their associated trajectories) look like rough approximations of exact turnpikes (see \Cref{fig: fig.3}).} $u^{\text{aux}}$, we note that $u^{\text{aux}}$ is precisely bounded by the right-hand-side in \eqref{eq: 10.12} via \Cref{def: ctrl} (we are in the desired ball by selecting $\tau>\sfrac{C_1^2}{r^2}$ in \eqref{eq: 10.10}), while the state tracking term can subsequently be covered by the Gr\"onwall inequality.
Estimate \eqref{eq: 10.12} combined with \eqref{eq: 10.10} leads to
\begin{equation*}
\|y_T(t)-\overline{y}\|\leqslant\frac{2C_1\cdot C(r)}{\sqrt{\tau}}
\end{equation*}
for all $t\in[\tau_1,\tau_2]$, and thus also for all $t\in[\tau,T-\tau]$. This is precisely \eqref{eq: bootstrap.ineq} for $n=1$. 
The above argument can then be repeated by induction on smaller intervals $[n\tau,T-n\tau]$ to obtain \eqref{eq: bootstrap.ineq} in the general case.
\end{itemize}
	
	\subsection{Discussion} \label{sec: discussion}
	 Several pertinent remarks are in order.

	\begin{remark}[On \Cref{def: ctrl}] \label{rem: later.on}

	\begin{itemize}	
	\item 
	In the driftless case ($f_0=0$ in \eqref{eq: nonlinearity.control.affine}), the  Chow-Rashevskii theorem (\cite[Chapter 3, Section 3.3]{coron2007control}), characterized by iterated Lie brackets, is a necessary and sufficient condition for the global exact controllability of systems with smooth vector fields.
	But general necessary and sufficient conditions which ensure the exact controllability of control-affine systems are not known to our knowledge -- see \cite[Chapter 3]{coron2007control}. 
	This is mainly due to the drift term $f_0$, which affects the geometry of the problem and may pose obstructions to the controllability in arbitrary time -- see \cite{beauchard2018quadratic} for a survey on these issues.
	We do insist however, that we merely require controllability in a possibly large time $T_0$, and not necessarily in any arbitrarily small time. 
	\smallskip
	
	\item 
	While we suppose that the underlying system is controllable for arbitrarily large data, through \eqref{eq: control.cost.estimate.1} -- \eqref{eq: control.cost.estimate.2} we solely assume that the cost is proportionate to the distance from the chosen steady state $\overline{y}$ in some, possibly arbitrarily small ball around this steady state.
	As the estimates \eqref{eq: control.cost.estimate.1} -- \eqref{eq: control.cost.estimate.2} are more commonly encountered in the linear systems setting $\dot{y}=Ay+Bu$, they thus also hold for semilinear systems where controllability is obtained by perturbation arguments. 
	In such contexts, it is well-known (see e.g. \cite[Remark 2.2]{zuazua2007controllability}) that the minimal $L^2$--norm control $u$ satisfies
	\begin{equation*}
	\|u\|_{L^2(0,T_0;\R^m)}\leqslant C({T_0})\Big(\left\|y^0\right\| +\left\|y^1\right\|\Big)
	\end{equation*}
	for some $C(T_0)>0$.
	This makes \Cref{def: ctrl} entirely plausible in the settings mentioned above. Indeed, we consider $z:=y-\overline{y}$, then either $z^0=0$ (if $y^0 = \overline{y}$) or $z^1=0$ (if $y^1=\overline{y}$). The control $u$ steering $y$ from $y^0$ to $y^1$ in time $T$ would then be the same as the one steering $z$ from either $0$ to $y^1-\overline{y}$ or from $y^0-\overline{y}$ to $0$ in time $T$, and the above estimate would yield the desired assumption.
	\smallskip
	
	\item 
	While there exist necessary and sufficient conditions for ensuring the exact controllability of driftless systems: $\dot{y}(t)=\sum_{j=1}^m u_j(t)f_j(y(t))$, we cannot ensure the validity of estimates \eqref{eq: control.cost.estimate.1} -- \eqref{eq: control.cost.estimate.2} in the underactuated regime, namely when $m<d$. This is due to the so-called ball-box theorem in sub-Riemannian geometry (\cite{agrachev2019comprehensive}), for smooth vector fields $f_1,\ldots,f_m$. This theorem states the following. Suppose that the vector fields $f_1,\ldots,f_m$ satisfy the H\"ormander condition, namely that the iterated Lie brackets of these vector fields at any point span $\mathbb{R}^d$. Denote $\bigtriangleup^1(x):=\mathrm{span}\{f_1(x),\ldots, f_m(x)\}$ for $x\in\mathbb{R}^d$, and $\bigtriangleup^{k+1}:=\bigtriangleup^1+[\bigtriangleup^{k},\bigtriangleup^1]$ for $k\geqslant1$. Then by virtue of the H\"ormander condition, there exists some $\kappa\geqslant1$ such that $\bigtriangleup^\kappa(x)=\mathbb{R}^d$ for all $x$. Furthermore, by the ball-box theorem, for $y^0$ close enough to $y^1$, an estimate of the form $\|y^0-y^1\|\lesssim d_{\mathrm{SR}}(y^0,y^1)\lesssim\|y^0-y^1\|^{\sfrac{1}{\kappa}}$ holds, where $d_{\mathrm{SR}}(y^0,y^1)$ is the sub-Riemannian distance of $y^0$ to $y^1$, equal (modulo a scalar multiple depending on $T_0$) to the $\inf$ defined in \eqref{eq: control.cost.estimate.1} -- \eqref{eq: control.cost.estimate.2}. 
	Herein, one sees that if $m\geqslant d$, it may happen to find at least $d$ among $m$ vector fields which are linearly independent, thus ensuring that $\kappa=1$, as desired; this is quite simply impossible when $m<d$. 
	 This exact constraint is also encountered in \cite[Theorem 5.1]{esteve2020large}, where the estimates \eqref{eq: control.cost.estimate.1} -- \eqref{eq: control.cost.estimate.2} are shown to hold for $m\geqslant d$ in the driftless setting. In view of this, generalizing the assumptions \eqref{eq: control.cost.estimate.1} -- \eqref{eq: control.cost.estimate.2} to fractional powers of the upper bounds appearing therein is an important open problem. 
	 Further clarity regarding this issue is also needed for general control-affine systems beyond semilinear systems, namely those for which linearization techniques might not apply. We refer to \cite{jean2015complexity, prandi2014holder} for developments in this direction.
	\end{itemize}
	\end{remark}
	
	\begin{remark}[On the time $T^*$] \label{rem: T*}
	Reading the sketch of proof, one notes that $T^*=2(T_0+\tau)$, where $T_0$ is the controllability time, and $\tau>0$ is chosen sufficiently large. Reading even further, one sees that $\tau$ needs to be at least larger than $\sfrac{C_1^2}{r^2}$, where the constant $C_1>0$ appears in \eqref{eq: JT.bounded} and \eqref{eq: unif.bound.explain}. 
	The latter constant is independent of $T$ and $r$, but does depend on $T_0$ (and the data $y^0,\overline{y}$) through the map 
	\begin{equation*}
	T_0\mapsto\mathfrak{C}(T_0):=\inf_{\substack{u\\y(0)=y^0\\y(T_0)=\overline{y}}}\left(\|u\|_{L^2(0, T_0;\mathbb{R}^m)}+\|y-\overline{y}\|_{L^2(0,T_0;\mathbb{R}^d)}\right).
	\end{equation*}
	Actually, due to the innate Gr\"onwall-based argument in \Cref{lem: C1}, $C_1$ will roughly be of the form $C_1=\mathfrak{C}(T_0)\exp({\mathfrak{C}(T_0)})$.
	If the system is controllable in any time, the cost $\mathfrak{C}(T_0)$ typically explodes as $T_0\searrow0$, and is bounded for $T_0\gg1$ (all relative to the distance of $y^0$ to $\overline{y}$). Therefore, according to our strategy, $T^*$ should increase at least linearly with $T_0\gg1$, and should explode with $\mathfrak{C}(T_0)$ when $T_0\searrow0$. This discussion also indicates the dependence of $T^*$ with respect to the radius $r$ of the ball in which the estimates of \Cref{def: ctrl} hold.
	\end{remark}
	
	\begin{remark}[On the constants $C$ and $\mu$] Once again, by reading the sketch of proof, one can see that the constants $C>0$ and $\mu>0$ appearing in the turnpike estimate \eqref{eq: turnpike.2} are explicit (albeit rather compound). 
	Per the sketch and \eqref{eq: mu}, $\mu>0$ is given by
	\begin{equation*}
	\mu:=\frac{\log\left(\frac{\sqrt{\tau}}{4C_\bullet^2}\right)}{\tau+T_0}=\frac{\log\left(\frac{\tau}{16C_\bullet^4}\right)}{T_*},
	\end{equation*}
	where $C_\bullet=C_\bullet(r,T_0)>0$ is the constant appearing in \eqref{eq: 10.12}; more specifically, the constant stemming from \Cref{lem: unif time estimates}. Moreover, $\tau>16C_\bullet^4+\frac{C_1^2+4C_1^2C_\bullet^2}{r^2}$ is arbitrary but fixed (as seen in \eqref{eq: buffer.tau}, where $C_1$ is the same as in \Cref{rem: T*}). On the other hand, per \eqref{eq: turnpike.constant.C}, the constant $C>0$ appearing in \eqref{eq: turnpike.2} takes the form
	\begin{equation*}
	C:=\max\left\{C_1,\frac{\sqrt{\tau}}{8C_\bullet^2}\max\left\{1,\frac{C_\bullet}{C_1}\right\}\right\}.
	\end{equation*} 
	\end{remark}
	
	\begin{remark}[On the nonlinearity] \label{rem: on.the.nonlinearity}

	With little modifications, \Cref{thm: turnpike} and \Cref{cor: stabilisation} also apply to system \eqref{eq: nonlinear.control.system} with nonlinearities $f$ of the form
	\begin{equation} \label{eq: nonlinearity.2}
		f(y,u) = \sum_{j=1}^m f_j(u_jy) \hspace{1cm} \text{ for } (y,u) \in \R^d\times \R^m
	\end{equation}
	where the vector fields $f_1,\ldots, f_m \in \Lip(\R^d;\R^d)$ are additionally assumed to be positively homogeneous of degree $1$.
	Such nonlinearities are motivated by \eqref{eq: neural.net}.
	Due to the homogeneity of the vector fields in \eqref{eq: nonlinearity.2}, the corresponding optimal steady states coincide with those of the driftless case, namely $(u_s, y_s) = (0, \overline{y})$ for any $\overline{y}\in\R^d$.
	\end{remark}
		   
	\section{Infinite-dimensional systems} \label{sec: pde}
	
	\noindent
	We illustrate the flexibility of the finite-dimensional arguments and adapt them to the semilinear wave and heat equation. 
	As a matter of fact, the only difference between the finite and infinite dimensional setting is in the proof of uniform control and state bounds by means of quasi-turnpike controls. The specific proof of turnpike is identical in both cases.
	We distinguish the case of the wave and heat equation because of the validity of the PDE analog of \Cref{def: ctrl}, as made more precise below.

	\subsection{Semilinear wave equation}
	
	Let $T>0$ and let $\Omega \subset \R^d$ be a bounded and (at least $C^2$)  regular domain.
	We will be interested in control systems of the form
	\begin{equation} \label{eq: nonlinear.wave.control}
	\begin{dcases}
	\del_t^2 y - \Delta y + f(y) = u\one_\omega &\text{ in } (0, T) \times \Omega\\
	y = 0 &\text{ on } (0, T) \times \del \Omega \\
	(y, \del_t y)|_{t=0} = \*y^0 &\text{ in } \Omega.
	\end{dcases}
	\end{equation}
	Here $f \in \Lip(\R)$, $\omega \subset \Omega$ is open (with geometric assumptions given in \eqref{eq: omega.def}), whereas $\*y^0 = \left(y^0_1, y^0_2\right)$ is a given initial datum. 
	It is well-known, by fixed-point arguments, that for any initial data $\*y^0 = \left(y^0_1, y^0_2\right) \in H^1_0(\Omega) \times L^2(\Omega)$ and for any $u \in L^2((0,T)\times\omega)$, there exists a unique finite-energy solution $
	y\in C^0([0,T]; H^1_0(\Omega))\cap C^1([0,T];L^2(\Omega))$
	to \eqref{eq: nonlinear.wave.control}.
	As in the finite-dimensional case, we will address the behavior when $T\gg 1$ of global minimizers $u_T \in L^2((0,T)\times \omega)$ to nonnegative functionals of the form
		\begin{equation} \label{eq: J_T.wave}
	J_T(u) := \phi(y(T)) + \int_0^T \|y(t)-\overline{y}\|_{H^1_0(\Omega)}^2\diff t + \int_0^T \|\del_t y(t)\|^2_{L^2(\Omega)}\diff t + \int_0^T \|u(t)\|_{L^2(\omega)}^2 \diff t,
	\end{equation}
    and of the corresponding solution $y_T$ to \eqref{eq: nonlinear.wave.control}. 
    Here $\phi:L^2(\Omega)\to\R_+$ is a given continuous functional, while $\overline{y}\in H^1_0(\Omega)$ is a running target which we select as an uncontrolled steady state of \eqref{eq: nonlinear.wave.control}, namely we assume that $\overline{y}$ is some solution\footnote{There is no need for the solution of \eqref{eq: steady.state.heat} to be unique.} to 
	\begin{equation} \label{eq: steady.state.heat}
	\begin{dcases}
	-\Delta \overline{y} + f(\overline{y}) = 0 &\text{ in }\Omega\\
	\overline{y} = 0 &\text{ on }\del\Omega.
	\end{dcases}
	\end{equation}
	We henceforth moreover assume that $f, \Omega$ are such that a solution to \eqref{eq: steady.state.heat} exists. This can be ensured in a variety of different cases, including, for instance (see \cite{cazenave2006introduction, lions1982existence} for further results):
	\begin{itemize}
	\item If $f(0) = 0$, then clearly $\overline{y}\equiv 0$ is one solution. 
	But if moreover there exist $p\in(1,\infty)$ for $d=1,2$ or $p\in\left(1, \sfrac{d+2}{d-2}\right)$, $\sigma<\lambda_1(\Omega)$ and $\theta>2$ such that
	\begin{align*}
	|f(s)|\leqslant C(1+|s|^p) \hspace{1cm} &\text{ for all }s\in\R \\
	-\int_0^s f(\zeta)\diff\zeta\leqslant\frac{\sigma}{2} s^2 \hspace{1cm} &\text{ for }|s|\ll1\\
	0 < -\theta \int_0^s f(\zeta)\diff \zeta \leqslant -s\int_0^s f(\zeta)\diff \zeta \hspace{1cm} &\text{ for }|s|\gg1
	\end{align*}
	then a nontrivial solution $\overline{y}\in H^1_0(\Omega)$, $\overline{y}\not\equiv 0$ also exists. We refer to \cite[Theorem 2.5.6]{cazenave2006introduction}.
	This fact is a consequence of the mountain pass theorem.
	Here $\lambda_1(\Omega)$ denotes the first eigenvalue of the Dirichlet Laplacian $-\Delta$.
	\smallskip
	
	\item When $d=1$ and $\Omega = (-R, R)$, then both necessary and sufficient conditions on $f$ can be provided ensuring the existence of nontrivial solutions -- see \cite[Theorem 1.2.3]{cazenave2006introduction}.
	\end{itemize}
	
	\noindent
	The case of a controlled steady state (namely adding $\overline{u}\one_\omega$ in \eqref{eq: steady.state.heat}) may also be considered, under the condition that the functional $J_T$ is modified appropriately as discussed in \Cref{rem: running.target}.
	The existence of minimizers to $J_T$ again follows by the direct method in the calculus of variations.
	We note that, since $\overline{y}$ is fixed as above, the pair $(u_s, y_s) \equiv (0, \overline{y})$ is the unique solution to the steady optimal control problem
	\begin{equation*}
	\inf_{\substack{(y,u)\in H^1_0(\Omega)\times L^2(\omega)\\ y \text{ solves} \eqref{eq: poisson.stat}}} \|y-\overline{y}\|^2_{H^1_0(\Omega)} + \|u\|^2_{L^2(\omega)}  
	\end{equation*}
	where the steady equation is 
	\begin{equation}\label{eq: poisson.stat}
	\begin{dcases}
	-\Delta y + f(y) = u\one_\omega &\text{ in } \Omega\\
	y = 0 &\text{ on } \del \Omega.
	\end{dcases}
	\end{equation}
	This is because the functional in the expression above attains its minimum, equal to $0$, precisely at $(0, \overline{y})$, a pair which satisfies the constraint provided by the elliptic equation.
	Before proceeding, we need to define the appropriate geometric setup for ensuring the exact controllability of \eqref{eq: nonlinear.wave.control} when $d\geqslant 2$. For any fixed $x_\circ\in\R^d \setminus \overline{\Omega}$, we define
	\begin{equation*}
	\Gamma(x_\circ):= \left\{x\in\del\Omega\colon(x-x_\circ)\cdot\nu(x)>0\right\}
	\end{equation*}
	where $\nu(x)$ denotes the outward unit normal at $x\in\del\Omega$. The set $\Gamma(x_\circ)$ coincides with the subset of the boundary arising usually in the context of the multiplier method \cite{lions1988controlabilite}. We will suppose that for some $\delta>0$ and $x_\circ\in\R^d \setminus \overline{\Omega}$,
	\begin{equation} \label{eq: omega.def}
	\omega = \mathscr{O}_\delta(\Gamma(x_\circ))\cap\Omega,
	\end{equation}
	where $\mathscr{O}_\delta(\Gamma(x_\circ)):=\left\{x \in \R^d \colon\|x-x'\| < \delta \text{ for some } x' \in \Gamma(x_\circ)\right\}$. 
	It is known that, under these geometric assumptions on $\omega$, and since $f \in \Lip(\R)$, the wave equation \eqref{eq: nonlinear.wave.control} is exactly controllable in any time $T>T_{\min}(\omega,\Omega)$, where 
	\begin{equation} \label{eq: Tmin}
	T_{\min}(\omega,\Omega)=2\max_{x\in\overline{\Omega}}\|x-x_\circ\|.
	\end{equation}
	(See \cite{fu2007exact, zhang2004exact, zuazua1991exact}, and also the introduction of \cite{joly2014note} for an ample survey of controllability results for semilinear wave equations.)		
	We may now state our main result in the context of the wave equation.  
	
	\begin{theorem}[Turnpike] \label{thm: turnpike.wave}
	
	Suppose that $f\in\Lip(\R)$ and $f(0)=0$. Let $\overline{y}\in H^1_0(\Omega)$ be any solution to \eqref{eq: steady.state.heat}.
	Let $\phi\in\mathscr{L}(L^2(\Omega);\R_+)$, and suppose that $\omega$ is as in \eqref{eq: omega.def}. 
	Then for any $\*y^0 \in H^1_0(\Omega)\times L^2(\Omega)$, there exists a time $T^*>T_{\min}(\omega, \Omega)$, and constants $C>0$ and $\mu>0$, such that for any $T>T^*$, any global minimizer $u_T\in L^2((0,T)\times\omega)$ to $J_T$ defined in \eqref{eq: J_T.wave} and corresponding optimal state $y_T$ solution to \eqref{eq: nonlinear.wave.control} satisfy
	\begin{equation*} \label{eq: turnpike.2.wave}
	\left\|y_T(t)-\overline{y}\right\|_{H^1_0(\Omega)} + \left\|\del_t y_T(t)\right\|_{L^2(\Omega)} \leqslant C\left( e^{-\mu t} + e^{-\mu(T-t)} \right)
	\end{equation*}
	for all $t \in [0, T]$, and
	\begin{equation*} \label{eq: turnpike.1.wave}
	\| u_T\|_{L^2((0,T)\times\omega)} \leqslant C.
	\end{equation*}
	Moreover, $\mu>0$ is independent of $\*y^0$.
	\end{theorem}	
	
	\noindent
	The proof of turnpike (see \Cref{sec: turnpike.wave.proof}) is identical to the finite-dimensional case.
	Some technical adaptations are however needed for obtaining the bounds through quasi-turnpike controls, wherein one uses the Duhamel formula for mild solutions in view of applying an integral Gr\"onwall inequality-based argument, in the spirit of the ODE setting. The assumption $f(0)=0$ is of technical nature, and is clarified in \Cref{rem: ctrl.wave}.
	
	\begin{remark}[On the choice of $J_T$]
	We note that in existing turnpike results for the wave equation, e.g. \cite{gugat2016optimal, trelat2018steady, zuazua2017large}, a slightly weaker functional is sometimes considered. For instance, in \cite{zuazua2017large} for the linear wave equation, only the $L^2(0,T;H^1_0(\Omega))$--norm of $y-\overline{y}$ is penalized, and not the $L^2((0,T)\times\Omega)$--norm of $\del_t y$, yet turnpike is shown to hold for the full state $(y, \del_t y)$. 
	This is justified by the equipartition of energy property, which states that, along a given time interval $[0,T]$, the energy concentrated on the $y$ component in $H^1_0(\Omega)$ and on the $\del_t y$ component in $L^2(\Omega)$ is comparably the same up to a compact remainder term.
	 We choose to work with a functional penalizing the full state of the system due to the specificity of our proof strategy.
	\end{remark}
	
	\noindent
	Similarly to the finite-dimensional case, when $\phi\equiv0$ in \eqref{eq: J_T.wave}, the strategy for proving \Cref{thm: turnpike.wave} can be slightly tweaked to obtain an exponential stabilization property for the optimal states. 
	
	\begin{cor}[Stabilization] \label{cor: stabilisation.wave}
	Suppose that $\phi\equiv 0$ in $J_T$ defined in \eqref{eq: J_T.wave}.
	Under the assumptions of \Cref{thm: turnpike.wave}, there exists a time $T^*>T_{\min}(\omega, \Omega)$, and constants $C>0$ and $\mu>0$, such that for any $T>T^*$, any global minimizer $u_T \in L^2((0,T)\times\omega)$ to $J_T$ defined in \eqref{eq: J_T.wave} and corresponding optimal state $y_T$ solution to \eqref{eq: nonlinear.wave.control} satisfy	
	\begin{equation*} 
	\left\|y_T(t)-\overline{y}\right\|_{H^1_0(\Omega)} + \left\|\del_t y_T(t)\right\|_{L^2(\Omega)} \leqslant Ce^{-\mu t}
	\end{equation*}
	for all $t\in [0,T]$ and
	\begin{equation*} 
	\|u_T\|_{L^2((0,T)\times\omega)} \leqslant C.
	\end{equation*}
	Moreover, $\mu>0$ is independent of $\*y^0$.
	\end{cor}	
	
	\subsection{Semilinear heat equation}
	
	To complete our presentation, we will also discuss control systems of the form
	\begin{equation} \label{eq: nonlinear.heat.control}
	\begin{dcases}
	\del_t y - \Delta y + f(y) = u\one_\omega &\text{ in } (0, T) \times \Omega\\
	y = 0 &\text{ on } (0, T) \times \del \Omega \\
	y|_{t=0} = y^0 &\text{ in } \Omega,
	\end{dcases}
	\end{equation}
	were $f \in \Lip(\R)$, $\omega \subset \Omega$ is any open, non-empty subset, 
	whereas $y^0$ is a given initial datum. 
	It is well-known that for any given $T>0$, $y^0 \in L^2(\Omega)$ and $u \in L^2((0,T)\times\omega)$, there exists a unique globally-defined solution $y \in C^0([0,T]; L^2(\Omega))\cap L^2(0,T;H^1_0(\Omega))$ to \eqref{eq: nonlinear.heat.control}. 
	We will again study global minimizers $u_T \in L^2((0,T)\times \omega)$ to nonnegative functionals of the form
		\begin{equation} \label{eq: J_T.heat}
	J_T(u) := \int_0^T \|y(t)-\overline{y}\|_{L^2(\Omega)}^2 \diff t + \int_0^T \|u(t)\|_{L^2(\omega)}^2 \diff t,
	\end{equation}
    and the corresponding solution $y_T$ to \eqref{eq: nonlinear.heat.control} in the regime $T\gg1$. 
    Once again, $\overline{y} \in L^2(\Omega)$ is a running target which we select as an uncontrolled steady state, namely a solution to \eqref{eq: steady.state.heat}.
    The existence of minimizers to $J_T$ defined in \eqref{eq: J_T.heat} follows by the direct method in the calculus of variations.
    		
    \begin{theorem}[Stabilization]  \label{thm: turnpike.heat}
    Suppose that $f\in\Lip(\R)$ and $\Omega\subset\R^d$ are such that \eqref{eq: steady.state.heat} admits at least one solution, and let $\overline{y}\in H^1_0(\Omega)$ be any such solution. 
   Then for any $y^0\in L^2(\Omega)$, there exist constants $C,\mu>0$ (depending solely on $f,\omega$) such that for any $T>0$, any global minimizer $u_T\in L^2((0,T)\times \omega)$ of $J_T$ defined in \eqref{eq: J_T.heat} and corresponding optimal state $y_T$ solution to \eqref{eq: nonlinear.heat.control} satisfy
	\begin{equation*} \label{eq: turnpike.3.heat}
	\left\|y_T(t)-\overline{y}\right\|_{L^2(\Omega)}\leqslant Ce^{-\mu t}\left\|y^0-\overline{y}\right\|_{L^2(\Omega)}
	\end{equation*}
	for all $t\in [0,T]$, and
	\begin{equation*}
	\|u_T\|_{L^2((0,T)\times\omega)}\leqslant C\left\|y^0-\overline{y}\right\|_{L^2(\Omega)}.
	\end{equation*}
    \end{theorem}
    
    \noindent
    We refer to \Cref{sec: turnpike.heat.proof} for the proof. 
       We consider the heat equation, in addition to the wave equation, because of the validity (or rather, the partial lack thereof) of the PDE analog of \Cref{def: ctrl}. 
    The heat equation is exactly controllable to controlled trajectories, namely solutions $\widehat{y}$ to \eqref{eq: nonlinear.heat.control} for given controls $\widehat{u}$. 
    Instead of an estimate such as \eqref{eq: control.cost.estimate.2}, one has 
    \begin{equation*}
    \left\|u-\widehat{u}\right\|_{L^2((0,T_0)\times\omega)}\leqslant C({T_0})\left\|y^0 - \widehat{y}(0)\right\|_{L^2(\Omega)}
    \end{equation*}
   for minimal $L^2$--norm controls $u$ steering $y$ to $\widehat{y}$ in time $T_0$ (see \cite[Lemma 8.3]{pighin2018controllability} and the references therein). 
    Such an estimate does not suffice for applying our methodology, as we clearly need to estimate the minimal $L^2$--norm control by means of the distance of the initial data to the target.
    Nonetheless, we illustrate that the stabilization result can be shown independently of the turnpike result. 
    Indeed, the proof closely follows that of \Cref{thm: turnpike}, with the exception that we only need to perform the bootstrap forward in time, whence we do not require that the system is controllable to anything else but a steady state.   
    The constants $C>0$ and $\mu>0$ are actually explicit (see \eqref{eq: heat.mu}, \eqref{eq: heat.C}) precisely due to the global validity of this estimate (i.e. \eqref{eq: cost.control.heat.1}); the factor $\left\|y^0-\overline{y}\right\|$ also appears because of this and due to the absence of a final cost $\phi$.
    
    The semilinear heat equation is a commonly used benchmark for nonlinear turnpike results, thus this example serves to compare with existing results. For instance, while we assume that the running targets are steady states, we make no smallness assumptions on the targets or on the initial data, unlike \cite{grune2021abstract, PZ2}. Furthermore, since we do not use (or thus linearize) the optimality system, we may work with solely globally Lipschitz nonlinearities, in which case the techniques of \cite{grune2021abstract, pighin2020semilinear, PZ2} do not apply. 
    
    \begin{remark}[On the nonlinearity] 
    The assumption that $f$ is globally Lipschitz in \eqref{eq: nonlinear.wave.control} and \eqref{eq: nonlinear.heat.control} could perhaps be relaxed to a locally Lipschitz $f$ (for which blow-up is avoided and controllability is ensured -- for instance, $f(y) = y^3$), under the condition that one can show a uniform $L^\infty((0,T)\times\Omega)$--estimate of $y_T$ with respect to $T>0$. 
    Arguments of this sort in the context of turnpike can be found in \cite{pighin2020semilinear} under smallness assumptions on the target.
    We refer to the end of \Cref{sec: outlook} for a discussion of a (possibly technical) impediment encountered in applying our methodology to the cubic heat equation.
    In addition to the controllability properties it entails for \eqref{eq: nonlinear.wave.control} -- \eqref{eq: nonlinear.heat.control} as blow-up is avoided,we use the Lipschitz character of $f$ in the estimates in \Cref{lem: quasiturnpike.wave.1}, \Cref{lem: quasiturnpike.2.wave} and \Cref{lem: quasiturnpike.heat.1}.
    \end{remark}
     
	\section{Preliminary results}
	\label{sec: preliminary.results}
	
	\noindent
	We begin by presenting a couple of simple but important lemmas, containing bounds of the quantity $\|y(t)-\overline{y}\|$ for both the nonlinear ODE and PDE setting, solely by means of $\|y^0-\overline{y}\|$ and the tracking terms appearing in the functional $J_T$.
	These bounds would thus imply that bounding the functional $J_T$ uniformly in $T$ would entail a bound for the desired quantity $\|y(t)-\overline{y}\|$.
	Let us begin with the ODE estimate.
		
	\begin{lemma} \label{lem: C1}
	Let $T>0$, and let $\overline{y} \in \R^d$ be as in \eqref{eq: f0}. 
	For any $u \in L^2(0,T; \R^m)$ and
	$y^0 \in \R^d$, let $y \in C^0([0, T]; \R^d)$ be the solution to \eqref{eq: nonlinear.control.system} with $y(0) = y^0$.
	Then there exist constants $C_1 = C_1(f, \overline{y})>0$ and $C_2=C_2(f)$ independent of $T$ such that
	\begin{equation*}
	\sup_{t\in[0,T]}\|y(t) -\overline{y}\|\leqslant C \left(\left\|y^0-\overline{y}\right\|+\|u\|_{L^2(0, T; \R^{m})}+\|y-\overline{y}\|_{L^2(0, T; \R^{d})} \right)
	\end{equation*} 
	holds, where
	\begin{equation*}
	C:= C_1\exp\Big(C_2 \|u\|_{L^2(0,T;\R^{m})}\Big).
	\end{equation*}
	\end{lemma}
	
	\noindent
	As insinuated by the form of the constant in the estimate, the proof follows a Gr\"onwall inequality-based argument. However, as this constant depends on $T$ only through the $L^2$--norm of the control $u$, we present the proof for the sake of clarity.
	
	\begin{proof}[Proof of \Cref{lem: C1}]
	
	Let us first suppose that $t\in[0,1]$. By integrating the equation satisfied by $y$ on $[0,t]$, namely writing
	\begin{align*}
	y(t) - \overline{y} &= y^0 - \overline{y} + \int_{0}^t \left(f_0(y) + \sum_{j=1}^m u_j f_j(y) \right) \diff \tau \\
	&= y^0 - \overline{y} + \int_{0}^t \left(f_0(y)-f_0(\overline{y})\right) \diff \tau  + \int_{0}^t \sum_{j=1}^m u_j \Big(f_j(y)-f_j(\overline{y})\Big)\diff \tau \\
	&\quad + \int_{0}^t \sum_{j=1}^m u_j f_j(\overline{y})\diff \tau,
	\end{align*}
	we see that, by using the fact that $f_0, \ldots, f_m \in \Lip(\R^d;\R^d)$ and the Cauchy-Schwarz inequality for the sums, 
	\begin{align*}
	\|y(t)-\overline{y}\|\leqslant\left\|y^0-\overline{y}\right\| + C(f)\int_{0}^t \Big(1+\|u(\tau)\|\Big) \|y(\tau)-\overline{y}\|\diff \tau + C_1(f,\overline{y})\int_{0}^t\|u(\tau)\|\diff\tau
	\end{align*}
	holds for some constant $C_1(f,\overline{y})>0$ independent of $T$. 
	Here and henceforth, $C(f)>0$ designates the largest among the Lipschitz constants of all $f_0, \ldots, f_m$. 
	Now applying the Cauchy-Schwarz inequality for the last term, the fact that $t\leqslant1$, and the Gr\"onwall inequality, in conjunction to the estimate just above, lead us to
	\begin{equation*}
	\|y(t)-\overline{y}\|\leqslant C_2\exp\left(C(f)\sqrt{1+\int_{0}^t \|u(\tau)\|^2\diff \tau}  \right) \left(\left\|y^0-\overline{y}\right\|+\left\|u\right\|_{L^2(0,T;\R^m)}\right),
	\end{equation*}
	for some $C_2(f, \overline{y})>0$, from which, using $\sqrt{x^2+y^2}\leqslant x+y$ for $x,y>0$, the desired statement readily follows.	
	Now suppose that $t\in(1,T]$. 
	We begin by showing that for any such $t$, there exists a $t^* \in (t-1, t]$ such that
	\begin{equation} \label{eq: step1.lemC1}
	\|y(t^*)- \overline{y}\|\leqslant\|y-\overline{y}\|_{L^2(0, T; \R^{d})}.
	\end{equation}
	To this end, we argue by contradiction. 
	Suppose that 
	\begin{equation*}
	\|y(t^*)-\overline{y}\|>\|y-\overline{y}\|_{L^2(0, T; \R^{d})}
	\end{equation*}
	for all $t^*\in(t-1, t]$. 
	Then
	\begin{align*}
		\|y-\overline{y}\|_{L^2(0, T; \R^{d})}^2 = \int_0^T \|y(t)-\overline{y}\|^2 \diff t \geqslant \int_{t-1}^t \|y(\tau)-\overline{y}\|^2 \diff \tau > \|y-\overline{y}\|_{L^2(0, T; \R^{d})}^2,
	\end{align*}
	which is a contradiction. 
	Thus \eqref{eq: step1.lemC1} holds. 
	Consequently, we know that there exists $t^* \in (t-1, t]$ such that \eqref{eq: step1.lemC1} holds. 
	By integrating the equation satisfied by $y$ in $[t^*, t]$, namely writing
	\begin{align*}
	y(t) - \overline{y} &= y(t^*) - \overline{y} + \int_{t^*}^t \left(f_0(y) + \sum_{j=1}^m u_j f_j(y) \right) \diff \tau \\
	&= y(t^*) - \overline{y} + \int_{t^*}^t \left(f_0(y)-f_0(\overline{y})\right) \diff \tau  + \int_{t^*}^t \sum_{j=1}^m u_j \Big(f_j(y)-f_j(\overline{y})\Big)\diff \tau \\
	&\quad + \int_{t^*}^t \sum_{j=1}^m u_j f_j(\overline{y})\diff \tau,
	\end{align*}
	we see that, by using the Lipschitz character of $f_0,\ldots, f_m$ and the Cauchy-Schwarz inequality for the sums,
	\begin{equation*}
	\|y(t)-\overline{y}\|\leqslant\|y(t^*)-\overline{y}\| + C(f)\int_{t^*}^t \Big(1+\|u(\tau)\|\Big) \| y(\tau)-\overline{y}\|\diff \tau + C_1(f,\overline{y}) \int_{t^*}^t\|u(\tau)\|\diff \tau.
	\end{equation*}
	Now applying the Cauchy-Schwarz inequality for the last term, the fact that $t-t^*\leqslant 1$, \eqref{eq: step1.lemC1}, and the Gr\"onwall inequality, in conjunction to the estimate just above, we obtain
	\begin{equation*}
	\|y(t)-\overline{y}\|\leqslant C_2\exp\left(C_3(f)\sqrt{1+\int_{t^*}^t \|u(\tau)\|^2\diff \tau}  \right) \left(\|y-\overline{y}\|_{L^2(0,T;\R^{d})} + \left\|u\right\|_{L^2(0,T;\R^m)}\right),
	\end{equation*}
	for some $C_2(f, \overline{y})>0$ and $C_3(f)>0$, from which, using $\sqrt{x^2+y^2}\leqslant x+y$ for $x,y>0$, the desired statement readily follows.
	\end{proof}
	
	\begin{remark} Let us make two brief observations.
	\begin{itemize}
	\item
	We note that in the case where the running target is $(\overline{u}, \overline{y})$ with $f(\overline{y}, \overline{u}) = 0$ and $\overline{u}\neq 0$, and thus we minimize $J_T$ defined in \eqref{eq: J_T.first},
	we argue as above to obtain a bound of the form
	\begin{equation*}
	\sup_{t\in[0,T]}\|y(t) -\overline{y}\|\leqslant C \left(\left\|y^0-\overline{y}\right\|+\|u-\overline{u}\|_{L^2(0, T; \R^{m})}+\|y-\overline{y}\|_{L^2(0, T; \R^{d})} \right)
	\end{equation*} 
	with $C \sim \exp\left(\|u-\overline{u}\|_{L^2(0,T; \R^m)}\right)$. Obtaining a dependence of the constant $C$ with respect to $\|u-\overline{u}\|_{L^2(0,T; \R^m)}$ rather than just $\|u\|_{L^2(0,T; \R^m)}$ is important, as by using the functional and optimality arguments, we will be able to obtain a uniform bound with respect to $T$ of the former, which does not necessarily entail a bound on the latter.
	The argument for deducing such a bound is identical to the proof of \Cref{lem: C1} -- assume that $m=1$ for notational simplicity, and observe that, since $f_0(\overline{y}) + \overline{u}f_1(\overline{y}) = 0$, 
	\begin{align*}
	y(t) - \overline{y} &= y(t^*) - \overline{y} + \int_{t^*}^t \left(f_0(y) - f_0(\overline{y})\right) \diff s + \int_{t^*}^t \left(u-\overline{u}\right)\left(f_1(y)-f_1(\overline{y})\right)\diff s \\
	&\quad+ \int_{t^*}^t \left(u-\overline{u}\right) f_1(\overline{y}) \diff s + \int_{t^*}^t \overline{u}\left(f_1(y)-f_1(\overline{y})\right)\diff s.
	\end{align*}
	One may then proceed as before.
	\smallskip
	\item It may readily be seen that if the control is of additive rather than multiplicative form, i.e. if $f_1, \ldots, f_m$ are nonzero constants, then the constant appearing in the estimate provided by \Cref{lem: C1} will not depend on the time horizon $T$.
	\end{itemize}
	\end{remark}
	
	\noindent
	We state and prove an analogous result for the semilinear heat equation \eqref{eq: nonlinear.heat.control}. 
	
	\begin{lemma} \label{lem: C2}
	Let $T>0$ be given, and let $\overline{y}$ be as in \eqref{eq: steady.state.heat}. 
	For any $u \in L^2((0,T)\times\omega)$ and
	$y^0 \in L^2(\Omega)$, let $y \in C^0([0, T]; L^2(\Omega)) \cap L^2(0,T; H^1_0(\Omega))$ be the unique weak solution to \eqref{eq: nonlinear.heat.control}.
	Then there exists a constant $C=C(f)>0$ independent of $T$ such that
	\begin{equation*}
	\|y(t) -\overline{y}\|_{L^2(\Omega)} \leqslant C \left(\left\|y^0-\overline{y}\right\|_{L^2(\Omega)}+\|u\|_{L^2((0,T)\times\omega)}+\|y-\overline{y}\|_{L^2((0, T)\times\Omega)} \right)
	\end{equation*} 
	holds for all $t \in [0, T]$.
	\end{lemma}
	
	\noindent
	The proof is almost identical to the ODE case, but we sketch it for the sake of clarity.
	
	\begin{proof}[Proof of \Cref{lem: C2}]
	The proof closely follows that of \Cref{lem: C1}. 
	We first note that by uniqueness, $y-\overline{y}$ can be shown (see \cite{amann1995linear}) to coincide with the unique mild solution to
	\begin{equation*}
	\begin{dcases}
	\del_t z - \Delta z + f(z+\overline{y}) - f(\overline{y}) = u\one_\omega &\text{ in } (0,T) \times \Omega \\
	z = 0 &\text{ on } (0,T) \times \del \Omega \\
	z|_{t=0} = y^0 - \overline{y} &\text{ in } \Omega
	\end{dcases}
	\end{equation*}
	 which is given by the Duhamel/variation by constants formula:
	\begin{align} \label{eq: duhamel.heat}
	y(t) - \overline{y} = e^{t\Delta}(y^0-\overline{y}) + \int_0^t e^{(t-s)\Delta} u(s)\one_\omega \diff s - \int_0^t e^{(t-s)\Delta}\Big(f(y) - f(\overline{y})\Big)\diff s,
	\end{align}
	where $\left\{e^{t\Delta}\right\}_{t\geqslant0}$ denotes the heat semigroup on $L^2(\Omega)$ generated by the Dirichlet Laplacian $-\Delta:H^2(\Omega)\cap H^1_0(\Omega)\to L^2(\Omega)$. 
	Of course, \eqref{eq: duhamel.heat} is interpreted as an identity in $L^2(\Omega)$.
	We may thus proceed and use \eqref{eq: duhamel.heat} throughout.
	First suppose that $0<t\leqslant 1$. 
	Using the well-known decay $\left\|e^{t\Delta}\right\|_{\mathscr{L}(L^2(\Omega))}\leqslant e^{-\lambda_1(\Omega)t}\leqslant 1$ of the heat semigroup, where $\lambda_1(\Omega)>0$ denotes the first eigenvalue of $-\Delta$, and the Lipschitz character of $f$, we find using \eqref{eq: duhamel.heat} that
	\begin{align*}
	\|y(t) - \overline{y}\|_{L^2(\Omega)}&\leqslant\left\|e^{t\Delta}\left(y^0-\overline{y}\right)\right\|_{L^2(\Omega)} + \int_0^t \left\|e^{(t-s)\Delta} u(s)\right\|_{L^2(\omega)}\diff s\\
	&\quad + \int_0^t \left\|e^{(t-s)\Delta}\Big(f(y(s)) - f(\overline{y})\Big)\right\|_{L^2(\Omega)}\diff s \\
	&\leqslant \left\|y^0-\overline{y}\right\|_{L^2(\Omega)} + \int_0^t \|u(s)\|_{L^2(\omega)}\diff s \\
	&\quad+ C_0 \int_0^t \|y(t)-\overline{y}\|_{L^2(\Omega)}\diff s,
	\end{align*}
	where $C_0 = C_0(f)>0$ is the Lipschitz constant of $f$.
	As $t\leqslant 1$, we may use the Cauchy-Schwarz and Gr\"onwall inequalities to conclude.
	Now suppose that $t\in(1, T]$. 
	Arguing as in the proof of \Cref{lem: C1}, we know that there exists a $t^* \in (t-1, t]$ such that
	\begin{equation} \label{eq: step1.C2}
	\|y(t^*)-\overline{y}\|_{L^2(\Omega)}\leqslant\|y-\overline{y}\|_{L^2((0,T)\times\Omega)}
	\end{equation}
	holds. By writing the Duhamel formula for $y-\overline{y}$ in $[t^*, t]$, namely writing
	\begin{align*}
	y(t) - \overline{y} = e^{t\Delta}\left(y(t^*)-\overline{y}\right) + \int_{t^*}^t e^{(t-s)\Delta} u(s)\*1_\omega \diff s- \int_{t^*}^t e^{(t-s)\Delta}\Big(f(y)-f(\overline{y})\Big)\diff s
	\end{align*}
	we see just as before that
	\begin{align*}
	\|y(t) - \overline{y}\|_{L^2(\Omega)} \leqslant \left\|y(t^*)-\overline{y}\right\|_{L^2(\Omega)}+\int_0^t \|u(s)\|_{L^2(\omega)}\diff s+C_0 \int_{t^*}^t \|y(t)-\overline{y}\|_{L^2(\Omega)}\diff s
	\end{align*}
	where $C_0 = C_0(f)>0$ is the Lipschitz constant of $f$.
	Using the fact that $t^*-t\leqslant 1$ and \eqref{eq: step1.C2}, we may, as before, apply the Cauchy-Schwarz and Gr\"onwall inequalities to conclude. 
	\end{proof}
	
	\noindent
	We finally show the analog estimate for the semilinear wave equation, which is, after defining the proper functional setup, identical to the proof of \Cref{lem: C2}.
	
	\begin{lemma} \label{lem: C2.wave} 
	Let $T>0$ be given, and let $\overline{y}$ be as in \eqref{eq: steady.state.heat}. 
	For any $u \in L^2((0,T)\times\omega)$ and
	$\*y^0 = (y^0_1, y^0_2) \in H^1_0(\Omega) \times L^2(\Omega)$, let $y \in C^0([0, T]; H^1_0(\Omega)) \cap C^1([0,T]; L^2(\Omega))$ be the unique weak solution to \eqref{eq: nonlinear.wave.control}.
	Then there exists a constant $C=C(f, \Omega)>0$ independent of $T$ such that
	\begin{align*}
	\|y(t) -\overline{y}\|_{H^1_0(\Omega)} + \|\del_t y(t)\|_{L^2(\Omega)}&\leqslant C\Big(\left\|y^0_1-\overline{y}\right\|_{H^1_0(\Omega)} + \left\|y^0_2\right\|_{L^2(\Omega)}+\|u\|_{L^2((0,T)\times\omega)}\\
	&\hspace{1cm}+\|y-\overline{y}\|_{L^2(0, T; H^1_0(\Omega))} + \left\|\del_t y\right\|_{L^2((0,T)\times\Omega)} \Big)
	\end{align*} 
	holds for all $t \in [0, T]$.
	\end{lemma}
	
	\begin{proof}[Proof of \Cref{lem: C2.wave}]
	Once \eqref{eq: nonlinear.wave.control} is written as a first order evolution equation in an appropriate Hilbert space $X$, the proof is identical to that of \Cref{lem: C2}.
	Define the energy space $X := H^1_0(\Omega) \times L^2(\Omega)$, and consider the closed, densely-defined operator 
	\begin{equation*}
	A:= \begin{bmatrix} 0 & \text{Id} \\ \Delta & 0 \end{bmatrix}, \hspace{1cm} D(A) = D(\Delta)\times H^1_0(\Omega),
	\end{equation*}
	where $D(\Delta) = H^2(\Omega)\cap H^1_0(\Omega)$.
	The operator $A$ is skew-adjoint and thus generates a strongly continuous semigroup $\left\{e^{tA}\right\}_{t\geqslant0}$ in $X$ by virtue of the Stone-Lumer-Phillips theorem (see e.g. \cite[Theorem 3.8.6]{tucsnak2009observation}). We now denote 
	\begin{equation*}
	\*y:= \begin{bmatrix} y \\ \del_t y \end{bmatrix}, \hspace{1cm} \overline{\*y} := \begin{bmatrix} \overline{y} \\ 0 \end{bmatrix}.
	\end{equation*}
	Analog arguments to those in \Cref{lem: C2} lead us to deduce that
	\begin{equation} \label{eq: duhamel.wave}
	\*y(t) - \overline{\*y} = e^{tA}\left(\*y^0 - \overline{\*y}\right) + \int_0^t e^{(t-s)A}\begin{bmatrix} 0 \\u(s)\one_\omega-f(y(s))+f(\overline{y})\end{bmatrix}\diff s
	\end{equation}
	for $t>0$ is the unique mild solution to the equation satisfied by the perturbation $\*y - \overline{\*y}$. Of course, \eqref{eq: duhamel.wave} is interpreted as an identity in $X$.
	By virtue of the conservative character of the semigroup, namely $\left\|e^{tA}g\right\|_X=\left\|g\right\|_X$ for all $t\geqslant0$ and $g\in X$, we see that one may apply precisely the same arguments as in the proof of \Cref{lem: C2}, this time to the integral formulation \eqref{eq: duhamel.wave} in $X$ (with an intermediate application of the Poincaré inequality after using the Lipschitz character of $f$) to conclude.  
	\end{proof}
		
	\section{Proof of \Cref{thm: turnpike}}
	\label{sec: proof.fin.dim}

	\noindent
	In this section, we present the proof of \Cref{thm: turnpike}, \Cref{cor: stabilisation} and \Cref{cor: control.decay}.
	The proof of \Cref{thm: turnpike} requires a couple of preliminary results. In particular, we will, by means of a quasi-turnpike control strategy, provide bounds -- uniform with respect to the time horizon $T$--  of the tracking terms appearing in the definition \eqref{eq: J_T} of the functional $J_T$ for the optimal control-state pairs $(u_T, y_T)$.

	\subsection{Quasi-turnpike lemmas}

	Both of the following results are heavily based on the specific choice of target $\overline{y}$ as a steady state of the nonlinear system without control, and on the Lipschitz character of the nonlinear terms.
	We begin with the following lemma.
	
	\begin{lemma} \label{lem: quasiturnpike.1} 
	Let $y^0 \in \R^d$ be given, and assume that system \eqref{eq: nonlinear.control.system} is controllable in some time $T_0>0$.
	Let $T>0$ be fixed, and let $u_T \in L^2(0,T; \R^m)$ be a 
	global minimizer of $J_T$ defined in \eqref{eq: J_T}, with $y_T$ denoting the associated solution to \eqref{eq: nonlinear.control.system} with $y_T(0)=y^0$.
	Then, there exists a constant $C = C(f, \phi, T_0,\overline{y}, y^0)>0$ independent of $T>0$ such that 
	\begin{equation} \label{eq: quasiturnpike.1}
	\left\|u_T\right\|_{L^2(0, T; \R^{m})} + \left\| y_T-\overline{y}\right\|_{L^2(0,T; \R^{d})} + \left\|y_T(t)-\overline{y}\right\| \leqslant C
	\end{equation}
	holds for all $t \in [0, T]$. 
	\end{lemma}
	
	\begin{figure}[h] 
	\center
	\includegraphics[scale=0.58]{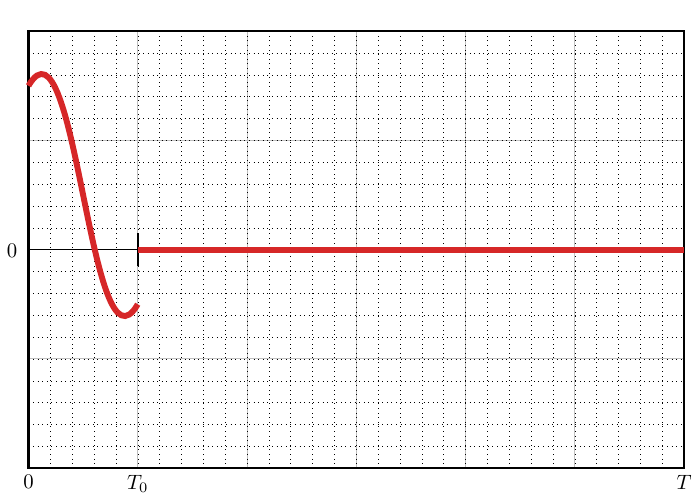}
	\includegraphics[scale=0.58]{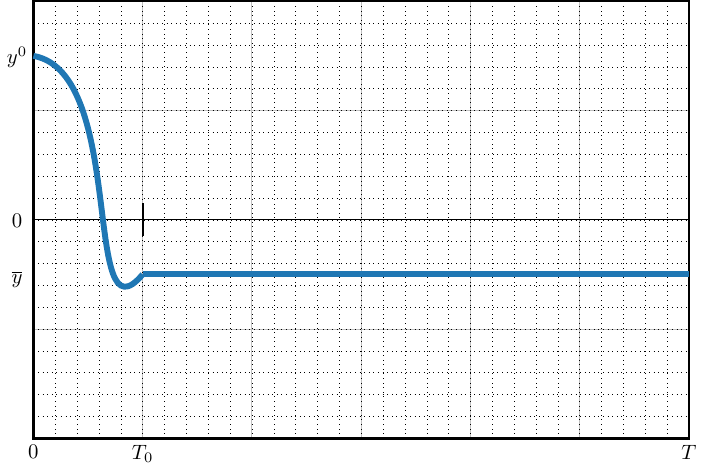}
        \caption{\textbf{Proof of \Cref{lem: quasiturnpike.1}.} The first two terms appearing in \eqref{eq: quasiturnpike.1} also appear in the functional $J_T(u_T)$. 
       	We construct a quasi-turnpike control $u^\aux$ (red), for which the corresponding state $y^\aux$ (blue) coincides with $\overline{y}$ over $(T_0, T)$.
	In this way, as $J_T(u_T) \leqslant J_T(u^\aux)$, and $J_T(u^\aux)$ is independent of $T$, we can conclude. The estimate of the third term then follows from \Cref{lem: C1}.
        } \label{fig: fig.2}
        \end{figure}
        
        \begin{proof}[Proof of \Cref{lem: quasiturnpike.1}] 
        
        \smallskip
        \noindent \textit{Case 1).}
        We begin by considering the case $T\geqslant T_0$. Using the controllability assumption, we know that there exists a control $u^{\dagger} \in L^2(0, T_0; \R^{m})$ such that the corresponding solution $y^{\dagger}$ to
	\begin{equation*}
	\begin{dcases}
	\dot{y}^\dagger = f\left(y^{\dagger}, u^{\dagger}\right) &\text{ in } (0, T_0) \\
	y^{\dagger}(0) = y^0
	\end{dcases}
	\end{equation*}
	satisfies $y^{\dagger}(T_0) = \overline{y}$.
	Now set
	\begin{equation*}
	u^\aux(t):= 
	\begin{dcases}
	u^{\dagger}(t) &\text{ in } (0, T_0) \\
	0 &\text{ in } (T_0, T)
	\end{dcases}
	\end{equation*}
	and let $y^\aux$ be the corresponding solution to \eqref{eq: nonlinear.control.system} with $y^\aux(0)=y^0$.
	Clearly $y^\aux(t) = \overline{y}$ for $t \in [T_0, T]$. 
	Hence, using $\phi\geqslant 0$ and $J_T(u_T) \leqslant J_T(u^\aux)$, we see that
	\begin{align*}
	\|y_T-\overline{y}\|_{L^2(0,T; \R^d)}^2 + \|u_T\|_{L^2(0, T;\R^m)}^2 \leqslant \phi(\overline{y}) + \left\|y^\dagger-\overline{y}\right\|_{L^2(0,T_0; \R^d)}^2 + \left\|u^\dagger\right\|_{L^2(0, T_0; \R^m)}^2.
	\end{align*}
	As the right-hand side in the above inequality is clearly independent of $T$, and depends solely on the $L^2(0,T_0)$ cost of controlling from $y^0$ to $\overline{y}$ in time $T_0$, we conclude the proof by applying \Cref{lem: C1} after noting the uniform boundedness of $\|u_T\|_{L^2(0,T;\R^{m})}$ with respect to $T>0$.  
	\smallskip
	
	\noindent
	\textit{Case 2).} Now suppose that $T\leqslant T_0$. 
        In this case, we use $\phi\geqslant 0$ and the optimality inequality $J_T(u_T)\leqslant J_T(u_{T_0})$ with the effect of obtaining
        \begin{align*}
	\left\|y_T-\overline{y}\right\|_{L^2(0,T;\R^d)}^2 &+ \left\|u_T\right\|_{L^2(0,T; \R^m)}^2 \leqslant\phi\left(y_{T_0}(T)\right) + \left\|y_{T_0}-\overline{y}\right\|_{L^2(0,T; \R^d)}^2 + \left\|u_{T_0}\right\|^2_{L^2(0,T; \R^m)}
        \end{align*}
        Now $y_{T_0}\in C^0([0,T_0]; \R^d)$ is uniformly bounded with respect to $T\in[0,T_0]$, whence, using the continuity of $\phi$, as well as $T\leqslant T_0$, we may conclude that 
        \begin{equation*}
        \left\|y_T-\overline{y}\right\|_{L^2(0,T; \R^d)}^2+ \left\|u_T\right\|_{L^2(0,T; \R^m)}^2 \leqslant C
        \end{equation*}
 	for some $C>0$ independent of $T$.
       	We may use \Cref{lem: C1} to conclude.     
        \end{proof}

	\noindent
	We will now focus on an auxiliary control problem with \emph{fixed} endpoints. 
	Namely, given $0\leqslant\tau_1<\tau_2 \leqslant T$, and $y^{\tau_1}, y^{\tau_2}\in\R^d$, this problem consists in minimizing the nonnegative functional
	\begin{equation} \label{eq: J_T-aux}
	J_{\tau_1, \tau_2}(u) := \int_{\tau_1}^{\tau_2} \|y(t)-\overline{y}\|^2\diff t + \int_{\tau_1}^{\tau_2} \|u(t)\|^2 \diff t
	\end{equation}
	over all $u \in U_\ad$, where $y \in C^0([\tau_1,\tau_2]; \R^d)$ denotes the unique solution to
	\begin{equation} \label{eq: system.tau1.tau2}
	\begin{dcases}
	\dot{y} = f(y, u) &\text{ in }(\tau_1, \tau_2) \\
	y(\tau_1)=y^{\tau_1}
	\end{dcases}
	\end{equation}
	and
	\begin{equation*}
	U_\ad := \Big\{ u \in L^2(\tau_1, \tau_2; \R^m) \colon y(\tau_2) = y^{\tau_2}\Big\}.
	\end{equation*} 
	The following lemma is of key importance in what follows. 
	It ensures that the optimal controls (for $J_{\tau_1, \tau_2}$) and trajectories are in fact bounded by means of the distance of the starting point $y^{\tau_1}$ and endpoint $y^{\tau_2}$ from the running target $\overline{y}$. This estimate will be the cornerstone of the bootstrap argument performed in the proof of \Cref{thm: turnpike}.
	
	\begin{lemma} \label{lem: unif time estimates} 
	Let $\overline{y} \in \R^d$ be as in \eqref{eq: f0}, and assume that system \eqref{eq: nonlinear.control.system} is controllable in some time $T_0>0$ in the sense of \Cref{def: ctrl}. Let $r>0$ be the radius provided by \Cref{def: ctrl}, let $0 \leqslant \tau_1 < \tau_2 \leqslant T$ be fixed such that $\tau_2-\tau_1\geqslant 2T_0$, and let $y^{\tau_1}, y^{\tau_2} \in \R^d$ be such that 
	\begin{equation*}
	\left\|y^{\tau_i}-\overline{y}\right\|\leqslant r
	\end{equation*}
	for $i=1,2$. Suppose $u_T \in U_\ad$ is a global minimizer to $J_{\tau_1, \tau_2}$ defined in \eqref{eq: J_T-aux}, with $y_T$ denoting the associated solution to \eqref{eq: system.tau1.tau2} with $y_T(\tau_2)=y^{\tau_2}$.
	Then, there exists a constant $C = C(f, T_0,\overline{y}, r)>0$ independent of $T, \tau_1, \tau_2>0$ such that 
	\begin{equation*}
	\left\|u_T\right\|_{L^2(\tau_1, \tau_2; \R^{m})}^2 + \left\|y_T-\overline{y}\right\|_{L^2(\tau_1,\tau_2; \R^{d})}^2 + \left\|y_T(t)-\overline{y}\right\|^2 \leqslant C \left(\left\|y^{\tau_1}-\overline{y}\right\|^2+\left\|y^{\tau_2}-\overline{y}\right\|^2\right)
	\end{equation*}
	holds for all $t \in [\tau_1, \tau_2]$. Moreover, the map $r \longmapsto C(f, T_0, \overline{y}, r)$ is non-decreasing as a function from $\R_+$ to $\R_+$.
	\end{lemma}
	
	\noindent
	The key idea of the proof of \Cref{lem: unif time estimates} lies in the construction of a quasi-turnpike control (steering the corresponding trajectory from $y^{\tau_1}$ to $y^{\tau_2}$ in time $\tau_2-\tau_1$, whilst remaining at $\overline{y}$ over an interval of length $\tau_2-\tau_1-2T_0$; see the figure just below) in view of estimating each individual addend of $J_{\tau_1,\tau_2}(u_T)$, which is the minimal value of the functional $J_{\tau_1,\tau_2}$. 
	This construction will yield the desired result.
	
	\begin{figure}[h]
	\center
	\includegraphics[scale=0.58]{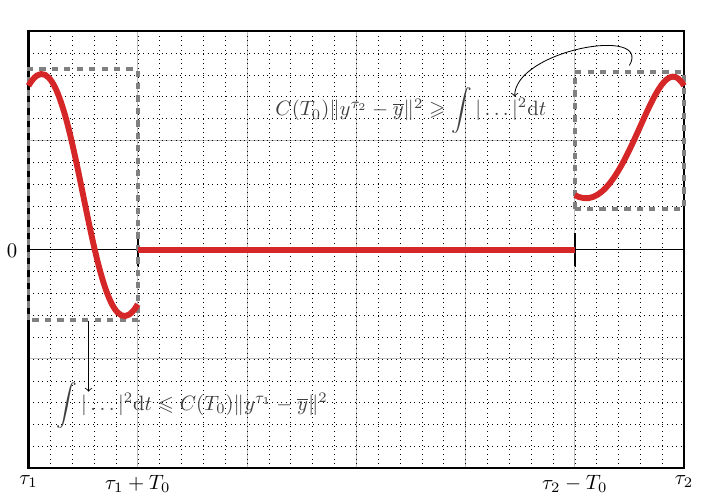}
	\includegraphics[scale=0.58]{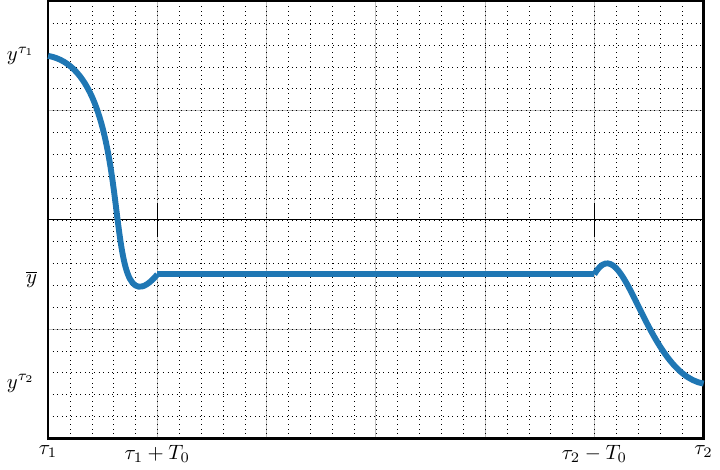}
	\caption{\textbf{Proof of \Cref{lem: unif time estimates}.} The first two terms appearing in the estimate implied by \Cref{lem: unif time estimates} also appear in the functional $J_{\tau_1, \tau_2}(u_T)$. 
       	We construct a quasi-turnpike control $u^\aux$ (red), for which the corresponding state $y^\aux$ (blue) coincides with $\overline{y}$ over $(\tau_1+T_0, \tau_2-T_0)$.
	In this way, as $J_{\tau_1, \tau_2}(u_T) \leqslant J_{\tau_1, \tau_2}(u^\aux)$, and $J_{\tau_1, \tau_2}(u^\aux)$ is independent of $T, \tau_1, \tau_2$, we can conclude. The estimate of the third term follows from \Cref{lem: C1}.
	} \label{fig: fig.3}
	\end{figure}

	\begin{proof}[Proof of \Cref{lem: unif time estimates}]
      
	Using the controllability assumption, we know the following.
	\begin{itemize}
	\item There exists a control $u^{\dagger} \in L^2(\tau_1, \tau_1+T_0; \R^{m})$ satisfying
	\begin{equation} \label{eq: first.control.estimate}
	\left\|u^{\dagger}\right\|_{L^2(\tau_1, \tau_1+T_0; \R^{m})}^2 \leqslant C(T_0) \left\|y^{\tau_1}-\overline{y}\right\|^2,
	\end{equation}
	for some $C(T_0)>0$ independent of $y^{\tau_1},\tau_1$, and which is such that the corresponding solution $y^{\dagger}$ to
	\begin{equation}\label{PC_xdagger}
	\begin{dcases}
	\dot{y}^\dagger = f\left(y^{\dagger}, u^{\dagger}\right) &\text{ in } (\tau_1, \tau_1+T_0) \\
	y^{\dagger}(\tau_1) = y^{\tau_1}
	\end{dcases}
	\end{equation}
	satisfies $y^{\dagger}(\tau_1+T_0) = \overline{y}$. By integrating \eqref{PC_xdagger}, and using the Lipschitz character of $f_0, \ldots, f_m$, the Gr\"onwall inequality, the Cauchy-Schwarz inequality, and \eqref{eq: first.control.estimate}, we see that
	\begin{align} \label{eq: first.state.estimate_boundedness}
	\left\| y^{\dagger}(t)\right\| &\leqslant C_0 \left(\left\|y^{\tau_1}\right\|+\left\|u^\dagger\right\|_{L^2(\tau_1,\tau_1+T_0; \R^{m})} + 1\right)\exp\left(C_0\left\|u^\dagger\right\|_{L^2(\tau_1,\tau_1+T_0; \R^{m})}\right)\nonumber \\
	&\leqslant C_1 \Big(\left\|y^{\tau_1}\right\|+ \left\|y^{\tau_1}-\overline{y}\right\|+1 \Big)\exp\Big(C_1\left\|y^{\tau_1}-\overline{y}\right\|\Big)\nonumber \\
	&\leqslant C_1\Big(\left\|y^{\tau_1}\right\|+r+1 \Big)\exp\Big(C_1 r\Big)\nonumber \\
	&\leqslant C_2 \Big(\left\|\overline{y}\right\|+r+1 \Big)\exp\Big(C_2 r\Big)
	\end{align}
	for some $C_0=C_0(f, T_0)>0$, $C_1=C_1(f,T_0)>0$, $C_2=C_2(f,T_0)>0$, and for every $t\in (\tau_1,\tau_1+T_0)$. 
	Then, by integrating \eqref{PC_xdagger} once again, and using $f_0(\overline{y})=0$, the Cauchy-Schwarz inequality and \eqref{eq: first.state.estimate_boundedness}, we moreover see that
	\begin{align} \label{eq: aux.quasiturnpike.est.1}
	\left\|y^\dagger(t) - \overline{y} \right\| &\leqslant \left\|y^{\tau_1} -\overline{y} \right\| + \int_{\tau_1}^t \sum_{j=1}^m \left|u^\dagger_j(s)\right| \left\|f_j(y^\dagger)\right\| \diff s + \int_{\tau_1}^t \left\|f\big(y^\dagger\big)-f(\overline{y})\right\|\diff s \nonumber\\
	&\leqslant \left\|y^{\tau_1} -\overline{y} \right\| + C_3\left\|u^\dagger\right\|_{L^2(\tau_1, \tau_1+T_0; \R^m)} + C(f) \int_{\tau_1}^t \left\|y^\dagger(s)-\overline{y}\right\|\diff s
	\end{align}
	for some $C_3(f, T_0, r,\overline{y})>0$, with $C(f)>0$ being the Lipschitz constant of the vector fields $f_j$.
	Finally, applying the Gr\"onwall inequality to \eqref{eq: aux.quasiturnpike.est.1} and using \eqref{eq: first.control.estimate}, we deduce that
	\begin{align} \label{eq: first.state.estimate}
	\left\|y^\dagger(t) - \overline{y}\right\|\leqslant C_4 \exp\left(C(f)T_0\right)\left\|y^{\tau_1} - \overline{y}\right\|
	\end{align}
	for some $C_4(f,T_0, \overline{y}, r)>0$ independent of $T, \tau_1$ and $\tau_2>0$, and for every $t\in(\tau_1, \tau_1+T_0)$.
	Note that in view of \eqref{eq: first.state.estimate_boundedness}, both $C_3$ and $C_4$ are non-decreasing with respect to the parameter $r>0$.
	\smallskip 
	
	\item There exists a control  $u^\ddagger \in L^2(\tau_1, \tau_1+T_0; \R^{m})$ satisfying
	\begin{equation} \label{eq: second.control.estimate}
	\left\|u^{\ddagger}\right\|_{L^2(\tau_1, \tau_1+T_0; \R^{m})}^2 \leqslant C(T_0) \left\|\overline{y}-y^{\tau_2}\right\|^2,
	\end{equation}
	and which is such that the corresponding solution $y^\ddagger$ to
	\begin{equation} \label{eq: xddagger}
	\begin{dcases}
	\dot{y}^\ddagger = f\left(y^\ddagger, u^\ddagger\right) &\text{ in } (\tau_1, \tau_1+T_0) \\
	y^\ddagger(\tau_1) = \overline{y}
	\end{dcases}
	\end{equation}
	satisfies $y^\ddagger(\tau_1+T_0) = y^{\tau_2}$. 
	By integrating \eqref{eq: xddagger}, and using the Lipschitz character of $f_0, \ldots, f_m$, the Gr\"onwall inequality, the Cauchy-Schwarz inequality and \eqref{eq: second.control.estimate}, we see that
	\begin{align}\label{eq: second.state.estimate_boundedness}
	\left\| y^{\ddagger}(t)\right\| &\leqslant C_5 \left(\left\|\overline{y}\right\|+\left\|u^\ddagger\right\|_{L^2(\tau_1,\tau_1+T_0; \R^m)}+1 \right)\exp\left(C_5\left\|u^\ddagger\right\|_{L^2(\tau_1,\tau_1+T_0; \R^{m})}\right)\nonumber \\
	&\leqslant C_6 \Big(\left\|\overline{y}\right\|+\left\|\overline{y}-y^{\tau_2}\right\|+1 \Big)\exp\Big(C_6 \left\|\overline{y}-y^{\tau_2}\right\|\Big)\nonumber \\
	&\leqslant C_6 \Big(\left\|\overline{y}\right\|+ r+1 \Big)\exp\Big(C_6 r\Big)
	\end{align}
	for some $C_5(f)>0$ and $C_6(f,T_0)>0$, and for every $t\in (\tau_1,\tau_1+T_0)$. 
	Then, by integrating \eqref{eq: xddagger} once again, and using $f_0(\overline{y})=0$, the Cauchy-Schwarz inequality and \eqref{eq: second.state.estimate_boundedness}, we moreover see that
	\begin{align} \label{eq: aux.quasiturnpike.est.2}
	\left\|y^\ddagger(t) - \overline{y} \right\| &\leqslant \int_{\tau_1}^t \sum_{j=1}^m \left|u^\ddagger_j(s)\right| \left\|f_j(y^\ddagger)\right\| \diff s + \int_{\tau_1}^t \left\|f\big(y^\ddagger\big)-f(\overline{y})\right\|\diff s \nonumber\\
	&\leqslant C_7\left\|u^\ddagger\right\|_{L^2(\tau_1, \tau_1+T_0; \R^m)} + C(f) \int_{\tau_1}^t \left\|y^\ddagger(s)-\overline{y}\right\|\diff s
	\end{align}
	for some $C_7(f, T_0, r,\overline{y})>0$, with $C(f)>0$ being the Lipschitz constant of the vector fields $f_j$.
	Finally, applying the Gr\"onwall inequality to \eqref{eq: aux.quasiturnpike.est.2} and using \eqref{eq: second.control.estimate}, we deduce that
	\begin{align} \label{eq: second.state.estimate}
	\left\|y^\ddagger(t) - \overline{y}\right\|\leqslant C_8 \exp\left(C(f)T_0\right)\left\|y^{\tau_2} - \overline{y}\right\|
	\end{align}
	for some $C_8(f,T_0, \overline{y}, r)>0$ independent of $T, \tau_1, \tau_2>0$, and for every $t \in (\tau_1, \tau_1+T_0)$.
	Note that in view of \eqref{eq: first.state.estimate_boundedness}, both $C_7$ and $C_8$ are non-decreasing with respect to the parameter $r>0$.
	\end{itemize}
	Now set
	\begin{equation*}
	u^\aux(t):= 
	\begin{dcases}
	u^{\dagger}(t) &\text{ in } (\tau_1, \tau_1+T_0) \\
	0 &\text{ in } (\tau_1+T_0, \tau_2-T_0) \\
	u^\ddagger\left(t-(\tau_2-\tau_1-T_0)\right) &\text{ in } (\tau_2-T_0, \tau_2),
	\end{dcases}
	\end{equation*}
	and let $y^\aux$ be the corresponding solution to \eqref{eq: system.tau1.tau2}.
	By construction, we have 
	\begin{equation*}
	y^\aux(t) = y^\dagger(t) \hspace{1cm} \text{ in } [\tau_1,\tau_1+T_0],
	\end{equation*}
	and thus
	\begin{equation} \label{eq: aux.eq.1}
	y^\aux(t) = \overline{y} \hspace{1cm} \text{ in } [\tau_1+T_0, \tau_2-T_0], 
	\end{equation}
	whereas we also have $y^\aux(\tau_2) = y^{\tau_2}$, whence $u^\aux \in U_\ad$.
	We now evaluate $J_{\tau_1,\tau_2}$ at $u^\aux$, which by virtue of a simple change of variable as well as \eqref{eq: aux.eq.1}, \eqref{eq: first.control.estimate}, \eqref{eq: first.state.estimate}, \eqref{eq: second.control.estimate} and \eqref{eq: second.state.estimate}, leads us to
	\begin{align} \label{eq: aux.estimate.1}
	J_{\tau_1, \tau_2}(u^\aux) &= \left\| u^\dagger \right\|_{L^2(\tau_1,\tau_1+T_0;\R^{m})} + \left\| u^\ddagger \right\|_{L^2(\tau_1,\tau_1+T_0;\R^{m})} \nonumber \\
	&\quad+ \int_{\tau_1}^{\tau_1+T_0} \left\|y^\dagger(t)-\overline{y}\right\|^2 \diff t + \int_{\tau_1}^{\tau_1+T_0} \left\|y^\ddagger(t)-\overline{y}\right\|^2\diff t \nonumber \\
	&\leqslant C_{9} \Big(\left\|\overline{y}-y^{\tau_1}\right\|^2+\left\|\overline{y}-y^{\tau_2}\right\|^2\Big)
	\end{align}
	where $C_{9} = C_{9}(f,\overline{y},T_0,r)>0$ is independent of $T, \tau_1, \tau_2>0$, and is non-decreasing with respect to $r$.
	Hence $u_T \in U_\ad$ is uniformly bounded with respect to $T, \tau_1, \tau_2>0$, as in view of \eqref{eq: aux.estimate.1} we have
	\begin{align*}
	\left\|y_T-\overline{y}\right\|_{L^2(\tau_1,\tau_2;\R^{d})}^2 + \left\|u_T\right\|_{L^2(\tau_1,\tau_2;\R^{m})}^2 \leqslant J_{\tau_1,\tau_2}\left(u_T\right) &\leqslant J_{\tau_1, \tau_2}\left(u^\aux\right) \\
	&\leqslant C_{9} \Big(\left\|\overline{y}-y^{\tau_1}\right\|^2+\left\|\overline{y}-y^{\tau_2}\right\|^2\Big).
	\end{align*} 
	An application of \Cref{lem: C1} combined with the uniform boundedness of $\|u_T\|_{L^2(\tau_1,\tau_2;\R^{m})}$ with respect to $T, \tau_2, \tau_1>0$ suffices to conclude.
	\end{proof}
	
	\noindent
	Before proceeding with the proof of \Cref{thm: turnpike}, we will need the following key lemma.

	\begin{lemma}
	\label{lem: D1}
	Let $X$ be a Banach space, $T>0$ and $f \in C^0([0, T]; X)$. For any $\tau \leqslant \frac{T}{2}$, there exist $t_1 \in [0, \tau)$ and $t_2 \in (T-\tau, T]$ such that 
	\begin{equation*}
	\|f(t_i)\|_X \leqslant \frac{\|f\|_{L^2(0, T; X)}}{\sqrt{\tau}} \hspace{1cm} \text{ for } i =1, 2.
	\end{equation*}
	\end{lemma}
	
	\begin{proof}[Proof of \Cref{lem: D1}]
	Denote 
	\begin{equation*}
	\eta(\tau):=\frac{\|f\|_{L^2(0, T; X)}}{\sqrt{\tau}}.
	\end{equation*}
	We argue by contradiction. Assume that either 
	\begin{equation*}
	\|f(t)\|_X > \eta(\tau) \hspace{1cm} \text{ for all } t \in [0, \tau)
	\end{equation*}
	or 
	\begin{equation*}
	\|f(t)\|_X > \eta(\tau) \hspace{1cm} \text{ for all } t \in (T- \tau, T].
	\end{equation*}
	hold. Then we have
	\begin{equation*}
	\int_0^T \|f(t)\|_X^2 \diff t \geqslant \int_0^\tau \|f(t)\|^2_X \diff t + \int_{T-\tau}^T \|f(t)\|^2_X \diff t > \tau \eta(\tau)^2.
	\end{equation*}
	Hence 
	\begin{equation*}
	\eta(\tau)^2 < \frac{1}{\tau} \int_0^T \|f(t)\|^2_X \diff t = \eta(\tau)^2,
	\end{equation*}
	which yields a contradiction. This concludes the proof.
	\end{proof}
	
	\subsection{Proof of \Cref{thm: turnpike}} \label{sec: proof.turnpike} 
	
	We are now in a position to prove our first main result.	
	
	\begin{proof}[Proof of \Cref{thm: turnpike}]
	
	We begin by noting that \eqref{eq: turnpike.1} follows from \Cref{lem: quasiturnpike.1}. 
	We thus concentrate on proving \eqref{eq: turnpike.2}, and we split the proof in two parts.
	Before proceeding, let us first note that by \Cref{lem: quasiturnpike.1}, there exists a constant $C_1>0$, depending only on $f, T_0, \overline{y}, y^0, \phi$, such that
	\begin{equation} \label{eq: c1.def}
	\sup_{t\in[0,T]}\|y_T(t) - \overline{y}\|+\|y_T-\overline{y}\|_{L^2(0,T;\mathbb{R}^d)} \leqslant C_1
	\end{equation}
	holds for any $T>0$ and pair $(u_T,y_T)$ which is optimal for \eqref{eq: J_T}.
	Let $r>0$ be the radius provided by \Cref{def: ctrl}.
	By \Cref{lem: unif time estimates}, we also know that there exists a constant $C_2>0$, depending only on $r, f, T_0,\overline{y}$, such that whenever $T\geqslant 2T_0$, for any $\tau_1, \tau_2 \in [0,T]$ such that $\tau_2-\tau_1\geqslant 2 T_0$ and
	\begin{equation*}
	\|y_T(\tau_i)-\overline{y}\| \leqslant r,
	\end{equation*}	
	the estimate
	\begin{equation} \label{eq: c2.def}
	\sup_{t \in [\tau_1, \tau_2]}\|y_T(t) - \overline{y}\|+\|y_T-\overline{y}\|_{L^2(\tau_1,\tau_2;\mathbb{R}^d)}\leqslant C_2\Big(\|y_T(\tau_2)-\overline{y}\|+ \|y_T(\tau_1)-\overline{y}\|\Big)
	\end{equation}
	holds for any pair $(u_T,y_T)$ which is optimal for \eqref{eq: J_T-aux}.
	Now fix
	\begin{equation} \label{eq: buffer.tau}
	\tau > 16C_2^4 + \frac{C_1^2}{r^2} + \frac{4C_1^2C_2^2}{r^2},
	\end{equation}
	and let
	\begin{equation*}
	T>2(\tau+T_0):=T^*
	\end{equation*}
	be fixed. Let $(u_T,y_T)$ thus be an optimal pair for \eqref{eq: J_T} with $T$ as such.
	The choice of the buffer time $\tau$ will become clear in what follows. 
	\smallskip
	
	\noindent
	\textbf{Part 1.} We note that for $t\in[0, \tau+T_0]$ and $t\in[T-(\tau+T_0), T]$, the desired estimate \eqref{eq: turnpike.2} can be obtained without too much difficulty, as the length of both time intervals is independent of $T$.
	Indeed, by \eqref{eq: c1.def}, for any $\mu>0$ we have
	\begin{align} \label{eq: 0-tau+T}
	\|y_T(t)-\overline{y}\|&\leqslant C_1 e^{\mu t} e^{-\mu t}\nonumber\\
	&\leqslant C_1 e^{\mu (\tau+T_0)} \Big(e^{-\mu t} + e^{-\mu(T-t)}\Big)
	\end{align} 
	for $t\in[0, \tau+T_0]$, and 
	\begin{align} \label{eq: T-tau+T0}
	\|y_T(t)-\overline{y}\|&\leqslant C_1 e^{\mu (T-t)} e^{-\mu (T-t)}\nonumber\\
	&\leqslant C_1 e^{\mu (\tau+T_0)} \Big(e^{-\mu t} + e^{-\mu(T-t)}\Big)
	\end{align}
	for $t\in[T-(\tau+T_0), T]$. 	
	\smallskip
	
	\noindent
	\textbf{Part 2.}	
	We now aim to show that \eqref{eq: turnpike.2} holds for $t\in[\tau+T_0, T-(\tau+T_0)]$. To this end, we proceed in three steps.
	\smallskip
	
	\noindent
	\emph{Step 1): Preparation.} 
	Since $\tau\leqslant\frac{T}{2}$, by \Cref{lem: D1} there exist a couple of time instances $\tau_1\in[0, \tau)$ and $\tau_2\in(T-\tau, T]$ such that 
	\begin{equation} \label{eq: tubular.1}
	\|y_T(\tau_i)-\overline{y}\|\leqslant\frac{\|y_T-\overline{y}\|_{L^2(0,T; \R^d)}}{\sqrt{\tau}}\stackrel{\eqref{eq: c1.def}}{\leqslant} \frac{C_1}{\sqrt{\tau}}.
	\end{equation}
	Note that, by virtue of the choice of $\tau$ in \eqref{eq: buffer.tau}, we have that $\frac{C_1}{\sqrt{\tau}}\leqslant r$ and thus 
	\begin{equation} \label{eq: tubular.2}
	\|y_T(\tau_i)-\overline{y}\| \leqslant r
	\end{equation}
	also holds.
	We shall now restrict our analysis onto $[\tau_1, \tau_2]$, and extrapolate onto the subset $[\tau, T-\tau]$. 
	First note that $u_T|_{[\tau_1, \tau_2]}$ is a global minimizer\footnote{This can be shown by contradiction.} of $J_{\tau_1, \tau_2}$ defined in \eqref{eq: J_T-aux} with fixed endpoints $y^{\tau_1}=y_T(\tau_1)$ and $y^{\tau_2}=y_T(\tau_2)$, and thus clearly $y_T|_{[\tau_1,\tau_2]}$ solves \eqref{eq: system.tau1.tau2}. 
	As 
	\begin{equation*}
	\tau_2 - \tau_1 \geqslant T-2\tau\geqslant 2T_0,
	\end{equation*}
	in view of \eqref{eq: tubular.2}, we may use \eqref{eq: c2.def} to find that  
	\begin{equation} \label{eq: aux.step.n=1}
	\|y_T(t) - \overline{y}\|\leqslant C_2\Big(\|y_T(\tau_1)-\overline{y}\| + \|y_T(\tau_2)-\overline{y}\|\Big)
	\end{equation}
	holds for all $t\in[\tau_1, \tau_2]$. 
	Setting 
	\begin{equation*}
	\kappa:= \max\left\{1, \frac{C_1}{C_2}\right\},
	\end{equation*}
	and applying \eqref{eq: tubular.1} to inequality \eqref{eq: aux.step.n=1}, we deduce that
	\begin{align} 
	\|y_T(t) - \overline{y}\|&\leqslant\frac{2C_1C_2}{\sqrt{\tau}}\label{eq: n=0.5}\\
	&\leqslant\frac{\kappa}{2}\frac{4C_2^2}{\sqrt{\tau}} \label{eq: n=1}
	\end{align}
	holds for all $t \in [\tau_1,\tau_2]$. As $\tau_1\leqslant\tau$ and $T-\tau\leqslant \tau_2$, estimates \eqref{eq: n=0.5} and \eqref{eq: n=1} clearly hold for all $t\in[\tau, T-\tau]$.
	\smallskip
	
	\noindent
	\emph{Step 2): Bootstrap.}
	Inequality \eqref{eq: n=1} motivates performing a bootstrap -- we will show that for any $n\in\N$ satisfying
	\begin{equation*} \label{eq: choice.n}
	n \leqslant \frac{1}{\tau}\left(\frac{T}{2}-T_0\right),
	\end{equation*}
	one has
	\begin{equation} \label{eq: induction.n}
	\sup_{t\in[n\tau, T-n\tau]}\|y_T(t) - \overline{y}\|\leqslant\frac{\kappa}{2}\left(\frac{4C_2^2}{\sqrt{\tau}}\right)^n.
	\end{equation}
	The choice of $n$ is done as to guarantee that $T-2n\tau\geqslant 2T_0$ in view of a repeated application of \Cref{lem: unif time estimates} (namely \eqref{eq: c2.def}).
	Note that \eqref{eq: n=0.5}, combined with the choice of $\tau$ in \eqref{eq: buffer.tau}, also implies that
	\begin{equation} \label{eq: tubular.3}
	\|y_T(t) - \overline{y}\|\leqslant r
	\end{equation}
	for all $t\in[\tau, T-\tau]$. To prove \eqref{eq: induction.n}, we proceed by induction. 
	The case $n=1$ clearly holds by \eqref{eq: n=1}. 
	Thus, assume that \eqref{eq: induction.n} holds -- we aim to show that \eqref{eq: induction.n} holds at step $n+1$.
	To this end, suppose that
	\begin{equation*}
	n+1 \leqslant \frac{1}{\tau}\left(\frac{T}{2}-T_0\right).
	\end{equation*}
	This is equivalent to $T-2n\tau-2T_0\geqslant2\tau$ (and recall that $\tau>0$ is fixed), and it also clearly implies that
	\begin{equation} \label{eq: ntau.lemma33}
	\tau \leqslant \frac{T-2n\tau}{2}.
	\end{equation}
	Since $T-2n\tau\geqslant 2T_0$, as in Step 1, it can be shown that $u_T|_{[n\tau, T-n\tau]}$ is a global minimizer of $J_{n\tau, T-n\tau}$ defined in \eqref{eq: J_T-aux}.
	Taking these facts into account, and noting that \eqref{eq: tubular.3} holds\footnote{Note that $n\tau\geqslant \tau$ and $T-n\tau \leqslant T-\tau$, so \eqref{eq: tubular.3} also holds for $t \in [n\tau, T-n\tau]$, hence \Cref{lem: unif time estimates} is applicable.}, we can apply \Cref{lem: D1} on $[n\tau, T-n\tau]$ (noting \eqref{eq: ntau.lemma33}), and \Cref{lem: unif time estimates} with $\tau_1=n\tau$ and $\tau_2 = T-n\tau$, to deduce that there exist a couple of times $t_1\in[n\tau, (n+1)\tau)$ and $t_2 \in(T-(n+1)\tau, T-n\tau]$ such that
	\begin{align*}
	\|y_T(t_i) - \overline{y}\|&\leqslant\frac{\|y_T-\overline{y}\|_{L^2(n\tau,T-n\tau; \R^d)}}{\sqrt{\tau}}\leqslant \frac{C_2}{\sqrt{\tau}} \Big(\|y_T(n\tau)-\overline{y}\| + \left\|y_T(T-n\tau)-\overline{y}\right\|\Big).
	\end{align*}
	We now use the induction hypothesis \eqref{eq: induction.n} in the above inequality to obtain
	\begin{equation} \label{eq: estimate.endpoints}
	\|y_T(t_i) - \overline{y}\|\leqslant\kappa\frac{C_2}{\sqrt{\tau}} \left(\frac{4C_2^2}{\sqrt{\tau}}\right)^n 
	\end{equation}
	Now since 
	\begin{equation*}
	t_2-t_1 \geqslant T-2(n+1)\tau \geqslant 2T_0,
	\end{equation*}
	and since $u_T|_{[t_1,t_2]}$ is a global minimizer of $J_{t_1, t_2}$ defined in \eqref{eq: J_T-aux}, combining \Cref{lem: unif time estimates}\footnote{May be applied once again since \eqref{eq: tubular.3} holds for $t=t_1\geqslant \tau$ and $t=t_2 \leqslant T-\tau$.} and \eqref{eq: estimate.endpoints} we are led to deduce that
	\begin{align} \label{eq: last.bootstrap}
	\|y_T(t) - \overline{y}\| &\leqslant C_2\Big(\|y_T(t_1)-\overline{y}\| + \|y_T(t_2) - \overline{y}\|\Big) \nonumber \\
	&\leqslant \frac{\kappa}{2} \frac{4C_2^2}{\sqrt{\tau}} \left(\frac{4C_2^2}{\sqrt{\tau}}\right)^n
	\end{align}
	for $t \in [t_1, t_2]$. Since $t_1 < (n+1)\tau$  and $T-(n+1)\tau<t_2$, estimate
	\eqref{eq: last.bootstrap} clearly also holds for $t\in[(n+1)\tau, T-(n+1)\tau]$.
	Identity \eqref{eq: induction.n} is thus proven.
	\smallskip 

	 \noindent
	 \emph{Step 3): Conclusion.}
	We now look to use \eqref{eq: induction.n} as to conclude the proof.
	Suppose that $t\in[\tau+T_0, T-(\tau+T_0)]$. We set
	\begin{equation*}
	n(t):=\min\left\{\left\lfloor\frac{t}{\tau+T_0} \right\rfloor, \left\lfloor \frac{T-t}{\tau+T_0} \right\rfloor\right\},
	\end{equation*}
	where $\lfloor z\rfloor$ denotes the integer part of $z \in \R$.
	Clearly $n(t) \geqslant 1$ and
	\begin{equation*}
	n(t)\tau \leqslant t \leqslant T-n(t)\tau.
	\end{equation*} 
	Moreover, since $z \mapsto \frac{z-2T_0}{z}$ is non-decreasing,
	\begin{align*}
	n(t) \leqslant \frac{T}{2(\tau+T_0)}=\frac{T}{2\tau} \frac{2(\tau+T_0)-2T_0}{2(\tau+T_0)}&\leqslant \frac{T}{2\tau} \frac{T-2T_0}{T} = \frac{1}{\tau}\left(\frac{T}{2}-T_0\right). 
	\end{align*}
	We may then apply \eqref{eq: induction.n} to obtain 
	\begin{equation} \label{eq: deduced.eq}
	\|y_T(t) - \overline{y}\|\leqslant\frac{\kappa}{2}\left(\frac{4C_2^2}{\sqrt{\tau}}\right)^{n(t)}. 
	\end{equation}
	By virtue of the choice of $\tau$ in \eqref{eq: buffer.tau}, we see that 
	\begin{equation*}
	\frac{4C_2^2}{\sqrt{\tau}}<1.
	\end{equation*}
	Moreover, since either 
	\begin{equation*}
	n(t) \geqslant\frac{t}{\tau+T_0}-1 \hspace{0.5cm} \text{ or } \hspace{0.5cm} n(t) \geqslant \frac{T-t}{\tau+T_0}-1,
	\end{equation*}
	we may rewrite \eqref{eq: deduced.eq} to obtain 
	\begin{align} \label{eq: hard.turnpike}
	\|y_T(t)-\overline{y}\|&\leqslant \frac{\kappa}{2}\exp\left(-n(t)\log\left(\frac{\sqrt{\tau}}{4C_2^2}\right)\right) \nonumber \\
	&\leqslant \frac{\kappa}{2}\frac{\sqrt{\tau}}{4C_2^2}\left(\exp\left(-\frac{\log\left(\frac{\sqrt{\tau}}{4C_2^2}\right)}{\tau+T_0}\, t\right)+\exp\left(-\frac{\log\left(\frac{\sqrt{\tau}}{4C_2^2}\right)}{\tau+T_0}\, (T-t)\right)\right).
	\end{align}
	Looking at \eqref{eq: hard.turnpike}, we see that estimate \eqref{eq: turnpike.2} thus holds for all $t\in[\tau+T_0, T-(\tau+T_0)]$, with 
	\begin{equation*}
	C:=\frac{\kappa\sqrt{\tau}}{8C_2^2}>0,
	\end{equation*}
	and 
	\begin{equation} \label{eq: mu}
	\mu:= \frac{\log\left(\frac{\sqrt{\tau}}{4C_2^2}\right)}{\tau+T_0}>0.
	\end{equation}
	By virtue of \eqref{eq: 0-tau+T}, \eqref{eq: T-tau+T0} and \eqref{eq: hard.turnpike}, we deduce that \eqref{eq: turnpike.2} holds for all $t \in [0, T]$, with $T^*:=2(\tau+T_0)$,  
	\begin{equation} \label{eq: turnpike.constant.C}
	C:= \max\left\{C_1,\frac{\kappa\sqrt{\tau}}{8C_2^2}\right\}>0,
	\end{equation}
	and $\mu>0$ as in \eqref{eq: mu}.
	This concludes the proof.
\end{proof}

    \subsection{Proof of \Cref{cor: control.decay}} \label{sec: proof.control.decay}
    
    We finish this section with the proof of \Cref{cor: control.decay}, which stipulates an exponential decay of optimal controls in the context of driftless control-affine systems, namely \eqref{eq: nonlinear.control.system} with a nonlinearity of the form
    \begin{equation} \label{eq: driftless}
    f(y,u) = \sum_{j=1}^m u_j f_j(y) \hspace{1cm} \text{ for } (y, u) \in \R^d\times \R^m.
    \end{equation} 
    We recall that $f_1, \ldots, f_m \in \Lip(\R^d; \R^d)$. We begin with the following  result.
    	
	\begin{lemma} \label{lem: scaling}
	Suppose $T_0>0$, $y^0 \in \R^d$ and $u_{T_0}\in L^2(0, T_0; \R^{m})$ are all given. Let $y_{T_0}\in C^0([0, T_0]; \R^{d})$ be the unique solution to 
	\begin{equation} \label{eq: 5.7}
	\begin{dcases}
	\dot{y}_{T_0} = f(y_{T_0}, u_{T_0}) &\text{ in } (0, T_0) \\
	y_{T_0}(0) = y^0
	\end{dcases}
	\end{equation}
	with $f$ as in \eqref{eq: driftless}.
	Let $T>0$, and define
	\begin{equation*} \label{eq: uT}
	u_T(t):= \frac{T_0}{T} u_{T_0}\left(t\frac{T_0}{T} \right)\hspace{1cm} \text{ for } t\in [0, T],
	\end{equation*}
	and 
	\begin{equation*} \label{eq: xT}
	y_T(t) := y_{T_0}\left(t\frac{T_0}{T}\right) \hspace{1cm} \text{ for } t \in [0, T].
	\end{equation*}
	Then $y_T\in C^0([0, T]; \R^{d})$ is the unique solution to \eqref{eq: nonlinear.control.system} with $y_T(0)=y^0$ and control $u_T$.
	\end{lemma}
	
	\noindent
	This sort of time-scaling in the context of driftless control affine systems is commonly used -- a canonical example is the proof of the Chow-Rashevskii theorem (\cite[Chapter 3, Section 3.3]{coron2007control}). We provide the short proof for completeness.
	
	\begin{proof}[Proof of \Cref{lem: scaling}]
	Using the fact that $y_{T_0}$ is the solution to \eqref{eq: 5.7} and the change of variable $\tau = s\frac{T}{T_0}$, we see that
	\begin{align*}
	y_T(t) :=  y_{T_0}\left(t\frac{T_0}{T}\right) &= y^0 +\int_0^{t\frac{T_0}{T}} f(y_{T_0}(s), u_{T_0}(s)) \diff s \nonumber\\
	&= y^0  + \int_0^t \frac{T_0}{T} f\left(y_{T_0}\left(\tau\frac{T_0}{T}\right), u_{T_0}\left(\tau\frac{T_0}{T}\right) \right)\diff \tau\nonumber\\
	&= y^0 + \int_0^t f\left(y_T(\tau), u_T(\tau)\right) \diff \tau.
	\end{align*}
	It follows that $y_T$ solves \eqref{eq: nonlinear.control.system} with $y_T(0)=y^0$, and we conclude by uniqueness.
	\end{proof}

    \begin{proof}[Proof of \Cref{cor: control.decay}]
	Fix any $t\in [0,T)$ and $0<h\ll1$ so that $t+2h\in [0,T]$, and set
    \begin{equation*}
        u^{\aux}(s) :=
        \begin{dcases}
       u_T(s) &\text{ for }s\in[0, t] \\
       \dfrac{1}{2}u_T\left(t+\dfrac{s-t}{2}\right) &\text{ for }s\in(t, t+2h] \\
       u_T(s-h) &\text{ for }s\in(t+2h, T]. 
        \end{dcases}
    \end{equation*}
    By \Cref{lem: scaling}, the state $y^\aux$, solution to \eqref{eq: nonlinear.control.system} associated to $u^{\aux}$ is precisely
    \begin{equation*}
    y^{\aux}(s) =
    \begin{dcases}
    y_T(s) &\text{ for }s\in[0, t] \\
    y_T\left(t+\dfrac{s-t}{2}\right) &\text{ for }s\in(t, t+2h] \\
    y_T(s-h) &\text{ for }s\in(t+2h, T].
    \end{dcases}
	\end{equation*}
	By means of simple changes of variables, and using the suboptimality of $u^\aux$, we can readily see that
	\begin{align}\label{cor: stabilisation_proof_eq60}
        J_T\left(u_T\right) \leqslant J_T\left(u^\aux\right) &= \int_{0}^{T}\left\|u^{\aux}(s)\right\|^2\diff s+\int_{0}^{T}\left\|y^{\aux}(s)-\overline{y}\right\|^2\diff s\nonumber\\
        &=\int_{0}^{T-h}\left\|u_T(s)\right\|^2\diff s-\dfrac{1}{2}\int_t^{t+h}\left\|u_T\left(s\right)\right\|^2\diff s\nonumber\\
        &\quad+\int_{0}^{T-h}\left\|y_{T}(s)-\overline{y}\right\|^2+\int_t^{t+h}\left\|y_T\left(s\right)-\overline{y}\right\|^2\diff s\nonumber\\
        &\leqslant\int_{0}^{T}\left\|u_T(s)\right\|^2\diff s-\frac{1}{2}\int_t^{t+h}\left\|u_T\left(s\right)\right\|^2\diff s \nonumber\\
        &\quad+\int_{0}^{T}\left\|y_{T}(s)-\overline{y}\right\|^2+\int_t^{t+h}\left\|y_T\left(s\right)-\overline{y}\right\|^2\diff s.
    \end{align}
   From \eqref{cor: stabilisation_proof_eq60}, one sees that
    \begin{equation} \label{eq: all.just.for.this.2}
        \frac{1}{2}\int_t^{t+h}\left\|u_T\left(s\right)\right\|^2\diff s\leqslant \int_t^{t+h}\left\|y_T\left(s\right)-\overline{y}\right\|^2\diff s.
    \end{equation}
    We combine \eqref{eq: all.just.for.this.2} with \eqref{eq: turnpike.4} to deduce that
    \begin{align} \label{eq: last.before.lebesgue.2}
    \frac{1}{h}\int_t^{t+h}\left\|u_T(s)\right\|^2\diff s\leqslant\frac{2}{h}\int_t^{t+h}\left\|y_T(s)-\overline{y}\right\|^2\diff s &\leqslant\frac{2C}{h}\int_t^{t+h}e^{-2\mu s}\diff s\nonumber\\
        &\leqslant \frac{2C}{h}\int_t^{t+h}e^{-2\mu t}\diff s\nonumber\\
        &=2C e^{-2\mu t}.
    \end{align} 
   By the Lebesgue differentiation theorem, using \eqref{eq: last.before.lebesgue.2} we deduce that
    \begin{equation*}
        \left\|u_T(t)\right\|=\lim_{h\searrow0}\left(\frac{1}{h}\int_t^{t+h}\left\|u_T(s)\right\|^2\diff s\right)^{\frac{1}{2}}\leqslant 2Ce^{-\mu t},
    \end{equation*}
    for a.e. $t\in(0,T)$, as desired.  This concludes the proof.
    \end{proof}
    
    	\section{Proof of \Cref{thm: turnpike.wave}} \label{sec: turnpike.wave.proof}
	
	\noindent
	In this section, we provide details of the proof of \Cref{thm: turnpike.wave}. 
	The proof of \Cref{cor: stabilisation.wave} follows by repeating the proof of \Cref{cor: stabilisation} in the appropriate functional setting, so we omit it.
	
	\begin{proof}[Proof of \Cref{thm: turnpike.wave}]
	Once \eqref{eq: nonlinear.wave.control} is written as a first order evolution equation set in $X:=H^1_0(\Omega)\times L^2(\Omega)$ (see the proof of \Cref{lem: C2.wave} for this setup), the only noticeable difference in the proof of \Cref{thm: turnpike.wave} with respect to the proof of \Cref{thm: turnpike} are the specific "quasi-turnpike" lemmas one applies in the preparation (\Cref{lem: quasiturnpike.wave.1} in Part 1 \& Step 1 of Part 2) and bootstrap (\Cref{lem: quasiturnpike.2.wave} in Step 2). 
	So one simply repeats the proof of \Cref{thm: turnpike} whilst applying \Cref{lem: quasiturnpike.wave.1}, \Cref{lem: quasiturnpike.2.wave} and \Cref{lem: D1} with $X$ as above. 
	Whence, the proof follows from these two lemmas, stated and proven just below.
	\end{proof}
	
	\begin{lemma} \label{lem: quasiturnpike.wave.1}
	Let $\*y^0 = (y^0_1, y^0_2) \in H^1_0(\Omega) \times L^2(\Omega)$ be given.
	Let $T>0$ be fixed, and let $u_T \in L^2((0,T)\times\omega)$ be a 
	global minimizer to $J_T$ defined in \eqref{eq: J_T.wave}, with $y_T$ denoting the associated solution to \eqref{eq: nonlinear.wave.control}.
	Then, there exists a constant $C = C(f, \phi, \omega, \Omega,\overline{y}, \*y^0)>0$ independent of $T>0$ such that 
	\begin{equation*} \label{eq: quasiturnpike.wave.1}
	J_T(u_T) + \left\|y_T(t)-\overline{y}\right\|_{H^1_0(\Omega)}^2 + \|\del_t y_T(t)\|_{L^2(\Omega)}^2 \leqslant C
	\end{equation*}
	holds for all $t \in [0, T]$. 
	\end{lemma}

	\begin{proof}[Proof of \Cref{lem: quasiturnpike.wave.1}]
	The proof follows the lines of that of \Cref{lem: quasiturnpike.1}, simply adapted to the PDE setting. Fix $T_0:=T_{\min}(\omega,\Omega)+1$ where $T_{\min}(\omega, \Omega)>0$ is the minimal controllability time for the semilinear wave equation, defined in \eqref{eq: Tmin}.
	\smallskip
	
	\noindent
	\textit{Case 1).}
	We begin by considering the case $T>T_0$. By controllability, we know that exists some control $u^{\dagger}\in L^2((0, T_0)\times\omega)$ such that the corresponding solution $y^{\dagger}$ to
	\begin{equation*}
	\begin{dcases}
	\del_t^2 y^\dagger - \Delta y^\dagger + f(y^\dagger) = u^\dagger\one_\omega &\text{ in }(0,T_0)\times\Omega\\
	y^\dagger = 0 &\text{ on }(0,T_0)\times\del\Omega\\
	(y^\dagger, \del_t y^\dagger)|_{t=0} = \*y^0 &\text{ in }\Omega.
	\end{dcases}
	\end{equation*}
	satisfies $y^{\dagger}(T_0) = \overline{y}$ and $\del_t y^\dagger(T_0) = 0$ (in $L^2(\Omega)$, and thus a.e.).
	Now set
	\begin{equation*}
	u^\aux(t):= 
	\begin{dcases}
	u^{\dagger}(t) &\text{ in }(0, T_0) \\
	0 &\text{ in }(T_0, T)
	\end{dcases}
	\end{equation*}
	and let $y^\aux$ be the corresponding solution to \eqref{eq: nonlinear.wave.control}.
	Clearly 
	\begin{equation*}
	y^\aux(t) = \overline{y} \hspace{0.25cm} \text{ and } \hspace{0.25cm} \del_t y^\aux(t) = 0 \hspace{0.5cm} \text{ for } t \in [T_0, T].
	\end{equation*}
	Combining this fact with $J_T(u_T) \leqslant J_T(u^\aux)$, we see that
	\begin{align*}
	J_T(u_T) \leqslant \phi(\overline{y}) + \left\|y^\dagger-\overline{y}\right\|_{L^2(0,T_0;H^1_0(\Omega))}^2 + \left\|\del_t y^\dagger\right\|_{L^2((0,T_0)\times\Omega)}^2+ \left\|u^\dagger\right\|_{L^2((0, T_0)\times\omega)}^2.
	\end{align*}
	As the right-hand side in the above inequality is clearly independent of $T$, we conclude by applying \Cref{lem: C2.wave}. 
	\smallskip
	
	\noindent
	\textit{Case 2).} 
        Now suppose that $T\leqslant T_0$. We use $J_T(u_T) \leqslant J_T(u_{T_0})$ to obtain
        \begin{align*} 
       J_T(u_T)\leqslant\phi\left(y_{T_0}(T)\right)+\left\|y_{T_0} - \overline{y}\right\|_{L^2(0,T;H^1_0(\Omega))}^2+\left\|\del_t y_{T_0}\right\|^2_{L^2((0,T)\times\Omega)} + \|u_{T_0}\|_{L^2((0,T)\times\omega)}^2.
        \end{align*} 
        Now $y_{T_0}\in C^0([0,T_{0}]; L^2(\Omega))$ is bounded uniformly with respect to $T\in[0,T_0]$.
        Hence, using the fact that $\phi \in \mathscr{L}(L^2(\Omega); \R_+)$ and $T\leqslant T_0$, we deduce that
        \begin{equation} \label{eq: TleqT0.wave}
        J_T(u_T)\leqslant C
        \end{equation}
        for some $C>0$ independent of $T$.
        Combining \eqref{eq: TleqT0.wave} with \Cref{lem: C2.wave} allows us to conclude.     
	\end{proof}
	
	\noindent
	Since \eqref{eq: nonlinear.wave.control} is a Lipschitz perturbation of an exactly controllable linear system, the following claim holds.
	
	\begin{claim}[Cost estimate] \label{def: ctrl.wave}
	Let $T_0> T_{\min}(\omega, \Omega)$, where $T_{\min}(\omega, \Omega)>0$ is defined in \eqref{eq: Tmin}.
	There exists $r>0$ and $C=C(T_0, \omega, f)>0$ such that 
	\begin{equation*} \label{eq: control.cost.estimate.1.wave}
	\inf_{\substack{u \\ \text{ such that } \\ (y, \del_t y)|_{t=0} = \*y^0 \\ \text{ and } \\ (y, \del_t y)|_{t=T_0} = (\overline{y}, 0)} } \|u\|_{L^2((0,T_0)\times\omega)}^2\leqslant C\left( \left\|y^0_1-\overline{y}\right\|^2_{H^1_0(\Omega)} + \left\|y^0_2\right\|_{L^2(\Omega)}^2\right),
	\end{equation*}
	and
	\begin{equation*} \label{eq: control.cost.estimate.2.wave}
	\inf_{\substack{u \\ \text{ such that } \\ (y, \del_t y)|_{t=0} = (\overline{y}, 0) \\ \text{ and } \\  (y, \del_t y)|_{t=T_0} = \*y^1} } \|u\|_{L^2((0,T_0)\times\omega)}^2\leqslant C\left( \left\|y^1_1-\overline{y}\right\|^2_{H^1_0(\Omega)} + \left\|y^1_2\right\|^2_{L^2(\Omega)}\right),
	\end{equation*}
	hold for any $\*y^0 = \left(y_1^0, y_2^0\right)$ and $\*y^1 = \left(y_1^1, y_2^1\right)$ such that
	\begin{equation*}
	\*y^0, \*y^1 \in \left\{\begin{bmatrix}y_1 \\y_2 \end{bmatrix} \in H^1_0(\Omega)\times L^2(\Omega) \colon \left\|\begin{bmatrix}y_1\\y_2\end{bmatrix}-\begin{bmatrix}\overline{y}\\0\end{bmatrix}\right\|_{H^1_0(\Omega)\times L^2(\Omega)}\leqslant r\right\},
	\end{equation*}
	 where $y$ solves \eqref{eq: nonlinear.wave.control} and $\overline{y} \in H^1_0(\Omega)$ is fixed as in \eqref{eq: steady.state.heat}.
	\end{claim}
	
	\begin{remark}[Regarding \Cref{def: ctrl.wave}]\label{rem: ctrl.wave}
	Let us provide more detail regarding \Cref{def: ctrl.wave}, following \cite{zhang2004exact, zuazua1991exact}, and also the proofs of \cite[Theorem 2.1]{zhang2001exact} and \cite[Theorem 2.2]{fu2007exact}. For showing the controllability of the semilinear wave equation, one typically proceeds by considering
	\begin{equation*}
	\partial_t^2 y-\Delta y + g(\zeta)y = f(0) \hspace{1cm} \text{ in } \Omega\times(0,T),
	\end{equation*}
	where $\zeta\in L^2((0,T)\times\Omega)$ is given, and $g(s)=\frac{f(s)-f(0)}{s}$ for $s\in\mathbb{R}$ is bounded and continuous. It can be shown that the above system is controllable in time $T> T_{\min}(\omega,\Omega)$ with continuous dependence of the minimal $L^2$-norm control with respect to the data and $f(0)$. The result is transferred to the semilinear system by Schauder's fixed point theorem. 
	To have precisely the same estimates as in \Cref{def: ctrl.wave}, namely, to remove the dependence of the minimal $L^2$-norm control with respect to $f(0)$,  we assume that $f(0)=0$. 
	\end{remark}
	
	\noindent
	As in the finite-dimensional case, the second "quasi-turnpike" result is one for an auxiliary control problem with fixed endpoints. 	
	For $0\leqslant \tau_1<\tau_2\leqslant T$ and given $\*y^{\tau_1}, \*y^{\tau_2}\in H^1_0(\Omega)\times L^2(\Omega)$, this auxiliary problem consists in minimizing the nonnegative functional
	\begin{equation} \label{eq: J_T-aux.wave}
	J_{\tau_1, \tau_2}(u) := \int_{\tau_1}^{\tau_2} \|y(t)-\overline{y}\|^2_{H^1_0(\Omega)}\diff t + \int_{\tau_1}^{\tau_2} \|\del_t y(t)\|_{L^2(\Omega)}^2 + \int_{\tau_1}^{\tau_2} \|u(t)\|^2_{L^2(\omega)} \diff t
	\end{equation}
	over all $u \in U_\ad$, where $y \in C^0([\tau_1,\tau_2]; H^1_0(\Omega))\cap C^1([\tau_1, \tau_2]; L^2(\Omega))$ denotes the unique solution to
	\begin{equation} \label{eq: system.tau1.tau2.wave}
	\begin{dcases}
	\del_t^2 y- \Delta y + f(y) = u\one_\omega &\text{ in }(\tau_1, \tau_2)\times\Omega\\
	y = 0 &\text{ on }(\tau_1, \tau_2)\times\del\Omega \\
	(y, \del_t y)|_{t=\tau_1}=\*y^{\tau_1} &\text{ in }\Omega.
	\end{dcases}
	\end{equation}
	and where
	\begin{equation*}
	U_\ad := \Big\{ u \in L^2((\tau_1, \tau_2)\times\omega) \colon (y, \del_t y)|_{t=\tau_2} = \*y^{\tau_2}\Big\}.
	\end{equation*} 
	We recall that $f \in \Lip(\R)$.
	We now state and prove the wave equation analog of \Cref{lem: unif time estimates}, which we recall, is the cornerstone of the bootstrap argument in our turnpike proof.
	
	\begin{lemma} \label{lem: quasiturnpike.2.wave} 
	Fix $T_0>T_{\min}(\omega,\Omega)$. 
	Suppose $T>0$ and $0\leqslant\tau_1<\tau_2\leqslant T$ are fixed such that $\tau_2-\tau_1\geqslant2T_0$. Let $r>0$ be as in \Cref{def: ctrl.wave}, and let $\*y^{\tau_1},\*y^{\tau_2}$ be such that 
	\begin{equation*}
	\*y^{\tau_i}\in\left\{\begin{bmatrix}y_1\\y_2\end{bmatrix} \in H^1_0(\Omega)\times L^2(\Omega) \, \colon \, \left\|\begin{bmatrix}y_1\\y_2\end{bmatrix}-\begin{bmatrix}\overline{y} \\0\end{bmatrix}\right\|_{H^1_0(\Omega)\times L^2(\Omega)}\leqslant r\right\}
	\end{equation*}	
	for $i=1, 2$.
	Let $u_T \in U_\ad$ be a global minimizer to $J_{\tau_1, \tau_2}$ defined in \eqref{eq: J_T-aux.wave}, with $y_T$ denoting the associated solution to \eqref{eq: system.tau1.tau2.wave}. 
	Then, there exists $C = C(f,T_0,\Omega,\omega)>0$ independent of $T,\tau_1,\tau_2,\*y^{\tau_i}, r$, such that 	
	\begin{align*}
	J_{\tau_1, \tau_2}(u_T) &+ \left\|y_T(t)-\overline{y}\right\|^2_{H^1_0(\Omega)} + \|\del_t y_T(t)\|^2_{L^2(\Omega)}  \\
	&\leqslant C \left(\left\|y^{\tau_1}_1-\overline{y}\right\|^2_{H^1_0(\Omega)}+\left\|y^{\tau_1}_2\right\|_{L^2(\Omega)}^2+\left\|y^{\tau_2}_1-\overline{y}\right\|^2_{L^2(\Omega)}+\left\|y^{\tau_2}_2\right\|_{L^2(\Omega)}^2\right)
	\end{align*}
	holds for all $t \in [\tau_1, \tau_2]$. 
	\end{lemma}
	
	\begin{proof}[Proof of \Cref{lem: quasiturnpike.2.wave}]
	
	The proof follows the lines of that of \Cref{lem: unif time estimates}, with some slight technical differences. We provide details for the sake of completeness.
	For notational purposes, it will be significantly simpler to operate in the canonical first order system framework presented in the proof of \Cref{lem: C2.wave}.
	For the same reason, we will also drop the subscripts of $T$.
	We set $X:=H^1_0(\Omega)\times L^2(\Omega)$, and we denote 
	\begin{equation*}
	\*y:= \begin{bmatrix} y \\ \del_t y \end{bmatrix}, \hspace{1cm} \overline{\*y} := \begin{bmatrix} \overline{y} \\ 0 \end{bmatrix}.
	\end{equation*}
	We also recall the definition of the skew-adjoint operator 
	\begin{equation*}
	A:= \begin{bmatrix} 0 & \text{Id} \\ \Delta & 0 \end{bmatrix}, \hspace{1cm} D(A) = D(\Delta)\times H^1_0(\Omega),
	\end{equation*}
	where $D(\Delta) = H^2(\Omega)\cap H^1_0(\Omega)$.
	Then the desired estimate simply writes as
	\begin{align*}
	J_{\tau_1, \tau_2}(u) &+ \left\|\*y(t)-\overline{\*y}\right\|^2_{X} \leqslant C \left(\left\|\*y^{\tau_1}-\overline{\*y}\right\|^2_{X}+\left\|\*y^{\tau_2}-\overline{\*y}\right\|^2_{X}\right)
	\end{align*}
	for all $t \in [\tau_1, \tau_2]$. We proceed similarly as in the proof of \Cref{lem: unif time estimates}. Using \Cref{def: ctrl.wave}, we know the following.
	\begin{itemize}
	\item There exists a control $u^{\dagger} \in L^2((\tau_1, \tau_1+T_0)\times\omega)$ satisfying
	\begin{equation} \label{eq: first.control.estimate.wave}
	\left\|u^{\dagger}\right\|_{L^2((\tau_1, \tau_1+T_0)\times\omega)}^2 \leqslant C_0 \left\|\*y^{\tau_1}-\overline{\*y}\right\|_{X}^2,
	\end{equation}
	for some $C_0=C_0(T_0,\omega,f)>0$, and such that the corresponding solution $\*y^{\dagger} = \begin{bmatrix}y^\dagger\\\del_t y^\dagger\end{bmatrix}$ to
	\begin{equation*}\label{PC_xdagger.wave}
	\begin{dcases}
	\del_t \*y^\dagger - A\*y^\dagger + \begin{bmatrix} 0 \\ f(y^\dagger)\end{bmatrix} = \begin{bmatrix} 0 \\ u^\dagger\one_\omega\end{bmatrix} &\text{ in } (\tau_1, \tau_1+T_0)\\
	\*y^\dagger|_{t=\tau_1}=\*y^{\tau_1}
	\end{dcases}
	\end{equation*}
	satisfies $\*y^{\dagger}(\tau_1+T_0) = \overline{\*y}$ in $X$. 
	By writing the Duhamel formula for $\*y^\dagger-\overline{\*y}$, and using the conservative character of $\left\{e^{tA}\right\}_{t>0}$ in $X$, the Cauchy-Schwarz inequality, the Lipschitz character of $f$ and the Poincaré inequality, we see that
	\begin{align} \label{eq: aux.quasiturnpike.est.1.wave}
	\left\|\*y^\dagger(t) - \overline{\*y} \right\|_X &\leqslant \left\|e^{tA}(\*y^{\tau_1} -\overline{\*y}) \right\|_{X} + \int_{\tau_1}^t \left\|e^{(t-s)A}\begin{bmatrix}0\\u^\dagger(s)\one_\omega\end{bmatrix}\right\|_{X} \diff s \nonumber\\
	&\quad + \int_{\tau_1}^t \left\|e^{(t-s)A}\begin{bmatrix}0\\\left(f\big(y^\dagger\big)-f(\overline{y})\right)\end{bmatrix}\right\|_{X}\diff s \nonumber\\
	&\leqslant \left\|\*y^{\tau_1} -\overline{\*y} \right\|_{X} + \sqrt{T_0}\left\|u^\dagger\right\|_{L^2((\tau_1, \tau_1+T_0)\times\omega)} \nonumber \\
	&\quad + C(f, \Omega) \int_{\tau_1}^t \left\|\*y^\dagger(s)-\overline{\*y}\right\|_{X}\diff s,
	\end{align}
	with $C(f, \Omega)>0$ depending solely on the Poincaré constant and the Lipschitz constant of $f$.
	Applying the Gr\"onwall inequality to \eqref{eq: aux.quasiturnpike.est.1.wave} and using \eqref{eq: first.control.estimate.wave}, we deduce that
	\begin{align} \label{eq: first.state.estimate.wave}
	\left\|\*y^\dagger(t) - \overline{\*y}\right\|_{X}\leqslant C_1\exp\left(C(f,\Omega)T_0\right)\left\|\*y^{\tau_1} - \overline{\*y}\right\|_{X}
	\end{align}
	holds for some $C_1(f,T_0, \omega)>0$ independent of $T, \tau_1, \tau_2>0$, and for every $t \in (\tau_1, \tau_1+T_0)$.
	
	\item There exists a control  $u^\ddagger \in L^2((\tau_1, \tau_1+T_0)\times\omega)$ satisfying
	\begin{equation} \label{eq: second.control.estimate.wave}
	\left\|u^{\ddagger}\right\|_{L^2((\tau_1, \tau_1+T_0)\times\omega)}^2 \leqslant C_0\left\|\overline{\*y}-\*y^{\tau_2}\right\|^2_{X},
	\end{equation}
	and which is such that the corresponding solution $\*y^\ddagger=\begin{bmatrix}y^\ddagger\\ \del_t y^\ddagger\end{bmatrix}$ to
	\begin{equation*} \label{eq: xddagger.wave}
	\begin{dcases}
	\del_t \*y^\ddagger - A \*y^\ddagger + \begin{bmatrix} 0 \\ f(y^\ddagger)\end{bmatrix} = \begin{bmatrix} 0 \\ u^\ddagger\one_\omega \end{bmatrix} &\text{ in } (\tau_1, \tau_1+T_0) \\
	\*y^\ddagger|_{t=\tau_1} = \overline{\*y}
	\end{dcases}
	\end{equation*}
	satisfies $\*y^\ddagger(\tau_1+T_0) = \*y^{\tau_2}$ in $X$. 
	Arguing just as above, we see that
	\begin{align} \label{eq: aux.quasiturnpike.est.2.wave}
	\left\|\*y^\ddagger(t) - \overline{\*y} \right\|_X &\leqslant \int_{\tau_1}^t \left\|e^{(t-s)A}\begin{bmatrix}0\\u^\ddagger(s)\one_\omega\end{bmatrix}\right\|_{X} \diff s + \int_{\tau_1}^t \left\|e^{(t-s)A}\begin{bmatrix}0\\\left(f\big(y^\ddagger\big)-f(\overline{y})\right)\end{bmatrix}\right\|_{X}\diff s \nonumber\\
	&\leqslant \sqrt{T_0}\left\|u^\ddagger\right\|_{L^2((\tau_1, \tau_1+T_0)\times \omega)}+ C(f,\Omega) \int_{\tau_1}^t \left\|\*y^\ddagger(s)-\overline{\*y}\right\|_{X}\diff s,
	\end{align}
	with $C(f, \Omega)>0$ depending solely on the Poincaré constant and the Lipschitz constant of $f$.
	Applying the Gr\"onwall inequality to \eqref{eq: aux.quasiturnpike.est.2.wave} and using \eqref{eq: second.control.estimate.wave}, we deduce that
	\begin{align} \label{eq: second.state.estimate.wave}
	\left\|\*y^\ddagger(t) - \overline{\*y}\right\|_{X}\leqslant C_2 \exp\left(C(f,\Omega)T_0\right)\left\|\*y^{\tau_2} - \overline{\*y}\right\|_{X}
	\end{align}
	holds for some $C_2(f,T_0,\omega)>0$ independent of $T, \tau_1, \tau_2>0$, and for every $t \in (\tau_1, \tau_1+T_0)$.
	\end{itemize}
	Now set
	\begin{equation*}
	u^\aux(t):= 
	\begin{dcases}
	u^{\dagger}(t) &\text{ in } (\tau_1, \tau_1+T_0) \\
	0 &\text{ in } (\tau_1+T_0, \tau_2-T_0) \\
	u^\ddagger\left(t-(\tau_2-\tau_1-T_0)\right) &\text{ in } (\tau_2-T_0, \tau_2),
	\end{dcases}
	\end{equation*}
	and let $\*y^\aux=\begin{bmatrix}y^\aux\\ \del_t y^\aux\end{bmatrix}$ be the corresponding solution to \eqref{eq: system.tau1.tau2.wave}.
	By construction, we have 
	\begin{equation*}
	\*y^\aux(t) = \*y^\dagger(t) \hspace{1cm} \text{ in } [\tau_1,\tau_1+T_0],
	\end{equation*}
	and thus
	\begin{equation} \label{eq: aux.eq.1.wave}
	\*y^\aux(t) = \overline{\*y} \hspace{1cm} \text{ in } [\tau_1+T_0, \tau_2-T_0], 
	\end{equation}
	whereas we also have $\*y^\aux(\tau_2) = \*y^{\tau_2}$, whence $u^\aux \in U_\ad$.
	We now evaluate $J_{\tau_1,\tau_2}$ at $u^\aux$, which by virtue of a simple change of variable as well as \eqref{eq: aux.eq.1.wave}, \eqref{eq: first.control.estimate.wave}, \eqref{eq: first.state.estimate.wave}, \eqref{eq: second.control.estimate.wave} and \eqref{eq: second.state.estimate.wave}, leads us to
	\begin{align} \label{eq: aux.estimate.1.wave}
	J_{\tau_1, \tau_2}(u^\aux) &= \left\| u^\dagger \right\|_{L^2((\tau_1,\tau_1+T_0)\times\omega)} + \left\| u^\ddagger \right\|_{L^2((\tau_1,\tau_1+T_0)\times\omega)} \nonumber \\
	&\quad+ \int_{\tau_1}^{\tau_1+T_0} \left\|\*y^\dagger(t)-\overline{\*y}\right\|^2_{X} \diff t + \int_{\tau_1}^{\tau_1+T_0} \left\|\*y^\ddagger(t)-\overline{\*y}\right\|^2_{X}\diff t \nonumber \\
	&\leqslant C_{3} \Big(\left\|\overline{\*y}-\*y^{\tau_1}\right\|^2_{X}+\left\|\overline{\*y}-\*y^{\tau_2}\right\|^2_{X}\Big)
	\end{align}
	where $C_{3}(f,T_0,\Omega,\omega)>0$ is independent of $T, \tau_1, \tau_2>0$.
	By virtue of the optimality of $u$ and \eqref{eq: aux.estimate.1.wave}, we have
	\begin{align*}
	J_{\tau_1,\tau_2}\left(u\right) &\leqslant J_{\tau_1, \tau_2}\left(u^\aux\right) \leqslant C_{3} \Big(\left\|\overline{\*y}-\*y^{\tau_1}\right\|^2_{X}+\left\|\overline{\*y}-\*y^{\tau_2}\right\|^2_{X}\Big).
	\end{align*} 
	An application of \Cref{lem: C2.wave} suffices to conclude.
	\end{proof}

	\section{Proof of \Cref{thm: turnpike.heat}} \label{sec: turnpike.heat.proof}
		
	\noindent
	For the semilinear heat equation, we can adapt the proof strategy of \Cref{thm: turnpike} to directly prove the stabilization result stipulated by \Cref{thm: turnpike.heat}. We provide details of the proof, as it is not an immediate application of that of \Cref{thm: turnpike}.
	We recall that since $f\in\Lip(\R)$, as presented in \cite[Lemma 8.3]{pighin2018controllability} (and the references therein), given any $T_0>0$, $y^0\in L^2(\Omega)$ and $\overline{y}\in H^1_0(\Omega)$ solution to \eqref{eq: steady.state.heat}, there exists a control $u\in L^2((0,T_0)\times\omega)$ such that the unique solution $y$ to \eqref{eq: nonlinear.heat.control} satisfies $y(T_0) = \overline{y}$, and
    \begin{equation} \label{eq: cost.control.heat.1}
    \|u\|_{L^2((0,T_0)\times\omega)}\leqslant C(T_0,\omega, f)\left\|y^0 - \overline{y}\right\|_{L^2(\Omega)}
    \end{equation}
    for some $C(T_0, \omega, f)>0$ (the dependence on $f$ is through the Lipschitz constant which is an upper bound for the potential appearing in the associated linear problem).
    Indeed, we may consider $z:=y-\overline{y}$, and the control $u$ steering $z$ to $0$ in time $T$ is the same as that steering $y$ to $\overline{y}$ in time $T$.
    But then, $\|u\|_{L^2((0,T_0)\times\omega)}\leqslant C(T_0, \omega, f)\|z(0)\|_{L^2(\Omega)}$ from the linear system and a fixed-point argument. 
	Let $T>0$ and $0 \leqslant \tau_1 < T$ be fixed, and suppose $y^{\tau_1}\in L^2(\Omega)$ is given. Consider
	\begin{equation} \label{eq: J_T-aux.heat}
	J_{\tau_1, T}(u) := \int_{\tau_1}^T \|y(t) - \overline{y}\|^2_{L^2(\Omega)}\diff t + \int_{\tau_1}^T \|u(t)\|^2_{L^2(\omega)} \diff t,
	\end{equation}
	where $y$ solves
	\begin{equation} \label{eq: heat.tau1}
	\begin{dcases}
	\del_t y - \Delta y + f(y) = u\one_\omega &\text{ in } (\tau_1, T) \times \Omega \\
	y = 0 &\text{ on } (\tau_1, T) \times \del \Omega \\
	y|_{t=\tau_1} = y^{\tau_1} &\text{ in } \Omega.
	\end{dcases}
	\end{equation}
	We will only need the following lemma, which is similar to \Cref{lem: quasiturnpike.2.wave}. In fact, the blueprint of the proof below is contained therein.
	
 	\begin{lemma} \label{lem: quasiturnpike.heat.1}
	Let $T>0$ and $\tau_1\geqslant0$ be given such that $T>\tau_1$, and let $y^{\tau_1}\in L^2(\Omega)$. 
	Let $u_T \in L^2((\tau_1,T)\times\omega)$ be any 
	global minimizer to $J_{\tau_1, T}$ defined in \eqref{eq: J_T-aux.heat}, with $y_T$ denoting the corresponding solution to \eqref{eq: heat.tau1}.
	Then, there exists a constant $C = C(f,\omega)>0$ independent of $T, \tau_1>0$ and $y^{\tau_1}$ such that 
	\begin{equation*} \label{eq: quasiturnpike.heat.1}
	J_{\tau_1, T}(u_T) + \left\|y_T(t)-\overline{y}\right\|_{L^2(\Omega)}^2 \leqslant C \left\|y^{\tau_1}-\overline{y}\right\|^2_{L^2(\Omega)}
	\end{equation*}
	holds for all $t\in[\tau_1,T]$. 
	\end{lemma}

	\begin{proof}[Proof of \Cref{lem: quasiturnpike.heat.1}]
	
	\textit{Case 1).} Let us first suppose that $T\geqslant\tau_1+1$. By controllability to the steady state $\overline{y}$ (see the discussion around \eqref{eq: cost.control.heat.1}), we know that exists a control $u^{\dagger}\in L^2((\tau_1,\tau_1+1)\times\omega)$ satisfying 
        \begin{equation} \label{eq: first.control.estimate.heat}
        \left\|u^\dagger\right\|_{L^2((\tau_1,\tau_1+1)\times\omega)}\leqslant C_1\left\|y^{\tau_1}-\overline{y}\right\|_{L^2(\Omega)}
        \end{equation}
        for some $C_1=C_1(\omega,f)>0$ and such that the corresponding solution $y^{\dagger}$ to
	\begin{equation*}
	\begin{dcases}
	\del_t y^\dagger - \Delta y^\dagger + f(y^\dagger) = u^\dagger\one_\omega &\text{ in }(\tau_1, \tau_1+1)\times\Omega\\
	y^\dagger=0 &\text{ on }(\tau_1,\tau_1+1)\times\del\Omega \\
	y^\dagger|_{t=0}=y^0 &\text{ in }\Omega.
	\end{dcases}
	\end{equation*}
	satisfies $y^{\dagger}(\tau_1+1)=\overline{y}$.
	Arguing as in the proof of \Cref{lem: C2}, we see that
	\begin{align} \label{eq: aux.quasiturnpike.est.1.heat}
	\left\|y^\dagger(t) - \overline{y} \right\|_{L^2(\Omega)} 
	&\leqslant \left\|y^{\tau_1}-\overline{y}\right\|_{L^2(\Omega)}+\left\|u^\dagger\right\|_{L^2((\tau_1,\tau_1+1)\times\omega)}\nonumber \\
	&\quad + C(f)\int_{\tau_1}^t \left\|y^\dagger(s)-\overline{y}\right\|_{L^2(\Omega)}\diff s
	\end{align}
	for $t\in(\tau_1,\tau_1+1)$, with $C(f)>0$ being the Lipschitz constant of $f$.
	Applying the Gr\"onwall inequality to \eqref{eq: aux.quasiturnpike.est.1.heat} and using \eqref{eq: first.control.estimate.heat}, we deduce that
	\begin{align} \label{eq: first.state.estimate.heat}
	\left\|y^\dagger(t) - \overline{y}\right\|_{L^2(\Omega)}\leqslant C_2 \exp\left(C(f)\right)\left\|y^{\tau_1} - \overline{y}\right\|_{L^2(\Omega)}
	\end{align}
	for some $C_2(f,\omega)>0$ independent of $T, \tau_1$, and for every $t\in(\tau_1,\tau_1+1)$.
	Now set
	\begin{equation*}
	u^\aux(t):= 
	\begin{dcases}
	u^{\dagger}(t) &\text{ in } (\tau_1, \tau_1+1) \\
	0 &\text{ in } (\tau_1+1, T)
	\end{dcases}
	\end{equation*}
	and let $y^\aux$ be the corresponding solution to \eqref{eq: nonlinear.heat.control}.
	Clearly 
	\begin{equation*}
	y^\aux(t) = \overline{y} \hspace{1cm} \text{ for } t\in[\tau_1+1, T].
	\end{equation*}
	Hence, using $J_{\tau_1,T}(u_T) \leqslant J_{\tau_1,T}(u^\aux)$, \eqref{eq: first.state.estimate.heat} and \eqref{eq: first.control.estimate.heat}, we see that
	\begin{align*}
	J_{\tau_1,T}(u_T) &\leqslant \left\|y^\dagger-\overline{y}\right\|_{L^2((\tau_1,\tau_1+1)\times\Omega)}^2 + \left\|u^\dagger\right\|_{L^2((\tau_1, \tau_1+1)\times\omega)}^2 \\
	&\leqslant C_3 \left\|y^{\tau_1}-\overline{y}\right\|_{L^2(\Omega)}^2
	\end{align*}
	for some $C_3(f,\omega)>0$ independent of $T, \tau_1>0$. Applying \Cref{lem: C2} suffices to conclude.       
	\smallskip
	
	\noindent
	\textit{Case 2).}
	Now suppose that $\tau_1<T<\tau_1+1$. 
	We may then use the optimality inequality $J_{\tau_1, T}(u_T)\leqslant J_{\tau_1,T}(u_{\tau_1+1})$, as well as $J_{\tau_1,T}(u_{\tau_1+1})\leqslant J_{\tau_1,\tau_1+1}(u_{\tau_1+1})$, and since by the previous step, we know that 
	\begin{equation*}
	J_{\tau_1,\tau_1+1}(u_{\tau_1+1})\leqslant C_3\left\|y^{\tau_1} - \overline{y}\right\|_{L^2(\Omega)}^2,
	\end{equation*}
	where $C_3 = C_3(f,\omega)>0$ is independent of $T, \tau_1$, we deduce
	\begin{equation} \label{eq: TleqT0.heat}
	J_{\tau_1, T}(u_T)\leqslant C_3\left\|y^{\tau_1} - \overline{y}\right\|_{L^2(\Omega)}^2.
	\end{equation}
	We may conclude by combining \eqref{eq: TleqT0.heat} with \Cref{lem: C2}.
	\end{proof}
		
	\begin{proof}[Proof of \Cref{thm: turnpike.heat}]
	The proof is of the same spirit\footnote{Actually, as already commented below the statement of \Cref{cor: stabilisation}, the proof presented below also roughly applies to show \Cref{cor: stabilisation}, where one also would need to account for the constants which should also depend on the radius $r>0$. In fact, just as for the heat equation, one could first adapt \Cref{lem: unif time estimates} to a functional of the form \eqref{eq: J_T-aux.heat}; an adaptation which would hold for initial data in a ball of radius $r>0$ around $\overline{y}$, and then use the global estimate of \Cref{lem: quasiturnpike.1} and argue as in the beginning of the proof of \Cref{thm: turnpike} to fit within this ball, where one bootstraps forward in time only (namely, over intervals of the form $[n\tau,T]$).} as that of \Cref{thm: turnpike}, the only difference being the fact that we only need to bootstrap forward in time due to the lack of final cost, which renders the proof significantly less technical. 
	The control estimate follows from \Cref{lem: quasiturnpike.heat.1}. We thus concentrate solely on estimating the state. Let $T_0>0$ be arbitrary, and fix
	\begin{equation*}
	\tau>C_1^4
	\end{equation*}
	where $C_1 = C_1(f,\omega)>0$ is the (square root of the) constant appearing in \Cref{lem: quasiturnpike.heat.1}. 
	We note that if $T\leqslant2\tau+T_0$, then the desired estimate clearly follows by arguing as in previous proofs. We thus suppose that
	\begin{equation*}
	T>2\tau+T_0
	\end{equation*}
	is fixed. 
	First note that for $t \in [0,\tau+T_0]$, just as in Part 1 of the proof of \Cref{thm: turnpike}, the desired estimate can easily be obtained for such $t$ since the length of the time interval is independent of $T$. 
	Hence, we will solely concentrate on the case $t\in[\tau+T_0, T]$.
	To this end, we will mimic the steps done in the proof of \Cref{thm: turnpike}.
	\smallskip
	
	\noindent
	\textbf{Step 1). Preparation.} Since $2\tau< T$ and thus $\tau\leqslant \frac{T}{2}$, by \Cref{lem: D1} there exists a $\tau_1 \in [0, \tau)$ such that 
	\begin{equation} \label{eq: 5.7}
	\left\|y_T(\tau_1) - \overline{y}\right\|_{L^2(\Omega)}\leqslant\frac{\left\|y_T- \overline{y}\right\|_{L^2((0,T)\times\Omega)}}{\sqrt{\tau}} \leqslant\frac{C_1}{\sqrt{\tau}}\left\|y^0-\overline{y}\right\|_{L^2(\Omega)},
	\end{equation}
	where we used \Cref{lem: quasiturnpike.heat.1} for the second estimate.
	The control $u_T|_{[\tau_1, T]}$ can be shown to minimize $J_{\tau_1, T}$ with initial data $y^{\tau_1}=y_T(\tau_1)$ for \eqref{eq: heat.tau1}, to which clearly the solution is $y_T|_{[\tau_1,T]}$.
	 So by \Cref{lem: quasiturnpike.heat.1} and \eqref{eq: 5.7}, 
	\begin{equation} \label{eq: 5.8}
	\|y_T(t) - \overline{y}\|_{L^2(\Omega)}\leqslant C_1 \|y_T(\tau_1) - \overline{y}\|_{L^2(\Omega)}\leqslant\frac{C_1^2}{\sqrt{\tau}}\left\|y^0-\overline{y}\right\|_{L^2(\Omega)}
	\end{equation}
	holds for all $t \in [\tau_1, T]$. 
	Since $\tau_1<\tau$, \eqref{eq: 5.8} also holds for all $t \in [\tau, T]$.
	\smallskip

	\noindent
	\textbf{Step 2). Bootstrap.} We bootstrap \eqref{eq: 5.8} and prove that
	for any $n\in \N$ satisfying 
	\begin{equation*}
	n \leqslant \frac{T}{2\tau},
	\end{equation*}
	the estimate
	\begin{equation} \label{eq: 5.9}
	\sup_{t \in [n\tau, T]}\|y_T(t) - \overline{y}\|_{L^2(\Omega)}\leqslant\left(\frac{C_1^2}{\sqrt{\tau}}\right)^n \left\|y^0-\overline{y}\right\|_{L^2(\Omega)}
	\end{equation}
	holds. We proceed by induction. The case $n=1$ holds by \eqref{eq: 5.8}. Thus assume that \eqref{eq: 5.9} holds at some stage $n\in \N$ and suppose that
	\begin{equation*}
	n+1 \leqslant \frac{T}{2\tau}.
	\end{equation*}
	This clearly implies that
	\begin{equation} \label{eq: 5.10}
	\tau \leqslant \frac{T-2n\tau}{2}.
	\end{equation}
	The control $u_T|_{[n\tau, T]}$ can again be shown to be a global minimizer of $J_{n\tau, T}$.
	We can thus apply \Cref{lem: quasiturnpike.heat.1} with $\tau_1=n\tau$, and \Cref{lem: D1} (noting \eqref{eq: 5.10}) on $[n\tau, T-n\tau]$, to deduce that there exists $t_1 \in [n\tau, (n+1)\tau)$ such that
	\begin{equation*}
	\|y_T(t_1)-\overline{y}\|_{L^2(\Omega)}\leqslant\frac{\|y_T-\overline{y}\|_{L^2((n\tau,T)\times\Omega)}}{\sqrt{\tau}}\leqslant\frac{C_1}{\sqrt{\tau}} \|y_T(n\tau)-\overline{y}\|_{L^2(\Omega)}.
	\end{equation*}
	We may apply the induction hypothesis \eqref{eq: 5.9} to deduce
	\begin{equation} \label{eq: 5.12}
	\|y_T(t_1)-\overline{y}\|_{L^2(\Omega)}\leqslant\frac{C_1}{\sqrt{\tau}} \left(\frac{C_1^2}{\sqrt{\tau}}\right)^n \left\|y^0-\overline{y}\right\|_{L^2(\Omega)}.
	\end{equation}
	Since $u_T|_{[t_1,T]}$ is a global minimizer of $J_{t_1,T}$, we can apply \Cref{lem: quasiturnpike.heat.1} and use \eqref{eq: 5.12} to deduce that
	\begin{equation} \label{eq: 5.13}
	\|y_T(t)-\overline{y}\|_{L^2(\Omega)}\leqslant C_1 \|y_T(t_1)-\overline{y}\|_{L^2(\Omega)}\leqslant \frac{C_1^2}{\sqrt{\tau}}\left(\frac{C_1^2}{\sqrt{\tau}}\right)^n \left\|y^0-\overline{y}\right\|_{L^2(\Omega)}
	\end{equation}
	holds for all $t \in [t_1,T]$. Clearly, as $t_1< (n+1)\tau$, \eqref{eq: 5.13} also holds for all $t \in [(n+1)\tau, T]$.
	This concludes the induction proof, and so \eqref{eq: 5.9} does indeed hold.
	\smallskip
	
	\noindent
	\textbf{Step 3). Conclusion.}
	We now use \eqref{eq: 5.9} to conclude the proof. Suppose $t \in [\tau+T_0, T]$ is arbitrary and fixed. 
	Set $n(t):=\left\lfloor\frac{t}{\tau+T_0}\right\rfloor$. Clearly $n(t)\geqslant 1$, $t\geqslant n(t)\tau$ and $n(t) \leqslant \frac{T}{2\tau}$ due to the choice of $T_0$. We may then apply \eqref{eq: 5.9} to find that
	\begin{equation} \label{eq: 5.14}
	\|y_T(t) - \overline{y}\|_{L^2(\Omega)} \leqslant \left(\frac{C_1^2}{\sqrt{\tau}}\right)^{n(t)} \left\|y^0-\overline{y}\right\|_{L^2(\Omega)}
	\end{equation}
	Now since $\tau>C_1^4$ and $n(t) \geqslant \frac{t}{\tau+T_0}-1$, we can see from \eqref{eq: 5.14} that
	\begin{align*}
	\left\|y_T(t) - \overline{y}\right\|_{L^2(\Omega)}&\leqslant \exp\left(-n(t)\log\left(\frac{\sqrt{\tau}}{C_1^2}\right)\right)\left\|y^0-\overline{y}\right\|_{L^2(\Omega)} \\
	&\leqslant\frac{\sqrt{\tau}}{C_1^2} \exp\left(-\frac{\log\left(\frac{\sqrt{\tau}}{C_1^2}\right)}{\tau+T_0}t\right)\left\|y^0-\overline{y}\right\|_{L^2(\Omega)}
	\end{align*}
	The desired estimate thus holds for all $t \in [\tau+T_0,T]$, with 
	\begin{equation} \label{eq: heat.mu}
	\mu:= \frac{\log\left(\frac{\sqrt{\tau}}{C_1^2}\right)}{\tau+T_0}>0
	\end{equation}
	and 
	\begin{equation} \label{eq: heat.C}
	C:=\frac{\sqrt{\tau}}{C_1^2}>0.
	\end{equation}
	This concludes the proof.
	\end{proof}
	
	\section{Numerics} \label{sec: numerics}
	
	\noindent
	We briefly comment on the setting of the numerical experiment shown in \Cref{fig: figure.dl}. We make use of the neural ODE \eqref{eq: neural.net.2} with $\sigma\equiv\tanh$, and discretize with an explicit midpoint rule with $\bigtriangleup t = \sfrac12$. 
	The operator $P:\mathbb{R}^3\to[-1,1]$ appearing in \eqref{eq: J_T.deep} is defined as $Px=\texttt{hardtanh}(p_1 x+p_2)$, through the nonlinear thresholding operator $\texttt{hardtanh}(z) = 1_{\{s\geqslant 1\}}(z) + z1_{s\in(-1,1)}(z)+1_{\{s\leqslant1\}}(z)$, and the parameters $p_1\in\mathbb{R}^{3\times 3}$ and $p_2\in\mathbb{R}^3$ which are randomly sampled\footnote{As a byproduct of the Johnson–Lindenstrauss lemma (\cite[Lemma 23.4]{shalev2014understanding}), such random projections are of low distortion with respect to the Euclidean distance, in the sense that distances between points are nearly preserved after projecting.} from a normal distribution. 
	We use $n=2400$ points for training, and $600$ points for testing (see \Cref{fig: train.test}). 
	We originally consider a dataset of points in $\mathbb{R}^2$, but we embed them in $\mathbb{R}^3$ by adding a $0$ to each point. This is to avoid the intersection of trajectories in $\mathbb{R}^2$, which takes place due to  uniqueness (\cite{dupont2019augmented}). This is why we actually plot the predictor as a map $\mathbb{R}^2\to[-1,1]$ in \Cref{fig: train.test}.
	The code for reproducing all figures is available at \href{https://github.com/borjanG/dynamical.systems}{{\color{dukeblue}\texttt{https://github.com/borjanG/dynamical.systems}}}.
	
	\begin{figure}[h!]
	\includegraphics[scale=0.55]{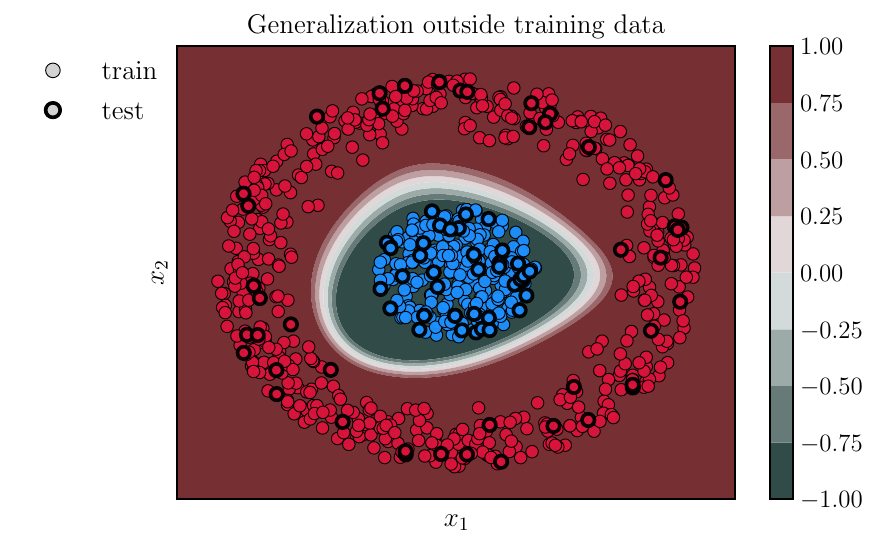}
	\caption{We see that the trained predictor, plotted on $[-2,2]^2$, has generalized the shape of the dataset, as desired.}
	\label{fig: train.test}
	\end{figure}
	
	\begin{figure}[h!]
	\includegraphics[scale=0.45]{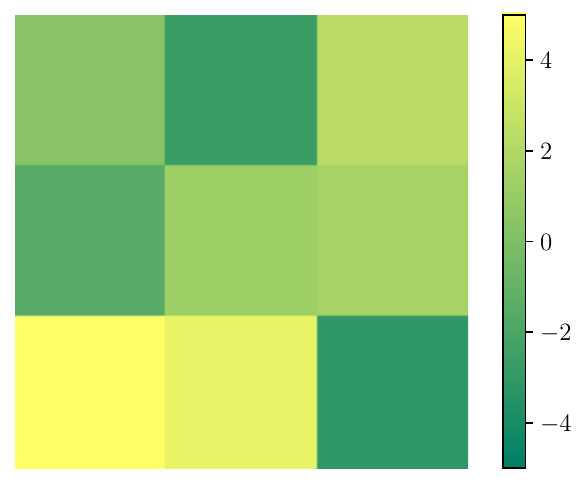}
	\includegraphics[scale=0.45]{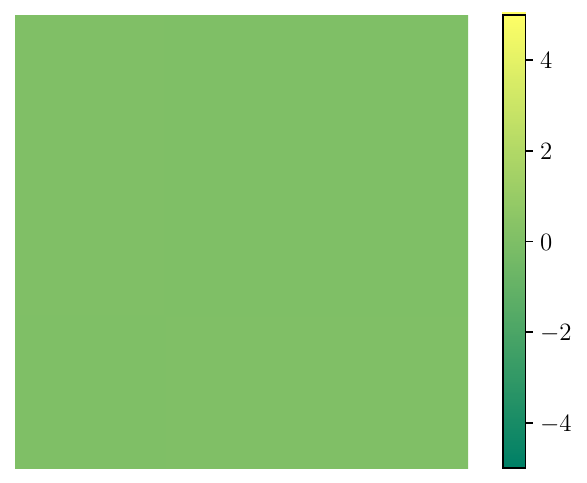}
	\includegraphics[scale=0.45]{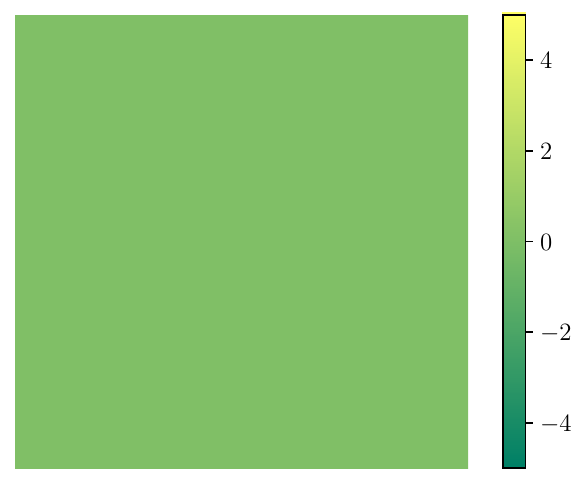}
	\caption{To further corroborate the graph showing the decay of optimal controls $\|u(t)\|$ in \Cref{fig: figure.dl}, where $u(t)=(w(t),b(t))$, we display the values of the matrix $w(t)\in\mathbb{R}^{3\times3}$ in times $t\in\{0.67, 1.33, 5\}$ (left to right).}
	\end{figure}

	\section{Concluding remarks and outlook}
	
	\label{sec: conclusion}
	
	\noindent
	We have presented a new methodology for proving the turnpike property for nonlinear optimal control problems set in large time horizons, under the assumption that the running target is a steady control-state pair, and that the system is controllable with a local estimate on the cost.
	These assumptions allow us to bypass necessary optimality conditions and a study of the adjoint system, and rather relies on calculus of variations--based arguments.		
	More precisely, we have concluded that
	\begin{itemize}
	\item The exponential turnpike property holds for optimal state trajectories of optimal control problems for nonlinear finite and infinite-dimensional dynamics, whenever the cost functional is coercive with respect to the distance of the state to the target steady state. 
	The nonlinearity may be assumed to be only globally Lipschitz continuous (and thus possibly nonsmooth).
	The result holds without any smallness assumptions on the initial data. 
	\smallskip
	\item The last exponential arc (near $t=T$) can be removed whenever the optimal control problem is considered without a final time cost, and thus entails an exponential stabilization estimate for the optimal state trajectory.
	\end{itemize}
	
	\noindent
	The motivation behind the consideration of steady state running targets in \eqref{eq: J_T} was the link with machine learning applications, namely problem \eqref{eq: J_T.deep} (although, we saw that our results also apply to many contexts which arise naturally in mechanics). While we see a turnpike phenomenon in the numerical simulations of \eqref{eq: J_T.deep}, our analysis done for \eqref{eq: J_T} strictly applies to \eqref{eq: J_T.deep} only when $d_x=d_y$ and $P$ is the identity. On another hand, since $P$ is surjective, we can see \eqref{eq: J_T} as a relaxed version of \eqref{eq: J_T.deep}, in which case, we select the running targets in the functional \eqref{eq: J_T} as $\overline{\*x}_i\in P^{-1}\left(\left\{y^{(i)}\right\}\right)$ for $i\in\{1,\ldots,n\}$, which are steady states of the underlying system with $0$ control (as are all constants, actually). 
	The numerical simulations in \Cref{fig: figure.dl} are all the more curious because of the fact that 1). we see the stabilization of the full state, while we solely penalize projections of it, and 2). the projection $x\mapsto Px$ is actually not coercive with respect to $x$ (see \Cref{sec: numerics}). 
	For problems manifesting a lack of observability of certain components of the state in the tracking term, the turnpike property for the observed components and the full controls has been shown in \cite{pighin2020turnpike}, in the setting of linear systems. Should such a property also hold for \eqref{eq: J_T.deep}, then the stability of the full state seen in \Cref{fig: figure.dl} can be explained through the decay of the optimal controls to $0$ and the specific form of the dynamics. 
	 One could envision a fusion of our strategy with Loyasewicz-type inequalities to provide theoretical guarantees, but this remains an open problem.

	\subsection{Outlook} \label{sec: outlook}
	
	Let us conclude with a select list of additional open problems.
	
	\begin{itemize}
	\item \textbf{Necessity of assuming that $\overline{y}$ is a steady state.}
	The assumption that the running target $\overline{y}$ in \eqref{eq: J_T} is a steady state of the dynamics allows us to easily obtain quasi-turnpike controls allowing us to obtain the key estimates in \Cref{lem: quasiturnpike.1} and \Cref{lem: unif time estimates} (resp. \Cref{lem: quasiturnpike.wave.1}, \Cref{lem: quasiturnpike.2.wave}, \Cref{lem: quasiturnpike.heat.1} in the PDE setting).
	The case of controlled steady states $\overline{y}$ associated to a presecribed control $\overline{u}$ can readily be addressed by penalizing $u-\overline{u}$ over $[0, T]$ instead of solely $u$ as noted in \Cref{rem: running.target}. 
	But we were unable to see if this is a necessary assumption in the nonlinear context in the absence of smallness conditions on the target, and whether the controlled steady state case can be covered by solely penalizing $u$. 
	These questions merit in-depth investigation.
	
	\smallskip
	\item\textbf{Weakening \Cref{def: ctrl}.} 
	An important hypothesis we made throughout is \Cref{def: ctrl}, which required that, at least for data $y^0, y^1$ in the vicinity of the free steady state $\overline{y}$, the minimal $L^2$--norm control steering the system from $y^0$ to $\overline{y}$ may be estimated by $\left\|y^0-\overline{y}\right\|$, and similarly for that from $\overline{y}$ to $y^1$.
	This is a hallmark of linear control systems, which is also expected for nonlinear systems for which controllability results are obtained by linearization or perturbation methods and a fixed-point argument.
	But in the general context of control-affine systems, such an assumption may appear restrictive, even-though it is local.
	It is thus of interest to see how the results and methodology can be pertained whilst weakening \Cref{def: ctrl}. 
	
	\smallskip
	\item \textbf{Turnpike with state or control constraints.} 
	A problem which has not been extensively covered in the literature is the turnpike property with positivity (or box) constraints on either the state or the control. Slightly weaker integral turnpike results under such constraints have been obtained in \cite{mazari2020quantitative} by means of \emph{quantitative inequalities}.
	Such a study would complement the already existent nonlinear \emph{controllability under constraints theory} -- a topic covered in several recent works, see e.g. \cite{le2020local, mazari2020constrained, pighin2018controllability, ruiz2020control} and the references therein.
	\smallskip
	
	\item\textbf{More general control systems.} 
	We have considered homogeneous Dirichlet boundary conditions in \eqref{eq: nonlinear.wave.control} and \eqref{eq: nonlinear.heat.control} merely to avoid additional technical details.
    The proofs of \Cref{thm: turnpike.wave} (resp. \Cref{thm: turnpike.heat}) only require that the underlying dynamics are exactly controllable (resp. controllable to a steady state), thus, the same results hold with Neumann boundary conditions. 
        Similarly, variable coefficients and lower order terms may be considered, as long as these coefficients are time-independent, as we are using a Duhamel formula along with a semigroup representation of the solution, and this semigroup ought to be uniformly bounded for all times.
    
    In fact, we have chosen the wave and heat equation for the sake of presentation, but the respective results could possibly be extended to a more general scenario of exactly controllable semilinear systems with similar assumptions, e.g. dispersive equations (Schr\"odinger, Korteweg-de Vries), coupled systems, and so on.
    
    The (apparent) necessity of a Duhamel formula may however be an impediment to the extension of our results to the context of quasilinear systems such as the porous medium equation (see \cite{geshkovski2019null} and the references therein).
   Similarly, boundary control systems may pose technical difficulties, since they require for the introduction of \emph{admissible control operators} (a general functional framework for lifting the trace on the boundary -- see \cite[Chapter 4]{tucsnak2009observation}) to be written in a canonical first order form, and consequently, to admit a Duhamel formula representation for the solution. The particular issue for boundary control systems is that there is no guarantee that the inferred control operator would be bounded with respect to $T$, which is of paramount importance to our strategy. We leave these extensions open to future studies.

    \smallskip
    
    \item\textbf{Bilinear control systems.}
    It would also be of interest to establish the turnpike property for bilinear control systems. 
    This would be the somewhat true analog of the control-affine systems presented herein, and under suitable assumptions on the nonlinearity, one could expect that our methodology applies to such cases as well. 
    We have not addressed such systems for the simplicity of presentation and due to the controllability assumptions we make, as the controllability theory for bilinear problems is not complete (albeit, see \cite{beauchard2020unexpected, cannarsa2017multiplicative, duprez2019bilinear, mazari2020fragmentation} for recent developments).
	Notwithstanding, our results should be applicable to a system of the form (see \cite{beauchard2020unexpected})
	\begin{equation*}
	\begin{dcases}
	\del_t y - \del_x^2 y = u(t)f(y) &\text{ in }(0,T) \times (0, \pi) \\
	\del_x y(t, 0) = \del_x y(t, \pi) = 0 &\text{ in }(0,T) \\
	y|_{t=0} = y^0 &\text{ in }(0,\pi)
	\end{dcases}
	\end{equation*}
	where $u$ is a scalar control and $f$ is an appropriate nonlinearity (see \cite{beauchard2020unexpected} for sufficient conditions for ensuring controllability, and globally Lipschitz for applying our methodology).
	
	\smallskip
	\item \textbf{More general nonlinearities.} 
	Finally, it would be of interest to investigate problems where our methodology does not immediately apply, such as the paradigmatic example of the cubic heat equation.
	This problem consists in seeing whether one may prove \Cref{thm: turnpike.heat} (with the estimate on $u_T$ changed by an estimate of $u_T-\overline{u}$) for minimizers $u_T$ of 
    \begin{equation*}
    J_T(u) := \int_0^T \|y(t)-\overline{y}\|^2 \diff t + \int_0^T \|u-\overline{u}\|^2 \diff t
    \end{equation*}    
    where $y_T$ is the unique solution to
    \begin{equation} \label{eq: cubic.heat}
    \begin{dcases}
    \del_t y - \Delta y + y^3 = u\one_\omega &\text{ in }(0,T)\times\Omega\\
    y = 0 &\text{ on } (0,T)\times\del\Omega\\
    y|_{t=0} = y^0 &\text{ in }\Omega,
    \end{dcases}
    \end{equation}
    and $\overline{y}\in H^1_0(\Omega)$ is a controlled steady state associated to some $\overline{u}\in L^2(\omega)$
    (the case $\overline{u}\equiv0$ is somewhat trivial due to the inherent stabilization to $\overline{y}\equiv 0$).
	Let us elaborate on a possible technical impediment in the direct application of our strategy. Clearly, for \Cref{thm: turnpike.heat} to hold in this case, it would suffice to prove \Cref{lem: quasiturnpike.heat.1} for $f(s) = s^3$ (while replacing the estimate of $u_T$ by an estimate of $u_T-\overline{u}$). 
	To this end, first of all, for any $u\in L^2((0,T)\times\omega)$, using the variational formulation and standard arguments including the Cauchy-Schwarz, Young and Poincaré inequalities, one can find
	\begin{equation*}
	\frac{\diff}{\diff t} \int_\Omega |y(t,x)|^2 \diff x\leqslant\epsilon\int_\omega \|u(t,x)\|^2 \diff x
	\end{equation*}
	for a.e. $t\in[0, T]$, where $\epsilon>\frac{C(\Omega)}{4}$, whereas $y$ solves \eqref{eq: cubic.heat}, and thus 
	\begin{equation} \label{eq: cubic.heat.estimate.1}
	\|y\|_{C^0([0,T]; L^2(\Omega))} \leqslant C_1(\Omega)\left(\|u\|_{L^2((0,T)\times\omega)}  + \left\|y^0\right\|_{L^2(\Omega)}\right).
	\end{equation}
	Following the proof of \Cref{lem: C2} for $f(s) = s^3$ and using \eqref{eq: cubic.heat.estimate.1}, we may find 
	\begin{equation*}
	\left\|y(t) - \overline{y}\right\|_{L^2(\Omega)} \leqslant C\left(\left\|y^0-\overline{y}\right\|_{L^2(\Omega)} + \left\|u-\overline{u}\right\|_{L^2((0,T)\times\omega)} + \left\|y-\overline{y}\right\|_{L^2((0,T)\times\Omega)}\right),
	\end{equation*}
	where now 
	\begin{equation*}
	C \sim \exp\left(\|u\|_{L^2((0,T)\times\omega)}\right).
	\end{equation*}
	It is precisely at this point where the issue appears, since simply by using the form of the functional, we are not in a position to prove that $\|u\|_{L^2((0,T)\times\omega)}$ is uniformly bounded with respect to $T$, but rather only $\|u-\overline{u}\|_{L^2((0,T)\times\omega)}$. Should this be possible, then one can expect our methodology to apply to the cubic heat equation as well, but as things stand, turnpike without smallness conditions in this case remains open.
	
	Further examples worth analyzing include the heat equation with a convective nonlinearity $f(y, \nabla y)$, even in one space dimension (e.g. the Burgers equation); along these lines we refer to \cite{zamorano2018turnpike} for a local turnpike result for the 2d Navier-Stokes system.
	Similar questions can be asked for the semilinear wave equation, where the nonlinearity is sometimes only assumed to be superlinear (see \cite{kunisch2020optimal} for a subcritical optimal control study) -- our methodology a priori applies if the nonlinearity is either truncated by some cut-off, or if one manages to prove uniform estimates of $\|y_T\|_{L^\infty((0,T)\times\Omega)}$ with respect to $T$. 
	Further nonlinear problems which could be investigated include hyperbolic systems (see \cite{gugat2019turnpike} for a related study) or free boundary problems (see \cite{geshkovski2019controllability} for a control perspective).
	\end{itemize}
	
	\subsubsection*{Acknowledgments} 
	
	B.G. thanks Idriss Mazari (U. Paris Dauphine) for helpful comments.
	The authors thank the anonymous reviewers for deeply insightful suggestions and remarks which have greatly improved the quality of this manuscript.
	\smallskip
	
	\noindent{\small{\textbf{Funding:}}
		\small{B.G. and E.Z. have received funding from the European Union's Horizon 2020 research and innovation programme under the Marie Sklodowska-Curie grant agreement No.765579-ConFlex.
		D.P., C.E. and E.Z. have received funding from the European Research Council (ERC) under the European Union’s Horizon 2020 research and innovation programme (grant agreement NO. 694126-DyCon).
The work of E. Z. has been supported by the Alexander von Humboldt-Professorship program, the Transregio 154 Project ‘‘Mathematical Modelling, Simulation and Optimization Using the Example of Gas Networks’’ of the German DFG, grant MTM2017-92996-C2-1-R COSNET of MINECO (Spain) and by the Air Force Office of Scientific Research (AFOSR) under Award NO. FA9550-18-1-0242.}}

	\bibliographystyle{acm}
	\bibliography{refs}{}

\end{document}